\documentclass[smallextended]{svjour3}

\smartqed  

\usepackage{amssymb}
\usepackage{amsfonts}
\usepackage{amsmath}
\usepackage[caption=false]{subfig}
\usepackage{multirow}
\usepackage{booktabs}
\usepackage{enumerate}
\usepackage{url}
\usepackage{algorithm}
\usepackage{algorithmic}   
\usepackage{hyperref}
\usepackage{graphicx}
\usepackage{cite}
\usepackage{newtxtext,newtxmath}
\usepackage{microtype}

\allowdisplaybreaks

\spnewtheorem{assumption}{Assumption}{\bfseries}{\rmfamily}

\def\bbP{\mathbb{P}}
\def\bbR{\mathbb{R}}
\def\bbS{\mathbb{S}}
\def\bbX{\mathbb{X}}

\DeclareMathOperator{\F}{F}
\DeclareMathOperator{\T}{T}

\DeclareMathOperator{\GRAD}{grad}
\DeclareMathOperator{\Prox}{Prox}
\DeclareMathOperator{\scf}{scf}
\DeclareMathOperator{\soft}{soft}
\DeclareMathOperator{\sym}{sym}
\DeclareMathOperator{\tol}{tol}
\DeclareMathOperator{\tr}{tr}

\newcommand\STM[2]{{\rm St}(#1,#2)}

\def\scrG{\mathscr{G}}
\def\scrL{\mathscr{L}}
\def\scrM{\mathscr{M}}

\def\wtd{\widetilde}
\def\what{\widehat}

\def\cM{\mathcal{M}}
\def\cR{\mathcal{R}}
\def\cT{\mathcal{T}}

\begin{document}

\setlength{\abovedisplayskip}{5pt plus 2pt minus 2pt}
\setlength{\belowdisplayskip}{5pt plus 2pt minus 2pt}
\setlength{\abovedisplayshortskip}{3pt plus 1pt minus 1pt}
\setlength{\belowdisplayshortskip}{3pt plus 1pt minus 1pt}


\title{Nonlinear Eigen-approach ADMM for Sparse Optimization on Stiefel Manifold\thanks{The work of R. Y. D. Xu was supported in part by Hong Kong Research Grants Council General Research Fund grant 12300322. The work of L. Wang was supported in part by National Science Foundation grants DMS-2009689 and DMS-2407692. The work of R.-C. Li was supported in part by National Science Foundation grant DMS-2407692.}}

\titlerunning{NEPvADMM for Sparse Optimization on the Stiefel Manifold}
\authorrunning{J. Wang, L. Wang, R. Y. D. Xu, and R.-C. Li}

\author{Jiawei Wang \and Li Wang \and Richard Yi Da Xu \and Ren-Cang Li}

\institute{J. Wang \and R. Y. D. Xu \at
              Department of Mathematics, Hong Kong Baptist University, Hong Kong \\
              \email{22483004@life.hkbu.edu.hk, xuyida@hkbu.edu.hk} 
           \and
           L. Wang \and R.-C. Li \at
              Department of Mathematics, University of Texas at Arlington, Arlington, TX 76019 USA \\
              \email{li.wang@uta.edu, rcli@uta.edu}
}

\date{Received: date / Accepted: date}

\maketitle

\begin{abstract}
With growing interest in machine learning and data science, creating
efficient algorithms to optimize various dimension reduction objectives
has become particularly important. This paper is concerned with
optimization problems on the Stiefel manifold whose objective functions
take the form as the sum of a smooth function and a nonsmooth one. A
common practice is to split such a problem into a smooth optimization
subproblem on the Stiefel manifold and a nonsmooth one and solve them
by some version of the alternating direction method of multipliers
(ADMM). Usually the smooth subproblem is solved by standard
optimization techniques skillfully adapting to the manifold. However,
often the smooth subproblem has strong linear algebra characteristics
which, if carefully exploited, can lead to much more efficient solvers
than the standard optimization techniques that ignore the
characteristics. In this paper, it is proposed to tackle the smooth
subproblem through a nonlinear eigen-approach which can fully utilize
rich linear algebra structures that dominantly present in objective
functions from today's data science applications. We conduct
convergence analysis and analyze computational complexity for obtaining
an $\epsilon$-stationary point. Numerical experiments demonstrate the
advantages of the proposed algorithm compared to the state-of-the-art
approaches.

\end{abstract}

\section{Introduction}\label{Se1}
In this paper, we consider a class of nonconvex and nonsmooth composite minimization problems
\begin{equation}\label{eq:TheProb}
\min_{X\in\STM{p}{n}} \,\Big\{ F(X):= f(X) + r(\Psi X)\Big\},
\end{equation}
where $\Psi \in \bbR^{m \times n}$,
the Stiefel manifold
\begin{equation}\label{eq:STM}
\STM{p}{n} = \left\{X\in\bbR^{n\times p}\mid X^{\T}X=I_p\right\},
\end{equation}
and $f\,:\,\bbR^{n\times p} \mapsto \bbR$ and $r\,:\,\bbR^{m\times p} \mapsto \bbR$ satisfy the following assumption:

\begin{assumption}
\;
\begin{enumerate}
\item[{\rm (i)}] $f$ is differentiable in some neighborhood of $\STM{p}{n}$ in $\bbR^{n\times p}$ and possibly nonconvex, and $\nabla f$ is Lipschitz continuous with Lipschitz constant $L_{f}$.
\item[{\rm (ii)}] $r$ is convex, Lipschitz continuous with Lipschitz constant $L_{r}$, and possibly nonsmooth, and the proximal mapping of $r$ can be efficiently evaluated.
\end{enumerate}
\end{assumption}

Often in application, $r$ is some $ \ell_1 $-related regularization term to induce sparsity, and the frequently encountered ones include some scalar multiples of the element-wise $ \ell_1 $-norm $|\!|\!| Y |\!|\!|_1$
and the matrix $(2,1)$-norm $\| Y \|_{2,1}$ of $Y\equiv[y_{ij}]\in\bbR^{m\times p}$:
\begin{align}
|\!|\!| Y |\!|\!|_1 &:=  \sum_{ij}|y_{ij}|, \label{eq:elewiseL1} \\
\| Y \|_{2,1}&:=\sum_{i=1}^n\|Y_{(i,:)}\|_2\equiv \sum_{i=1}^n\Big( \sum_{j=1}^p |y_{ij}|^2 \Big)^{1/2}, \label{eq:L21}
\end{align}
where $Y_{(i,:)}$ stands for the $i$th row of $Y$, namely
\begin{equation}\label{sparfun}
r(Y)=\mu|\!|\!| Y |\!|\!|_1
\quad\mbox{or}\quad
\mu\| Y \|_{2,1},
\end{equation}
for some $\mu>0$, a regularization parameter that needs to be chosen properly.
It is noted that convexity, Lipschitz continuity, and smoothness mentioned here are interpreted as in the ambient Euclidean space.

{\bf Motivation and our contributions}.
Most optimization problems on the Stiefel manifold from machine learning have strong numerical linear algebra
characteristics. Specifically, the objective functions are compositions of simple matrix trace functions, e.g.,
principal component analysis (PCA) \cite{joca:2016}, canonical correlations analysis (CCA) \cite{cunningham2015linear, zhang2020self}, the $\theta$-trace ratio \cite{wang2023trace}, among others.
They are smooth functions, which, when combined
with certain nonsmooth functions to fulfill particular learning objectives, leads to nonsmooth optimization problems on the Stiefel manifold \cite{jolliffe2003modified, lai2013compressed, yang2011norm,tang2012unsupervised}. The optimal solution $X\in\STM{p}{n}$ to \eqref{eq:TheProb} is often used for
(linear) dimensionality reduction, the cornerstone for studying  high-dimensional data, where the integer $p$ is the number of extracted features and often is much smaller than the original dimension $n$ of data, i.e., $p \ll n$.

Recently, methods based on nonlinear eigenvalue problem with eigenvector dependency (NEPv) have been showing great promises for smooth optimization on the Stiefel manifold \cite{cai2018eigenvector, li2024theory, wang2022maximizing, wang2023trace, zhang2020self}.
Li~\cite{li2024theory} laid out a general theory
and called it the NEPv approach, which makes it possible to take advantage of vast existing resources from
numerical linear algebra for fast and robust numerical solutions \cite{abbd:1999,bddrv:2000,demm:1997,govl:2013,li:2015,parl:1998}. However,
one pre-requisite for the approach to be even considered is that the objective
must be differentiable and thus is not applicable straightforwardly to our problem \eqref{eq:TheProb}.
This is where the splitting techniques employed by ADMM and its variants
come in to separate out the nonsmooth part from the smooth part in \eqref{eq:TheProb}, namely, turning it
into a smooth subproblem on the Stiefel manifold coupled with a nonsmooth one.
Our contribution is in the aspect of solving the smooth subproblem by the NEPv approach \cite{li2024theory}
much more efficiently than existing standard optimization techniques. We name our eventual approach as NEPvADMM
which includes two algorithms (Algorithms~\ref{alg:smooth.nepvadmm} and \ref{locg.smooth.nepvadmm-simplified}).
Besides being more efficient, the approach is stepsize-free, unlike that any method adopted from standard optimization
techniques has to perform some kind of line-searching for feasible stepsize.

{\bf Organization}.
The rest of this paper is organized as follows. 
Section~\ref{sec.lite} reviews existing works for nonsmooth manifold optimization.
We develop the methodology of our NEPvADMM and provide its related algorithms in Section~\ref{Se2}.
In Section~\ref{Se3}, we propose an acceleration technique to speedup our method.
In Section~\ref{Se4}, we analyze the computational complexity of our algorithms to obtain an $\epsilon$-stationary point and conduct  convergence analysis.
Section~\ref{Se5} discusses some representative real-life applications, and then in Section~\ref{Se6}, we report numerical results by our algorithms against existing methods on two different problems. Finally, we draw our conclusions in Section~\ref{Se7}.

{\bf Notation}.
$\bbR^{n \times m}$ is the set of $n \times m$ real matrices,  $\bbR^n = \bbR^{n \times 1}$, $\bbR = \bbR^1$.
Denote by $I_n \in \bbR^{n \times n}$ the identity matrix.
$\bbS^{n \times n}$ denotes the set of $n \times n$ real symmetric matrices,
and $\bbR^{n\times n}\ni A \succeq (\succ) \; 0$ means that $A\in\bbS^{n \times n}$
and positive semi-definite (positive definite).
The eigenvalues of a symmetric matrix $A\in\bbS^{n \times n}$ are denoted by
$\lambda_i(A)$ in the descending order for $1\le i\le n$. Beside notation $\STM{p}{n}$
in \eqref{eq:STM} for
the Stiefel manifold, we will also introduce for $D\in\bbR^{n\times p}$
\[
\STM{p}{n}_{D+} := \left\{X \in \STM{p}{n} : X^{\T}D \succeq 0\right\}.
\]
The orthogonal complement of a subspace $\mathcal{S}$ in $\bbR^n$ is $\mathcal{S}^{\perp}$.
For a matrix $X$,
$X^{\T}$ takes its transpose, and $\sym(X) = \left(X + X^{\T}\right)/2$ for a square matrix $X$.
The $i$-th row of a matrix $X$ is denoted by $X_{(i,:)}$,
while $X_{(i,j)}$ is the $(i,j)$-th entry of $X$.
The Euclidean inner product of $X,\,Y \in \bbR^{n \times p}$ is defined as
$\left\langle X, Y \right\rangle = \tr\left(X^{\T}Y\right)$, where $\tr(\cdot)$
takes the trace of a square matrix.
The singular values of $B\in\bbR^{m\times n}$ are denoted by $\sigma_i(B)$ in the descending order for $1\le i\le s:=\min\{m,n\}$,
and accordingly, its spectral, Frobenius, and trace norms are given by
$$
\| B \|_2 = \sigma_1(B), \quad
\| B \|_{\F} =\sqrt{\tr(B^{\T}B)} = \Big( \sum_{j=1}^s [ \sigma_j ( B ) ]^2\Big)^{1/2}, \quad
\| B \|_{\tr} = \sum_{j=1}^s \sigma_j(B).
$$
$\cR(B)$ is the column subspace of matrix $B$, and its dimension is the rank of $B$.
Finally, the indicator function $\delta_{\mathcal{S}}(X)$ for any set $\mathcal{S}$ is defined by
\[
\delta_{\mathcal{S}}(X) = \left\{ \begin{array}{lcl}
0 & \mbox{if} \;\;
X \in \mathcal{S}, \\ +\infty & \mbox{otherwise.}
\end{array}\right.
\]

\section{Literature Review}\label{sec.lite}
Optimization on matrix manifold has been an active research area in the past few decades. Absil, Mahony, and Sepulchre \cite{absil2008optimization}
contains a very comprehensive exposition on this issue up to 2008 in terms of theory and numerical computations. Most existing methods \cite{xiao2022class,gao2019parallelizable,gao2018new,li2024theory,wang2023trace} require the derivative of objective functions
and therefore cannot be applied to nonsmooth objective functions such as the one of problem \eqref{eq:TheProb}.
Due to non-convexity, non-smoothness as well as the manifold constraint, algorithms that can be applied here are limited.
According to different ways of dealing with orthogonality constraints, existing methods can be basically divided into two categories.

\subsection{Methods of the First Category}
Methods of the first type focus on the manifold optimization techniques. 
Early approaches relied on Riemannian subgradient methods \cite{ferreira1998subgradient,grohs2016subgradient,borckmans2014riemannian,hosseini2015convergence,hosseini2017riemannian,hosseini2018line} to handle the non-smoothness in the objective function.
However, these algorithms generally suffer from slow convergence rates and fail to exploit the rich structural properties inherent in the modern applications.

To overcome the slow convergence of subgradient methods, splitting techniques based on the alternating direction method of multipliers (ADMM) \cite{boyd2011distributed} have been adapted to the manifold setting. 
Both MADMM (manifold ADMM) \cite{kovnatsky2016madmm} and RADMM (Riemannian ADMM) \cite{li2024riemannian} introduce auxiliary variables to decouple the objective into a smooth optimization problem on the Stiefel manifold and a nonsmooth Euclidean subproblem solvable via proximal mapping.
The major difference between these two methods lies in their execution: RADMM smoothens the nonsmooth term $r$ by its Moreau envelope and approximates the smooth manifold subproblem using a single Riemannian gradient step, whereas MADMM works with $r$ directly and tackles its manifold subproblem iteratively using generic solvers like MANOPT \cite{boumal2014manopt}. 
While the iteration complexity of RADMM for obtaining an $\epsilon$-stationary point has been analyzed in \cite{li2024riemannian}, MADMM currently lacks theoretical convergence guarantees.

When $\Psi = I_n$ in \eqref{eq:TheProb}, Chen, Ma, So and Zhang \cite{chen2020proximal} proposed a manifold proximal gradient method (ManPG), and it can be formulated as follows:
\begin{subequations}\label{eq:ManPG4Xsub}
\begin{align}
    V_k & := 
    \arg\min_{V \in \cT_{X_k}\STM{p}{n}} \left\langle\nabla f(X_k), V\right\rangle + \frac{1}{2t}\left\| V \right\|_{\F}^2 + r\left(X_k + V\right), \label{eq:ManPG4Xsub-1}\\
    X_{k+1} & := {\rm Retr}_{X_k} \left(\alpha V_k\right),  \label{eq:ManPG4Xsub-2}
\end{align}
\end{subequations}
where $\cT_{X_k}\STM{p}{n}$ is the tangent space of $\STM{p}{n}$ at $X_k$, $\alpha, t > 0$ are stepsizes, $\nabla f$ is the Euclidean gradient of $f$, and ${\rm Retr}(\cdot)$ denotes the retraction mapping.
While the authors of \cite{chen2020proximal} suggested solving the subproblem \eqref{eq:ManPG4Xsub-1} via the semismooth Newton method, repeatedly executing this inner solver at every iteration incurs a substantial computational cost, which significantly limits the overall efficiency of the algorithm for large-scale problems.

\subsection{Methods of the Second Category}
Methods in the second category, based on splitting and alternating, introduce one or more variables to separate the objective function away from the Stiefel manifold constraint. A defining characteristic of these approaches is that they handle the orthogonality constraints strictly through Euclidean projections rather than Riemannian geometry. 
For example, with an auxiliary variable $Y$, the method of splitting orthogonality constraints (SOC) \cite{lai2014splitting,lai2013compressed} transforms \eqref{eq:TheProb} into:
\begin{equation}\label{eq:laos}
     \min_{X,Y}  \; \; f(X) + r(\Psi X) \quad
     \mathrm{ s.t.}  \; \; X = Y, \; Y\in\STM{p}{n}.
\end{equation}
By formulating the augmented Lagrangian $\scrL_{\beta}(X,Y,\Lambda)$, SOC solves \eqref{eq:laos} alternatingly as follows:
\begin{subequations}\label{eq:SOC-frame}
\begin{align}
    X_{k+1} & := \arg\min_{X} \scrL_{\beta}\left(X,Y_k,\Lambda_k\right), \label{eq:SOC-frame:X}\\
    Y_{k+1} & := \arg\min_{Y \in \STM{p}{n}} \scrL_{\beta}\left(X_{k+1},Y,\Lambda_k\right), \label{eq:SOC-frame:Y}\\
    \Lambda_{k+1} & := \Lambda_k + \beta \left(X_{k+1} - Y_{k+1}\right),  \label{eq:SOC-frame:Lamb}
\end{align}
\end{subequations}
where the $X$-subproblem \eqref{eq:SOC-frame:X} is a unconstrained
minimization problem that can be solved iteratively by methods such as proximal gradient, and the $Y$-subproblem \eqref{eq:SOC-frame:Y} corresponds to the projection onto $\STM{p}{n}$. It is noted that the theoretical convergence guarantee of SOC remains an open question.

When $\Psi = I_n$, the proximal alternating minimized augmented Lagrangian (PAMAL) method \cite{chen2016augmented} uses the indicator function $\delta_{\STM{p}{n}}$ and two auxiliary variables, $Y$ and $Z$, to reformulate \eqref{eq:TheProb} as:
\begin{equation}\label{pamaloriprob}
    \begin{split}
         \min_{X,Y,Z \in \bbR^{n \times p}}  f(X) + r(Y) + \delta_{\STM{p}{n}}(Z) \quad
     \mathrm{ s.t.}  \; \; Y = X, \; Z = X. 
    \end{split}
\end{equation}
To solve \eqref{pamaloriprob}, PAMAL employs the augmented Lagrangian method (ALM). Because the resulting primal subproblem remains structurally complex, PAMAL approximates its solution using proximal alternating minimization (PAM) \cite{attouch2013convergence}. Within this inner PAM iteration, the orthogonality constraint is entirely isolated in the $Z$-subproblem, which reduces to a direct Euclidean projection onto the Stiefel manifold computed via the singular value decomposition (SVD).
It is noted that, apart from its complicated algorithm structure, the result of PAMAL is sensitive to the choices of its parameters.

\section{The NEPvADMM Approach}\label{Se2}
In this section, we discuss how to combine the NEPv approach and ADMM to efficiently solve the nonsmooth and nonconvex optimization problem \eqref{eq:TheProb}. 
Following the splitting strategy introduced in \cite{li2024riemannian}, we introduce an auxiliary variable $Y \in \bbR^{m \times p}$ to decouple the nonsmooth term $r$ from the manifold constraint, yielding the equivalent formulation:
\begin{equation}\label{splitprob}
     \min_{X,Y} \;\; f(X) + r(Y) \quad
     \mathrm{ s.t.}\;\; \Psi X - Y=0, \; X\in\STM{p}{n}.
\end{equation}
Next, we smooth $r$ via its Moreau-Yosida regularization. For a chosen parameter $\gamma>0$, the Moreau envelope of $r$ is defined as:
\begin{equation}\label{eq:Moreau}
r_{\gamma}(Y) = \min_{Z} \left\{ r(Z) + \frac{1}{2\gamma}\left\| Z - Y \right\|_{\F}^2 \right\}.
\end{equation}
It is noted that the domain of $r_{\gamma}$ is $\bbR^{m \times p}$, and $r_{\gamma}$ is not only continuously differentiable but also convex due to the convexity of $r$. Replacing $r$ with $r_{\gamma}$ yields the following smooth optimization problem:
\begin{equation}\label{smoothprob}
     \min_{X,Y} \;\; f(X) + r_{\gamma}(Y) \quad
     \mathrm{ s.t.}\;\; \Psi X - Y=0,\; X\in\STM{p}{n}.
\end{equation}
The augmented Lagrangian function of problem \eqref{smoothprob} is
\begin{equation}\label{smoothprob.auglag}
    \scrL_{\beta,\gamma}(X,Y,\Lambda) := f(X) + r_{\gamma}(Y) + \left\langle \Lambda, \Psi X - Y \right\rangle + \frac{\beta}{2}\| \Psi X - Y \|_{\F}^2,
\end{equation}
where $\beta > 0$ is the penalty parameter. Applying ADMM to this framework yields the following iterative updates:
given an initial $(X_0,Y_0,\Lambda_0)$ with $Y_0 = \Psi X_0$,
\begin{subequations}\label{eq:MRADMM-frame}
\begin{align}
    X_{k+1} & := \arg\max_{X\in\STM{p}{n}} \;
                \big\{g_k(X):= -\scrL_{\beta,\gamma}(X,Y_k, \Lambda_k)\big\}, \label{eq:MRADMM-frame:X} \\
    Y_{k+1} & := \arg\min_{Y} \; \scrL_{\beta,\gamma}(X_{k+1},Y, \Lambda_k), \label{eq:MRADMM-frame:Y} \\
    \Lambda_{k+1} & := \Lambda_k + \beta (\Psi X_{k+1} - Y_{k+1}). \label{eq:MRADMM-frame:Lamb}
\end{align}
\end{subequations}

\subsection{The \texorpdfstring{$X$}{X}-subproblem}\label{se2.1.1}
The objective of the $X$-subproblem \eqref{eq:MRADMM-frame:X} is
\begin{equation}\label{eq:gk-unified}
g_k(X) := -f(X) - r_\gamma(Y_k) -\langle\Lambda_k,\Psi X-Y_k\rangle
                            - \frac{\beta}{2}\| \Psi X-Y_k\|_{\F}^2.
\end{equation}
We adopt a customized NEPv approach to tackle \eqref{eq:MRADMM-frame:X} \cite{cai2018eigenvector,li2024theory}.
The basic framework of the NEPv approach applied to the $X$-subproblem goes as follows:
\begin{enumerate}
    \item Construct an NEPv
    \begin{equation}\label{nepv}
          H_k(X)X = X\Omega_k, \;\; X \in \STM{p}{n}, \;\; H_k(X) \in \bbS^{n \times n}
      \end{equation}
    from the first-order optimality condition of the $X$-subproblem, where 
      $H_k(X)$ is a matrix-valued function dependent on $X$, and $\Omega_k = X^{\T}H_k(X)X \in \bbS^{p \times p}$.
    
    \item Solve NEPv \eqref{nepv} using the SCF iteration that is provably convergent: given $X_{k,0}$, iteratively compute partial eigen-decomposition 
    \begin{equation}\label{scf.basic.step}
        H_k(X_{k,i}) \what X_{k, i} = \what X_{k, i} \Omega_{k,i}
    \end{equation}
    associated with the $p$ largest eigenvalues of $H_k(X_{k,i})$ for $\what X_{k, i} \in \STM{p}{n}$, and post-process $\what X_{k, i}$ to $X_{k,i+1}$.
\end{enumerate}
To proceed, we start by establishing the first-order optimality condition, also known as the KKT condition, of \eqref{eq:MRADMM-frame:X}. Note that
the KKT condition is obtained by setting the Riemannian gradient of its objective function $g_k(X)$
with respect to $\STM{p}{n}$ at $X$ to $0$. It can be calculated that the Riemannian gradient is
(see, e.g., \cite{absil2008optimization,li2024theory})
\begin{equation}\label{grad.xsub}
\GRAD \; g_{k \vert \STM{p}{n}}(X) = \mit{\Pi}_{X}\left( \frac{\partial g_k(X)}{\partial X} \right)
     = \frac{\partial g_k(X)}{\partial X} - X\, \sym\left(X^{\T}\frac{\partial g_k(X)}{\partial X}\right),
\end{equation}
where the Euclidean gradient
\[
  \frac{\partial g_k(X)}{\partial X} \in \bbR^{n \times p} \;\; {\rm with} \;\; \left[ \frac{\partial g_k(X)}{\partial X} \right]_{(i,j)} := \frac{\partial g_k(X)}{\partial X_{(i,j)}},
\]
$\mit{\Pi}_{X}(P) := P - X \sym\left(X^{\T}P\right)$ projects $P$ onto the tangent space of $\STM{p}{n}$ at $X$.
Hence the KKT condition, $\GRAD \; g_{k \vert \STM{p}{n}}(X) = 0$, is explicitly
given by
\begin{subequations}\label{kkt.xsub}
\begin{align}
\frac{\partial g_k(X)}{\partial X} & \equiv -\frac{\partial f(X)}{\partial X} - \Psi^{\T}\Lambda_k - \beta \Psi^{\T}\left(\Psi X - Y_k\right) = X\widehat{\Omega}_k, \label{kkt.xsub.a} \\
    X & \in \STM{p}{n}, \;\; \widehat{\Omega}_k \in \bbS^{p\times p}. \label{kkt.xsub.b}
\end{align}
\end{subequations}
Immediately, $\widehat{\Omega}_k = X^{\T} \left[ \partial g_k(X)/\partial X \right]$,
implying that $X^{\T} \left[ \partial g_k(X)/\partial X \right]$ is symmetric at any KKT point $X$.

Next we explain how to turn the KKT condition \eqref{kkt.xsub} equivalently into NEPv form~\eqref{nepv} which needs
$H_k(X) \in \bbS^{n \times n}$. Before we set out to construct such a symmetric matrix-valued function,
we cite the following theorem which establish a condition on
$H_k(X) \in \bbS^{n \times n}$ to ensure the equivalency between the KKT condition \eqref{kkt.xsub} and NEPv \eqref{nepv}.

\begin{theorem}[{\cite[Theorem 6.1]{li2024theory}}]\label{thm:KKT=NEPv}
    Let $H_k(X) \in \bbS^{n \times n}$ be a matrix-valued function on $\STM{p}{n}$ and suppose
    \begin{equation}\label{equa.nepv.parde}
        H_k(X)X - \frac{\partial g_k(X)}{\partial X} = X\scrM_k(X) \;\; for \;\; X \in \STM{p}{n},
    \end{equation}
     where $\scrM_k(X) \in \bbR^{p \times p}$ is some matrix-valued function. Then $X \in \STM{p}{n}$ is a solution to the first-order optimality condition \eqref{kkt.xsub} if and only if it is a solution to NEPv \eqref{nepv} and $\scrM_k(X)$ is symmetric.
\end{theorem}

Now we are ready to construct $H_k(X)$. There are two cases to consider.
The first case is when there exists some $H^{f}(X) \in \bbS^{n \times n}$ such that
\begin{equation}\label{cons.H.firstcase.f}
    \frac{\partial f(X)}{\partial X} \equiv H^{f}(X)X \;\; {\rm for} \; X \in \STM{p}{n}.
\end{equation}
In such a case, we can construct $H_k(X)$ as
\begin{equation}\label{cons.H.firstcase.H}
    H_k(X) := -H^{f}(X) + D_kX^{\T} + XD_k^{\T} - \beta \Psi^{\T} \Psi,
\end{equation}
for use in \eqref{nepv}, where
\begin{equation}\label{eq:Dk}
D_k = \Psi^{\T}(\beta Y_k-\Lambda_k).
\end{equation}
It can be verified that, with \eqref{cons.H.firstcase.H},
\[
H_k(X)X - \frac{\partial g_k(X)}{\partial X} = X \big( D_k^{\T}X \big).
\]
Hence, by Theorem~\ref{thm:KKT=NEPv}, we have following corollary.

\begin{corollary}\label{cor:KKT=NEPv_case1}
Suppose that \eqref{cons.H.firstcase.f} holds. Then any $X \in \STM{p}{n}$ is a solution to the first-order optimality condition \eqref{kkt.xsub}
of \eqref{eq:MRADMM-frame:X} if and only if it is a solution to NEPv \eqref{nepv} with \eqref{cons.H.firstcase.H} and $D_k^{\T}X$ is symmetric.
\end{corollary}

On other hand, when \eqref{cons.H.firstcase.f} does not hold, which is the second case,
inspired by \cite{li2024theory}, we choose
\begin{equation}\label{cons.H.secondcase.H}
H_k(X)  = -\left[ \frac{\partial f(X)}{\partial X} \right] X^{\T} - X\left[\frac{\partial f(X)}{\partial X}\right]^{\T}
            + D_kX^{\T} + XD_k^{\T} - \beta \Psi^{\T} \Psi,
\end{equation}
where $D_k$ is as in \eqref{eq:Dk}.

\begin{corollary}\label{cor:KKT=NEPv_case2}
Any $X \in \STM{p}{n}$ is a solution to the first-order optimality condition \eqref{kkt.xsub}
of \eqref{eq:MRADMM-frame:X} if and only if it is a solution to NEPv \eqref{nepv} with \eqref{cons.H.secondcase.H} and
$-[\partial f(X)/\partial X]^{\T}X + D_k^{\T} X $ is symmetric.
\end{corollary}

We have turned the KKT condition \eqref{kkt.xsub} equivalently into NEPv~\eqref{nepv}, which can be solved
by the SCF iteration. Proposing such an SCF is a no-brainer, but making sure that it computes a maximizer of
\eqref{eq:MRADMM-frame:X} is not. To that end, we will invoke the NEPv theory in \cite{li2024theory}
to ensure that $\{g_k(X_{k,i})\}_{i=0}^{\infty}$ produced by SCF is monotonically
increasing and any accumulation point of $\{X_{k,i}\}_{i=0}^{\infty}$ satisfies the KKT condition \eqref{kkt.xsub}, i.e.,
the accumulation point is guaranteed a stationary point.
The NEPv theory in \cite{li2024theory} is built upon
an assumption called {\bf the NEPv Ansatz}. 
It is desirable to verify the NEPv Ansatz in advance before applying the SCF iteration, and here we state it specifically for \eqref{eq:MRADMM-frame:X} with objective function
\eqref{eq:gk-unified} as follows.

\smallskip\noindent
\textbf{The NEPv Ansatz} for \eqref{eq:MRADMM-frame:X}. \label{NEPvAnsatz}
{\em
There exists $H_k(X) \in \bbS^{n \times n}$ and a subset $\bbX \subseteq \STM{p}{n}$ such that for $\what X \in \STM{p}{n}$ and $X \in \bbX$, if
\begin{equation}\label{ansatz.traceineq}
    \tr\big(\what X^{\T}H_k(X)\what X\big) \geq \tr\big(X^{\T}H_k(X)X\big) + \eta \; \; \mbox{for some $\eta \geq 0$},
\end{equation}
then there exists $Q\in\STM{p}{p}$ such that $\wtd X = \what XQ \in \bbX$ satisfies
\begin{equation}\label{ansatz.objineq}
    g_k(\wtd X) \geq g_k(X) + \omega_k\eta,
\end{equation}
for some $\omega_k > 0$, independent of $X$ and $\what X$.
}

The Ansatz involves $\bbX$, a subset of $\STM{p}{n}$, which is problem-dependent.
The purpose of $Q \in \STM{p}{p}$ is to align $\what X$ to $\wtd X$
subject to $\cR(\wtd X) = \cR(\what X)$ so that \eqref{ansatz.objineq} holds.
The reader is referred to \cite{li2024theory} for more
discussions on that matter.
We will provide the detailed procedures of verifying the NEPv Ansatz for some common concrete functions $f$ that defines $g_k$,
as well as choosing the corresponding $Q$ in Section~\ref{Se6}.
It should be pointed out that given $\partial g_k(X)/\partial X$, there may be more than
one ways to construct $H_k(X)$ that satisfies \eqref{equa.nepv.parde}, but the NEPv Ansatz may not hold for all these $H_k(X)$.

\subsection{The \texorpdfstring{$Y$}{Y}-subproblem}
We restate the $Y$-subproblem \eqref{eq:MRADMM-frame:Y} as
\[
\left(Y_{k+1}, Z_{k+1}\right) := \arg\min_{Y, Z} \;r(Z) + \frac{1}{2\gamma}\| Z - Y \|_{\F}^2
         + \left\langle \Lambda_k, \Psi X_{k+1} - Y \right\rangle + \frac{\beta}{2}\| \Psi X_{k+1} - Y \|_{\F}^2.
\]
Its KKT conditions are
\begin{subequations}\label{smoothprob.z.kkt}
    \begin{align}
        0 & = \frac{1}{\gamma}\left(Y_{k+1} - Z_{k+1}\right) - \Lambda_k + \beta\left(Y_{k+1} - \Psi X_{k+1}\right), \label{smoothprob.z.kkt.a}\\
        0 & \in \partial r(Z_{k+1}) + \frac{1}{\gamma}\left(Z_{k+1} - Y_{k+1}\right), \label{smoothprob.z.kkt.b}
    \end{align}
\end{subequations}
where $\partial r(\cdot)$ is the subgradient of $r$.
Equation \eqref{smoothprob.z.kkt.a} immediately yields the following $Y$-updating formula
\begin{subequations}\label{eq:MRADMM-hat(r)}
\begin{equation}\label{smoothprob.z.solu}
    Y_{k+1} = \frac{\gamma}{1+\beta\gamma} \Big(\frac{1}{\gamma}Z_{k+1} + \Lambda_k + \beta \Psi X_{k+1}\Big).
\end{equation}
Plugging \eqref{smoothprob.z.solu} into \eqref{smoothprob.z.kkt.b}, we have
\begin{align}
    Z_{k+1} &= \arg\min_{Z} \frac{\beta\gamma + 1}{\beta}\, r(Z)
          + \frac{1}{2} \Big\| Z - \Big(\Psi X_{k+1} + \frac{1}{\beta}\Lambda_k\Big) \Big\|_{\F}^2 \nonumber \\
    &\equiv \Prox_{[(1+\beta\gamma)/\beta]r} \Big(\Psi X_{k+1} + \frac{1}{\beta}\Lambda_k\Big), \label{smoothprob.y.solu}
\end{align}
\end{subequations}
where $\Prox_{[(1+\beta\gamma)/\beta]r}\left(\Psi X_{k+1} + \Lambda_k/\beta\right)$ is the proximal mapping of the scaled function $[(1+\beta\gamma)/\beta]r$.
The following lemma shows that \eqref{smoothprob.y.solu} can be readily evaluated for those $r(\cdot)$ in \eqref{sparfun}. 

\begin{lemma}\label{le1}
      For the scaled element-wise $\ell_1$-norm $r(X) = \mu |\!|\!| X |\!|\!|_1$ defined as in \eqref{eq:elewiseL1}, we have
      \begin{equation}\label{splitprob.y.l1}
          Z_{k+1} = \soft\Big(\Psi X_{k+1} + \frac{1}{\beta}\Lambda_k, \; \frac{(1 + \beta\gamma)\mu}{\beta} \Big),
      \end{equation}
      where $\soft(W,\alpha)$ is understood as applying the soft-thresholding function $\soft(\cdot,\alpha)$
      $$
      \soft(x,\alpha) = {\rm sign}(x) \; \max\left\{ \left\vert x \right\vert - \alpha, 0\right\}
       \equiv
       \begin{cases}
       x - \alpha, &  \mbox{if $x > \alpha$}, \\
       0,          & \mbox{if $x \in \left[ -\alpha, \alpha \right]$}, \\
       x + \alpha, & \mbox{if $x < -\alpha$},
       \end{cases}
      $$
      entrywise to a matrix $W$.
      For the scaled $\ell_{2,1}$-norm
      $r(X)=\mu \| X \|_{2,1}$ defined as in \eqref{eq:L21}, we have
      \begin{equation}\label{splitprob.y.l21}
          \left[Z_{k+1}\right]_{(i,:)} = \Big[\Psi X_{k+1} + \frac{1}{\beta}\Lambda_k\Big]_{(i,:)}
                 \left(1 - \frac{\mu (1 + \beta\gamma)}
                                 {\beta \max \Big\{\Big\| \Big[\Psi X_{k+1} + \frac{1}{\beta}\Lambda_k\Big]_{(i,:)} \Big\|_2, \frac{\mu (1 + \beta\gamma)}{\beta} \Big\}}
                                 \right).
      \end{equation}
\end{lemma}

\subsection{The NEPvADMM Algorithm}
Having explained how to solve the $X$- and $Y$-subproblems, we summarize our plain NEPvADMM algorithm as Algorithm~\ref{alg:smooth.nepvadmm}. 
To maintain computational efficiency, here the $X$-subproblem \eqref{eq:MRADMM-frame:X} need not be solved to full convergence via SCF iterations at every outer step.
Instead, it suffices to perform a small, pre-selected number of SCF iterations, denoted by $n_{\scf} \geq 1$, to ensure that the updated iterate $X_{k+1}$ has enough accuracy that is sufficiently better than $X_k$.

\begin{algorithm}[tbhp]
\caption{NEPvADMM}\label{alg:smooth.nepvadmm}
\begin{algorithmic}[1]
\REQUIRE $X_{0} \in \bbX \subseteq \STM{p}{n} $, $Y_{0}$, $\Lambda_{0}$, $\beta > 0$,
         $\gamma > 0$, and $n_{\scf} \ge 1$;
\ENSURE an approximate solution of \eqref{splitprob}.
\STATE $k=0$;

\WHILE{certain stopping criterion is not satisfied}

\STATE $X_{k,0}=X_k$;

\FOR{$i=0,\ldots,n_{\scf}-1$}

\STATE construct $H_{k,i} = H_k(X_{k,i})$;
\STATE compute the partial eigen-decomposition $H_{k,i}\what X_{k,i} = \what X_{k,i}\Omega_{k,i}$
       for $\what X_{k,i} \in \STM{p}{n}$,
       an orthonormal basis matrix of the eigenspace associated with the $p$ largest eigenvalues of $H_{k,i}$;

\STATE compute $Q_{k,i} \in \STM{p}{p}$ and update $X_{k,i+1} = \what X_{k,i}Q_{k,i} \in \bbX$, according to the NEPv Ansatz
       on page~\pageref{NEPvAnsatz};
\ENDFOR
\STATE $X_{k+1}=X_{k,n_{\scf}}$;

\STATE $Z_{k+1} = \Prox_{[(1+\beta\gamma)/\beta]r} \Big(\Psi X_{k+1} + \frac{1}{\beta}\Lambda_k\Big) $;

\STATE $Y_{k+1} = \frac{\gamma}{1+\beta\gamma} \Big(\frac{1}{\gamma}Z_{k+1} + \Lambda_k + \beta \Psi X_{k+1}\Big) $;        

\STATE $\Lambda_{k+1} = \Lambda_k + \beta (\Psi X_{k+1} - Y_{k+1})$;

\STATE $k = k+1$;

\ENDWHILE

\RETURN  $X_k$ (and corresponding $Y_k$ and/or $Z_k$ if also desired).
\end{algorithmic}
\end{algorithm}

\section{LOCG-accelerated NEPvADMM}\label{Se3}
Algorithm~\ref{alg:smooth.nepvadmm} requires solving an $n \times n$ symmetric eigenvalue problem (SEP) in each SCF iterative step, which can be expensive when $n$ is large.
In this section, we introduce a locally optimal conjugate gradient (LOCG) acceleration technique to improve the speed of the NEPv approach. This acceleration idea draws inspiration from optimization \cite{polyak1987introduction, iwaro1965note} and has been adopted for different numerical linear algebra problems
\cite{imakura2016locally, knyazev2001toward, li:2015, wang2022maximizing, yang2022heavy}, and Li \cite{li2024theory} provides a generalization of its use in both
NPDo and NEPv approaches there. This section is essentially \cite[Section~3.3]{li2024theory}.

Without loss of generality, we denote $X_{-1} \in \STM{p}{n}$ as the approximate maximizer of the $X$-subproblem from the very previous iterative step and $X \in \STM{p}{n}$ as the current approximate maximizer.
Along the line of LOCG, the next approximate maximizer $X_1$ is computed according to
\begin{equation}\label{locg.ori.prob}
    \begin{split}
        X_1 \; = \arg\max_{S \in \STM{p}{n}} \; g_k(S) \;\;\;\;\;
        {\rm s.t.} \;\; \cR(S) \subseteq \cR\left(\left[X,\scrG(X),X_{-1}\right]\right),
    \end{split}
\end{equation}
where
\[\scrG(X) := \GRAD\; g_{k\vert\STM{p}{n}}(X) = \frac{\partial g_k(X)}{\partial X} - X \sym\left(X^{\T}\frac{\partial g_k(X)}{\partial X}\right).\]
For the first iteration where $X_{-1}$ is absent, we can simply use $\cR(S) \subseteq \cR\left(\left[X,\scrG(X)\right]\right)$.

In order to solve \eqref{locg.ori.prob} numerically, we start by finding an orthonormal basis matrix $W \in \STM{\hat n}{n}$ of the subspace $\cR\left(\left[X,\scrG(X),X_{-1}\right]\right)$. Generically, $\hat n = 3p$ but $\hat n < 3p$ can happen.
Since $X \in \STM{p}{n}$, $W$ can be computed by the Gram-Schmidt orthogonalization process, starting with orthogonalizing the columns of $\scrG(X)$ against $X \in \STM{p}{n}$.
In MATLAB, we implement it by setting $W = \left[\scrG(X), X_{-1}\right]$ (or $W = \scrG(X)$ for the first iteration) and then performing the classical Gram-Schmidt orthogonalization twice to almost ensure that the columns of the  computed $W$ at the end of the second line below orthogonal to the columns of $X$:
\begin{equation}\label{eq:LOCG-W}
\framebox{
\begin{minipage}{4.9cm}
\tt  W=W-X*(X'*W); W=orth(W); \\
\tt W=W-X*(X'*W); W=orth(W);\\
\tt W=[X,W];
\end{minipage}
}
\end{equation}
where $\texttt{orth}$ is a MATLAB function for orthogonalization\footnote{Another option is to use MATLAB's thin {\tt qr}:
    {\tt [W,$\sim$]=qr(W,0)}.
    }, and we note that the first $p$ columns of the final $W$ are the same as those of $X$.
The constraint $\cR(S) \subseteq \cR\left(\left[X,\scrG(X),X_{-1}\right]\right)$ in \eqref{locg.ori.prob}
can now be satisfied by expressing $S$ as
\[
    S = WP \;\; {\rm for} \;\; P \in \STM{p}{\hat n}.
\]
The problem \eqref{locg.ori.prob} can be rewritten as
\begin{equation}\label{locg.re.prob}
    P_{\ast} = \arg\max_{P \in \STM{p}{\hat n}} \; \big\{\widetilde{g}_k(P) := g_k(WP)\big\},
\end{equation}
and $X_1 = WP_\ast$. The KKT condition of \eqref{locg.re.prob} is
\begin{equation}\label{locg.kkt}
    \frac{\partial \widetilde{g}_k(P)}{\partial P} = P \widetilde{\Lambda}_k \quad \text{with} \quad \widetilde{\Lambda}_k \in \bbS^{p \times p}, \  P \in \STM{p}{\hat n}.
\end{equation}
It turns out that $\frac{\partial \widetilde{g}_k(P)}{\partial P}$ and
$\frac{\partial g_k(X)}{\partial X}$ are closely related:
\begin{equation}\label{eq:tg2g:partd}
\frac{\partial \widetilde{g}_k(P)}{\partial P} = W^{\T} \frac{\partial g_k(X)}{\partial X}\bigg|_{X=WP}
        = W^{\T}\mathcal{H}_k(WP).
\end{equation}
Guided by Theorem~\ref{thm:KKT=NEPv}, we suppose we already have $H_k(X)\in \bbS^{n \times n}$
for $g_k(X)$ such that
\begin{equation}\label{locg.nepv.parde.equa}
     H_k(X)X - \frac{\partial g_k(X)}{\partial X} = X\scrM_k(X).
\end{equation}
Then we have for $X = WP$ where $P \in \STM{p}{\hat n}$,
\begin{align}
    &W^{\T}\left[ H_k(X)X - \frac{\partial g_k(X)}{\partial X} \right]
     = W^{\T}X\scrM_k(X), \nonumber \\ 
\Rightarrow\quad &[ W^{\T}H_k(WP)W ]P - \frac{\partial \widetilde{g}_k(P)}{\partial P}
        = P\widetilde{\cM}_k(P) , \label{locg.nepv.parde.equa.trans.b}
\end{align}
because of \eqref{eq:tg2g:partd},  where $\widetilde{\cM}_k(P) := \scrM_k(WP)$.
Therefore, the NEPv for the reduced problem~\eqref{locg.re.prob} becomes
\begin{equation}\label{locg.h}
    \wtd H_k(P)P := [ W^{\T}H_k(WP)W ] P = P\widetilde{\Omega}_k, \;\; P \in \STM{p}{\hat n}, \; \widetilde{\Omega}_k \in \bbS^{p \times p}.
\end{equation}
Similar to Theorem \ref{thm:KKT=NEPv}, the following result can be derived accordingly.

\begin{theorem}\label{th4}
Let $H_k(X) \in \bbS^{n \times n}$ for $g_k(X)$ be a matrix-valued function on $\STM{p}{n}$ that satisfies \eqref{locg.nepv.parde.equa}.
Then $P \in \STM{p}{\hat n}$ satisfies the KKT condition \eqref{locg.kkt} of the reduced problem~\eqref{locg.re.prob} if and only if $\widetilde{\cM}_k(P)$ from \eqref{locg.nepv.parde.equa.trans.b} is symmetric and $P$ is a solution to the NEPv \eqref{locg.h}.
\end{theorem}

Importantly, $\widetilde{g}_k$ with $\wtd H_k$ inherits the NEPv Ansatz for $g_k$ with $H_k$ of the original
$X$-subproblem.

\begin{lemma}[{\cite[Lemma 6.1]{li2024theory}}]
Suppose that the NEPv Ansatz holds for $g_k(X)$ with $H_k(X)$, and let $\bbP := W^{\T}\bbX \subseteq \STM{p}{\hat n}$. If $W\bbP \subseteq \bbX$, then the NEPv Ansatz holds for $\widetilde{g}_k(P)$ with
$\wtd H_k(P) = W^{\T}H_k(WP)W\in\bbS^{\hat n\times \hat n}$.
\end{lemma}

Therefore, SCF iteration can be called upon to solve the reduced problem
\eqref{locg.re.prob}. This complete LOCG-accelerated SCF procedure for the $X$-subproblem \eqref{eq:MRADMM-frame:X} is summarized in Algorithm~\ref{scf.iteration-LOCG}.

\begin{algorithm}[htbp]
\caption{LOCG-Accelerated SCF iteration for the $X$-subproblem \eqref{eq:MRADMM-frame:X}}
\label{scf.iteration-LOCG}
\begin{algorithmic}[1]
\REQUIRE $H_k(X) \in \bbS^{n \times n}$ satisfying the NEPv Ansatz, $X_{k-1}$ if $k\ge 1$, $X_k \in \bbX$;
\ENSURE an approximate solution $X_{k+1}$ of \eqref{eq:MRADMM-frame:X}.
\STATE if $k=0$ then $X_{k,-1} = [\;]$ (null matrix), or else $X_{k,-1} = X_{k-1}$;
\STATE $i=0$, $X_{k,0}=X_k$;
\WHILE{certain stopping criterion is not reached}
    \STATE form problem~\eqref{locg.re.prob} with $X=X_{k,i}$ and $X_{-1}=X_{k,i-1}$ in \eqref{locg.ori.prob} (for the ease of discussion later,
           we name the corresponding $W$-matrix, generated according to \eqref{eq:LOCG-W}, as $W_i$);
    \STATE call SCF iteration to solve problem~\eqref{locg.re.prob} to compute its solution
       $P_{i,*}$, with input
       $\wtd H_{k,i}(\cdot)=W_i^{\T}H_k(W_i\,\cdot\,)W_i \in \bbS^{\hat n \times \hat n}$ (satisfying the NEPv Ansatz by inheritance), and
       initial $P_{i,0}$ as the first $p$ columns of $I_{\hat n}$;
    \STATE $X_{k,i+1}=W_i P_{i,*}$;
    \STATE $i=i+1$;
\ENDWHILE
\RETURN $X_{k+1}$ as the last $X_{k,i}$ as a maximizer of $g_k(X)$.
\end{algorithmic}
\end{algorithm}

A key practical consideration when implementing Algorithm~\ref{scf.iteration-LOCG} is establishing the stopping criteria for both its main loop and the nested SCF iteration at Line 5.
It is rather complicated. Evidently, we only need to solve the $X$-subproblem~\eqref{eq:MRADMM-frame:X} accurately enough so that each  alternating iteration sweep \eqref{eq:MRADMM-frame} that goes from $(X_k,Y_k,\Lambda_k)$ to $(X_{k+1},Y_{k+1},\Lambda_{k+1})$ makes enough progress as if each subproblem
there is solved exactly.
Therefore, we execute only a single outer LOCG update per ADMM iteration, and approximate the solution to the reduced $X$-subproblem using a fixed, pre-selected number of inner SCF iterations, denoted as $n_{\rm LOCGscf} \geq 1 $. 
The complete LOCG-accelerated NEPvADMM framework for solving \eqref{eq:TheProb} is summarized in Algorithm~\ref{locg.smooth.nepvadmm-simplified}.

\begin{algorithm}[tbhp]
\caption{NEPvADMM with LOCG Acceleration}\label{locg.smooth.nepvadmm-simplified}
\begin{algorithmic}[1]
\REQUIRE $X_{0} \in \bbX \subseteq \STM{p}{n} $, $Y_{0}$, $\Lambda_{0}$, $\beta > 0$,
         $\gamma > 0$, and $n_{\rm LOCGscf}\ge 1$;
\ENSURE an approximate solution of \eqref{eq:TheProb}.

\STATE $k=0$, $X_{-1} = [\;]$ (null matrix);
\WHILE{certain stopping criterion is not reached}

\STATE compute $W_k \in \STM{\hat n}{n}$ such that
       $\cR\left(W_k\right) =\cR\left(\left[X_k,\scrG(X_k),X_{k-1}\right]\right)$ and $X_k$ occupies the first $p$ columns of $W_k$;

\STATE take $P_{k,0}$ as the first $p$ columns of $I_{\hat n}$;

\FOR{$i=0,\ldots,n_{\rm LOCGscf}-1$}

\STATE construct $\wtd H_{k,i} = W_k^{\T}H_k(W_kP_{k,i})W_k$;

\STATE compute the partial eigen-decomposition $\wtd H_{k,i}\widehat{P}_{k,i} = \widehat{P}_{k,i}\widetilde{\Omega}_{k,i}$
       for $\widehat{P}_{k,i} \in \STM{p}{\hat n}$, an orthonormal basis matrix of the eigenspace associated with the $p$
       largest eigenvalues of $\wtd H_{k,i}$;

\STATE compute $\wtd Q_{k,i} \in \STM{p}{p}$ such that $P_{k,i+1} = \widehat{P}_{k,i}\wtd Q_{k,i} \in \mathbb{P}$, according to the NEPv Ansatz;

\ENDFOR

\STATE $X_{k+1} = W_kP_{k, n_{\rm LOCGscf}}$;

\STATE $Z_{k+1} = \Prox_{[(1+\beta\gamma)/\beta]r} \Big(\Psi X_{k+1} + \frac{1}{\beta}\Lambda_k\Big)$;

\STATE $Y_{k+1} = \frac{\gamma}{1+\beta\gamma} \Big(\frac{1}{\gamma}Z_{k+1} + \Lambda_k + \beta \Psi X_{k+1}\Big)$;

\STATE $\Lambda_{k+1} = \Lambda_k + \beta (\Psi X_{k+1} - Y_{k+1})$;

\STATE  $k = k+1$;

\ENDWHILE

\RETURN  $X_k$ (and corresponding $Y_k$ and/or $Z_k$ if also desired).
\end{algorithmic}
\end{algorithm}

\section{Convergence and Iteration Complexity}\label{Se4}
In this section, we will focus on two subcases: the $X$-subproblem \eqref{eq:MRADMM-frame:X} is assumed to solve exactly or by some pre-selected number of SCF iterations.

Before presenting our main results, we briefly review several key properties of the Moreau envelope $r_{\gamma}$.
The idea of smoothening $r$ by the Moreau envelope for composite optimization traces back to \cite{zeng2022moreau} and more recently \cite{li2024riemannian}.
Because of smoothening, $r_{\gamma}$ becomes differentiable and its gradient is Lipschitz continuous \cite{beck2017first,boyd2013proximal,rockafellar2009variational}, and that will help us analyze the iteration complexity of our NEPvADMM for obtaining an $\epsilon$-stationary solution. 
Lemma~\ref{le2} below adapts results established in \cite{beck2017first, li2024riemannian},
which were originally formulated for vector arguments but extend straightforwardly to the matrix setting.

\begin{lemma}[{\cite{beck2017first}, \cite[Lemma 2]{li2024riemannian}}]
\label{le2}
For convex function $r$ which is $L_{r}$-Lipschitz continuous, its Moreau envelope $r_{\gamma}$ in \eqref{smoothprob} is also $L_{r}$-Lipschitz continuous and, moreover,
     \[
     0 \leq r(Y) - r_{\gamma}(Y) \leq \gamma L_{r}^2,
     \]
     and the gradient of $r_{\gamma}$,
    \[
    \nabla r_{\gamma}(Y) = \frac{1}{\gamma}\big[Y - \Prox_{\gamma r}(Y)\big],
    \]
    has Lipschitz constant $1/\gamma$.
\end{lemma}

\subsection{When the $X$-subproblem is solved exactly}\label{Se3.1}

Previously, we introduced $\scrL_{\beta,\gamma}(\cdot,\cdot,\cdot)$ in \eqref{smoothprob.auglag}.
Going forward, we need
the Lagrangian function
\begin{equation}\label{eq:Lagrangian(gamma)}
\scrL_{\gamma}(X,Y,\Lambda) := f(X) + r_{\gamma}(Y) + \left\langle \Lambda, \Psi X - Y \right\rangle,
\end{equation}
of \eqref{smoothprob}, as well as the Lagrangian function $\scrL(\cdot,\cdot,\cdot)$ of \eqref{splitprob}
in \eqref{eq:Lagrangian(0)} later.

The first-order optimality condition of \eqref{smoothprob} is given by \cite{yang2014optimality}:
\begin{subequations}\label{exact.proof.kkt}
    \begin{align}
        0 & = \GRAD_{X}{\scrL_{\gamma}}_{\vert \STM{p}{n}}\left(X_{\ast},Y_{\ast},\Lambda_{\ast}\right), \label{exact.proof.kkt.a} \\
        0 & = \nabla_{Y}\scrL_{\gamma}(X_{\ast},Y_{\ast},\Lambda_{\ast}) = \nabla r_{\gamma}(Y_{\ast}) - \Lambda_{\ast}, \label{exact.proof.kkt.b} \\
        0 & = \Psi X_{\ast} - Y_{\ast}, \label{exact.proof.kkt.c} \\
        X_{\ast} & \in \STM{p}{n}. \label{exact.proof.kkt.d}
    \end{align}
\end{subequations}
We first establish the following results, which adapt the vector-argument findings of \cite{li2024riemannian,wang2019global} to the matrix setting.

\begin{lemma}
\label{le3}
     The dual variable $\Lambda$ and the primal variable $Y$ satisfy
    \begin{subequations}\label{bound.dual.by.primal}
    \begin{align}
    \nabla r_{\gamma}(Y_{k+1}) & = \Lambda_{k+1}, \label{bound.dual.by.primal-1}\\
    \left\| \Lambda_{k+1} - \Lambda_k \right\|_{\F} & \leq \frac{1}{\gamma}\left\| Y_{k+1} - Y_k\right\|_{\F}.
           \label{bound.dual.by.primal-2}
    \end{align}
    \end{subequations}
\end{lemma}

\begin{proof}
    The optimality condition for the $Y$-subproblem \eqref{eq:MRADMM-frame:Y} is given by
    \begin{equation}\label{eq:KKY4Y-sub}
    \nabla r_{\gamma}(Y_{k+1}) - \Lambda_k + \beta \left(Y_{k+1} - \Psi X_{k+1}\right) = 0.
    \end{equation}
    Rearranging these terms and applying the multiplier update rule yields \eqref{bound.dual.by.primal-1}. Using this identity and the $1/\gamma$-Lipschitz continuity of $\nabla r_{\gamma}$, we derive \eqref{bound.dual.by.primal-2} as follows:
    \[
    \begin{aligned}
        \| \Lambda_{k+1} - \Lambda_k \|_{\F} = \; & \| \nabla r_{\gamma}(Y_{k+1}) - \nabla r_{\gamma}(Y_k) \|_{\F} \\
        \le \; & \frac{1}{\gamma}\left\| Y_{k+1} - Y_k\right\|_{\F}. 
    \end{aligned}
    \]
\end{proof}

The next lemma shows the descent of $\scrL_{\beta,\gamma}$ at each ADMM iteration.

\begin{lemma}
\label{le3.1}
    If $\beta > 2/\gamma + 2 $, then the sufficient descent property
    \begin{equation}\label{exact.proof.suffi.descent}
        \scrL_{\beta,\gamma}\left(X_k,Y_k,\Lambda_k\right) - \scrL_{\beta,\gamma}\left(X_{k+1},Y_{k+1},\Lambda_{k+1}\right) \geq \| Y_{k+1} - Y_k \|_{\F}^2
    \end{equation}
    holds for all $k\ge 0$.
\end{lemma}

\begin{proof}
     After solving the $X$-subproblem \eqref{eq:MRADMM-frame:X}, we have
     \begin{equation}\label{exact.proof.suffi.descent.xsub}
        \scrL_{\beta,\gamma}\left(X_k,Y_k,\Lambda_k\right) \geq \scrL_{\beta,\gamma}\left(X_{k+1},Y_k,\Lambda_k\right).
    \end{equation}
    Next, after updating $Y$ and $\Lambda$ according to \eqref{eq:MRADMM-frame:Y} and \eqref{eq:MRADMM-frame:Lamb}, we have
    \begin{align}
    &\scrL_{\beta,\gamma}\left(X_{k+1},Y_k,\Lambda_k\right) - \scrL_{\beta,\gamma}\left(X_{k+1},Y_{k+1},\Lambda_{k+1}\right)
                       \nonumber\\
        & = r_{\gamma}(Y_k) - r_{\gamma}(Y_{k+1}) + \left\langle \Lambda_{k+1}, Y_{k+1} - Y_k \right\rangle + \frac{\beta}{2}\Big(\| Y_k \|_{\F}^2 - \| Y_{k+1} \|_{\F}^2 \nonumber\\
            &\quad  + 2\left\langle \Psi X_{k+1}, Y_{k+1} - Y_k \right\rangle - 2\left\langle \Psi X_{k+1} - Y_{k+1}, \Psi X_{k+1} - Y_k \right\rangle\Big) \nonumber \\
        & = r_{\gamma}(Y_k) - r_{\gamma}(Y_{k+1}) + \langle \nabla r_{\gamma}(Y_{k+1}), Y_{k+1} - Y_k \rangle \nonumber\\
         &\quad + \frac{\beta}{2}\left\| Y_k - Y_{k+1} \right\|_{\F}^2
             - \frac{1}{\beta}\left\| \Lambda_{k+1} - \Lambda_k \right\|_{\F}^2\nonumber\\
        &\geq -\frac{1}{2\gamma}\left\| Y_k - Y_{k+1} \right\|_{\F}^2
            + \frac{\beta}{2}\left\| Y_k - Y_{k+1} \right\|_{\F}^2
            - \frac{1}{\beta \gamma^2}\left\| Y_k - Y_{k+1} \right\|_{\F}^2 \label{exact.proof.suffi.descent.zsub'}\\
        &=\left(-\frac{1}{2\gamma} + \frac{\beta}{2} - \frac{1}{\beta \gamma^2}\right)\left\| Y_k - Y_{k+1} \right\|_{\F}^2
            \nonumber\\
        & \geq \left\| Y_k - Y_{k+1} \right\|_{\F}^2,
           \label{exact.proof.suffi.descent.zsub}
    \end{align}
    where, to obtain \eqref{exact.proof.suffi.descent.zsub'}, we have used \eqref{bound.dual.by.primal-2} from Lemma~\ref{le3} along with the inequality
    \[
    r_{\gamma}(Y_k) - r_{\gamma}(Y_{k+1}) + \left\langle \nabla r_{\gamma}(Y_{k+1}), Y_{k+1} - Y_k \right\rangle \geq -\frac{1}{2\gamma}\left\| Y_k - Y_{k+1} \right\|_{\F}^2,
    \]
    and to \eqref{exact.proof.suffi.descent.zsub}, we have invoked the assumption $\beta\gamma>2+2\gamma$ to get
    \begin{align*}
    -\frac{1}{2\gamma} + \frac{\beta}{2} - \frac{1}{\beta \gamma^2}
      >-\frac{1}{2\gamma} + \frac 1{\gamma}+1 - \frac{1}{2(1+\gamma) \gamma} = 1+\frac 1{2(1+\gamma)} > 1.
    \end{align*}
    Combining \eqref{exact.proof.suffi.descent.xsub} with \eqref{exact.proof.suffi.descent.zsub} yields the desired result \eqref{exact.proof.suffi.descent}.
\end{proof}

\begin{lemma}
\label{le3.3}
Suppose $\beta > 2/\gamma + 2$. Then the following statements hold.
\begin{enumerate}[{\rm (a)}]
\item The sequence $\left\{\scrL_{\beta,\gamma}(X_k,Y_k,\Lambda_k)\right\}$ is lower bounded for all $k \in \mathbb{N}$ and converges as $k \rightarrow \infty$.
\item The sequence $\left\{(X_k, Y_k, \Lambda_k)\right\}$ is bounded.
\end{enumerate}
\end{lemma}

\begin{proof}
    For $X \in \STM{p}{n}$, the function $f(X) + r_{\gamma}(\Psi X)$ is continuous and hence bounded.
    First, recall that $r_{\gamma}$ is $1/\gamma$-Lipschitz continuous, which yields the following bound:
    \[
    r_{\gamma}(\Psi X_k) \leq r_{\gamma}(Y_k)
          + \left\langle \nabla r_{\gamma}(Y_k), \Psi X_k - Y_k \right\rangle + \frac{1}{2\gamma}\| \Psi X_k - Y_k \|_{\F}^2.
    \]
    Additionally, invoking the assumption $\beta > 2/\gamma + 2$ guarantees that the penalty parameter satisfies:
    \[
    \frac{\beta}{2} - \frac{1}{2\gamma} > 1 + \frac{1}{\gamma} - \frac{1}{2\gamma} = 1 + \frac{1}{2\gamma}.
    \]
    By substituting the Lipschitz bound into the definition of $\scrL_{\beta,\gamma}$ and applying \eqref{bound.dual.by.primal-1} from Lemma~\ref{le3}, we can use the parameter bound to show that the augmented Lagrangian is uniformly lower bounded:
    \begin{align*}
        \scrL_{\beta,\gamma}(X_k,Y_k,\Lambda_k) & = f(X_k) + r_{\gamma}(Y_k)
          + \left\langle \nabla r_{\gamma}(Y_k), \Psi X_k - Y_k \right\rangle
          + \frac{\beta}{2}\| \Psi X_k - Y_k \|_{\F}^2 \\
        & \geq f(X_k) + r_{\gamma}(\Psi X_k)
          + {\left(\frac{\beta}{2} - \frac{1}{2\gamma}\right)}\| \Psi X_k - Y_k \|_{\F}^2 \\
        & > f(X_k) + r_{\gamma}(\Psi X_k)
          + {\left(1+ \frac{1}{2\gamma}\right)}\| \Psi X_k - Y_k \|_{\F}^2 \\
        &\ge\min_{X\in \STM{p}{n}}\big[f(X) + r_{\gamma}(\Psi X)\big].
    \end{align*}
    Therefore, $\left\{\scrL_{\beta,\gamma}(X_k,Y_k,\Lambda_k)\right\}$ converges as $k \rightarrow \infty$ because the sequence is both monotone and bounded from below. For item (b), $\scrL_{\beta,\gamma}(X_k,Y_k,\Lambda_k)$ is upper bounded by $\scrL_{\beta,\gamma}(X_{0},Y_{0},\Lambda_{0})$ due to monotonicity, and so are $f(X_k) + r_{\gamma}(\Psi X_k)$ and $\| \Psi X_k - Y_k \|_{\F}^2$. As the Stiefel manifold is compact, $\left\{X_k\right\}$ is bounded, and thus $\left\{Y_k\right\}$ is bounded. Using Lemma \ref{le2}, we know $\left\{\left\| \nabla r_{\gamma}(Y_k) \right\|_{\F}\right\}$ is bounded. Then by Lemma \ref{le3}, we see that $\left\{\Lambda_k\right\}$ is also bounded.
\end{proof}

For $(X, Y, \Lambda)$ with $X \in \STM{p}{n}$, it can be seen that 
\[
    \nabla \scrL_{\gamma}(X, Y, \Lambda) = \left[ \begin{array}{c}
\GRAD_{X}{\scrL_{\gamma}}_{\vert \STM{p}{n}}\left(X, Y, \Lambda\right) \\ \nabla r_{\gamma}(Y) - \Lambda \\ \Psi X - Y \end{array} \right].
\]
The following lemma gives an upper bound on $\nabla \scrL_{\gamma}(X_{k+1},Y_{k+1},\Lambda_{k+1})$.

\begin{lemma}
    There exists a constant $\hat{c} > 0$ such that
    \begin{equation}\label{exact.upper.bound}
        \| \nabla \scrL_{\gamma}(X_{k+1},Y_{k+1},\Lambda_{k+1}) \|_{\F} \leq \hat{c}\| Y_{k+1} - Y_k \|_{\F}.
    \end{equation}
\end{lemma}

\begin{proof}
    Since the $X$-subproblem is solved exactly at each iteration, we have
    \[
    \begin{aligned}
        0 = \; &  \GRAD_{X}{\scrL_{\gamma}}_{\vert \STM{p}{n}}\left(X_{k+1}, Y_k, \Lambda_k \right)  \\
        = \; &  \mit{\Pi}_{X}\left(\frac{\partial f(X_{k+1})}{\partial X} + \Psi^{\T}\Lambda_k
          + \beta \Psi^{\T}(\Psi X_{k+1} - Y_{k})\right).
    \end{aligned}
    \]
    Using this identity, we can bound the first block of $\nabla \scrL_{\gamma}(X_{k+1},Y_{k+1},\Lambda_{k+1})$ as follows:
    \begin{align}
    &\big\|\GRAD_{X}{\scrL_{\gamma}}_{\vert \STM{p}{n}}\left(X_{k+1}, Y_{k+1}, \Lambda_{k+1}\right)\big\|_{\F} \nonumber\\
      &= \Big\|\mit{\Pi}_{X}\left(\frac{\partial f(X_{k+1})}{\partial X} + \Psi^{\T}\Lambda_{k+1}\right)\Big\|_{\F} \nonumber\\
      &\leq \Big\| \mit{\Pi}_{X}\left(\frac{\partial f(X_{k+1})}{\partial X} + \Psi^{\T}\Lambda_k
          + \beta \Psi^{\T}(\Psi X_{k+1} - Y_{k})\right) \Big\|_{\F}
          + \Big\| \mit{\Pi}_{X}\left(\beta \Psi^{\T}(Y_k - Y_{k+1})\right) \Big\|_{\F} \nonumber \\
      & \leq \beta\|\Psi\|_2\|Y_k - Y_{k+1}\|_{\F}.  \label{exact.bound.first}
    \end{align}
    By Lemma \ref{le3}, we known $\|\nabla r_{\gamma}(Y_{k+1}) - \Lambda_{k+1}\|_{\F} = 0$
    and then use \eqref{eq:KKY4Y-sub} to get
    \begin{equation}\label{exact.bound.third}
        \|\Psi X_{k+1} - Y_{k+1}\|_{\F}
         = \Big\| \frac{1}{\beta}(\Lambda_{k+1} - \Lambda_k) \Big\|_{\F}
         \leq \frac{1}{\beta\gamma}\|Y_k - Y_{k+1}\|_{\F}.
    \end{equation}
    The upper bound \eqref{exact.upper.bound} thus holds.
\end{proof}

Finally, we have the following convergence result for ADMM \eqref{eq:MRADMM-frame}.

\begin{theorem}
    Suppose $\beta > 2/\gamma + 2$.
    If each subproblem in ADMM \eqref{eq:MRADMM-frame} is solved exactly, then there exists at least one accumulation point $(X_\ast, Y_\ast, \Lambda_\ast)$ of the sequence $\{(X_k, Y_k, \Lambda_k)\}_{k=0}^{\infty}$, and each accumulation point is a KKT point of problem \eqref{smoothprob}.
\end{theorem}

\begin{proof}
    The boundedness of the sequence $\{(X_k,Y_k,\Lambda_k)\}$ by Lemma \ref{le3.3} ensures the existence of a convergent subsequence $\{(X_{k_{s}},Y_{k_{s}},\Lambda_{k_{s}})\}$ and its limit point $(X_{\ast},Y_{\ast},\Lambda_{\ast})$. 
    As the sequence $\{\scrL_{\beta,\gamma}(X_k,Y_k,\Lambda_k)\}$ is monotone and lower bounded, the sufficient descent property \eqref{exact.proof.suffi.descent} indicates that
\[
\| Y_{k+1} - Y_k \|_{\F} \rightarrow 0 \;\;\; {\rm as} \;\;\; k \rightarrow \infty.
\]
Then we have $\|\GRAD_{X}{\scrL_{\gamma}}_{\vert \STM{p}{n}}(X_k, Y_k, \Lambda_k)\|_{\F} \rightarrow 0$
and $\|\Psi X_k - Y_k\|_{\F} \rightarrow 0$ by \eqref{exact.bound.first} and \eqref{exact.bound.third}.
Specifically, $ \|\Psi X_{k_{s}} - Y_{k_{s}}\|_{\F} \rightarrow 0 $ and $\|\GRAD_{X}{\scrL_{\gamma}}_{\vert \STM{p}{n}}(X_{k_{s}}, Y_{k_{s}}, \Lambda_{k_{s}})\|_{\F} \rightarrow 0$ as $s \rightarrow \infty$. We know from Lemma \ref{le3} that \eqref{exact.proof.kkt.b} always holds. Thus the limit point $(X_{\ast},Y_{\ast},\Lambda_{\ast})$ satisfies the KKT condition \eqref{exact.proof.kkt} of problem \eqref{smoothprob}.
\end{proof}

\subsection{Iteration Complexity of NEPvADMM}\label{Se3.2}

It is noted that our eventual goal is to solve problem~\eqref{eq:TheProb} which is mathematically equivalent to the split formulation~\eqref{splitprob}. 
Our Algorithm~\ref{alg:smooth.nepvadmm} addresses~\eqref{splitprob} via problem~\eqref{smoothprob} whose variable $Y$ at optimality is actually approximated by $Z$ generated by the algorithm.
Therefore, the accuracy of the iterates $(X_k, Z_k, \Lambda_k)$ generated by Algorithm~\ref{alg:smooth.nepvadmm} should be ultimately evaluated according to the original problem~\eqref{splitprob}.
For better correspondence, we reformulate~\eqref{splitprob}, by replacing variable $Y$ with $Z$, into
\begin{equation}\label{splitprob.z}
\begin{split}
     \min_{X,Z} \;\; f(X) + r(Z) \quad 
     \mathrm{ s.t.}\;\; \Psi X - Z=0, \; X\in\STM{p}{n},
\end{split}
\end{equation}
so that the approximation pair $(X_k,Z_k)$ by Algorithm~\ref{alg:smooth.nepvadmm} at convergence
corresponds to the optimal solution $(X,Z)$ of this problem~\eqref{splitprob.z}.
In what follows, we will estimate how accurately the solution computed by Algorithm~\ref{alg:smooth.nepvadmm} satisfies
the KKT condition of \eqref{splitprob.z}. To this end, we recall that the KKT conditions of \eqref{splitprob.z} is given by \cite{yang2014optimality}:
\begin{subequations}
    \begin{align}
        0 & = \GRAD_{X}\scrL_{\vert \STM{p}{n}}\left(X_\ast, Z_\ast, \Lambda_\ast \right), \label{smooth.proof.kkt.a} \\
        0 & \in \partial_{Z}\scrL(X_\ast, Z_\ast, \Lambda_\ast) = \partial r(Z_\ast) - \Lambda_\ast, \label{smooth.proof.kkt.b} \\
        \Psi X_\ast & = Z_\ast, \label{smooth.proof.kkt.c} \\
        X_\ast & \in \STM{p}{n}, \label{smooth.proof.kkt.d}
    \end{align}
\end{subequations}
where
\begin{equation}\label{eq:Lagrangian(0)}
\scrL(X, Z, \Lambda) := f(X) + r(Z) + \left\langle \Lambda, \Psi X - Z \right\rangle
\end{equation}
is the Lagrangian function associated with~\eqref{splitprob.z}. We have
\[
    \partial \scrL(X, Z, \Lambda) = \left[ \begin{array}{c}
\GRAD_{X}\scrL_{\vert \STM{p}{n}}\left(X, Z, \Lambda\right) \\ \partial r(Z) - \Lambda \\ \Psi X - Z \end{array} \right].
\]
A triplet
$(\overline{X},\overline{Z},\overline{\Lambda})\in\STM{p}{n}\times\bbR^{n\times p}\times\bbR^{n\times p}$
is said an $\epsilon$-stationary point of the problem \eqref{splitprob.z} if there exists
$G \in \partial \scrL(\overline{X},\overline{Z},\overline{\Lambda})$ such that
$\left\| G \right\|_{\F} \leq \epsilon$.

Combine \eqref{eq:MRADMM-frame:Lamb} and \eqref{smoothprob.z.kkt.a} to get 
$
\Lambda_{k+1} = (Y_{k+1} - Z_{k+1})/\gamma,
$
and thus, by \eqref{smoothprob.z.kkt.b},
\begin{equation}\label{smooth.proof.lam.in.parde.r}
    \Lambda_{k+1} \in \partial r(Z_{k+1}),
\end{equation}
which yields
\begin{equation}\tag{\ref{smooth.proof.kkt.b}$'$}
0 \in \partial_{Z}\scrL(X_{k+1}, Z_{k+1}, \Lambda_{k+1}) = \partial r(Z_{k+1}) - \Lambda_{k+1}.
\end{equation}

In order to assess the progress made by  SCF in Algorithm~\ref{alg:smooth.nepvadmm}, we need to invoke the following result.

\begin{lemma}[{\cite[Lemma B.10]{li2024theory}}]\label{le3.5}
    Let $H \in \bbS^{n \times n}$ and $X_{\ast} \in \STM{p}{n}$ whose column space $\cR\left(X_{\ast}\right)$ is the invariant subspace of $H$ associated with its $p$ largest eigenvalues. Suppose that $\lambda_p(H) - \lambda_{p+1}(H) > 0$. Given $X \in \STM{p}{n}$, let
    \[
    \eta = \tr(X^{\T}_{\ast}HX_{\ast}) - \tr( X^{\T}HX ), \;\;\; \epsilon = \sqrt{\frac{\eta}{\lambda_p(H) - \lambda_{p+1}(H)}}.
    \]
    Then we have
    \[
    \frac{\left\| HX - X\left(X^{\T}HX\right) \right\|_{\F}}{\lambda_1(H) - \lambda_n(H)}
    \leq \epsilon.
    \]
\end{lemma}

The following lemma shows the sufficient decrease by one SCF iteration for the $X$-subproblem \eqref{eq:MRADMM-frame:X}.

\begin{lemma}\label{le4}
     Suppose that $H_k(X) \in \bbS^{n \times n}$ satisfies \eqref{equa.nepv.parde} and the NEPv Ansatz on page~\pageref{NEPvAnsatz}
     holds for $g_k(X)$ with $H_k(X)$. If $\scrM_k\left(X_{k,i}\right)$ is symmetric, and there exist positive constants $\delta$ and $\lambda$ such that for any $X \in \STM{p}{n}$ and $k \geq 0$,
     \begin{equation}\label{smooth.proof.suff.cond}
         \lambda_p\left(H_k(X)\right) - \lambda_{p+1}\left(H_k(X)\right) \geq \delta, \quad
         \left\vert \lambda_1\left(H_k(X)\right) \right\vert \leq \lambda .
     \end{equation}
Then for all $k \geq 0$ and $i \geq 0$, we have
    \begin{equation}\label{smooth.proof.suff.res.1scf}
        g_k(X_{k,i+1}) - g_k(X_{k,i}) \geq c \left\| \GRAD\;g_{k\vert \STM{p}{n}}(X_{k,i}) \right\|^2_{\F},
    \end{equation}
    where $c = \omega \delta / \left(4\lambda^2\right)$, and $\omega > 0$ is a constant of the NEPv Ansatz.
\end{lemma}

\begin{proof}
We claim that
\begin{equation}\label{smooth.proof.suff.relat.grad.nepv}
    \left\| \GRAD\;g_{k\vert \STM{p}{n}}(X_{k,i}) \right\|_{\F}^2 = \left\| H_k(X_{k,i})X_{k,i} - X_{k,i}\Phi_{k,i} \right\|_{\F}^2,
\end{equation}
where $\Phi_{k,i} = X_{k,i}^{\T}H_k(X_{k,i})X_{k,i}$. To prove \eqref{smooth.proof.suff.relat.grad.nepv},
we notice by \eqref{equa.nepv.parde}:
\begin{equation}\label{smooth.proof.suff.partde.nepv.equa}
    \frac{\partial g_k(X)}{\partial X} + X\scrM_k(X) = H_k(X)X.
\end{equation}
Then inserting \eqref{smooth.proof.suff.partde.nepv.equa} into the right-hand side of \eqref{smooth.proof.suff.relat.grad.nepv}, we get
\[
    \begin{aligned}
&{\rm RHS} \; {\rm of} \; \eqref{smooth.proof.suff.relat.grad.nepv} \\
 & = \left\| \frac{\partial g_k(X_{k,i})}{\partial X} + X_{k,i}\scrM_k(X_{k,i}) - X_{k,i}\sym\left(X^{\T}_{k,i}\left(\frac{\partial g_k(X_{k,i})}{\partial X} + X_{k,i}\scrM_k(X_{k,i})\right)\right) \right\|_{\F}^2 \\
        & = \left\| \frac{\partial g_k(X_{k,i})}{\partial X}
            - X_{k,i}\sym\left(X^{\T}_{k,i}\frac{\partial g_k(X_{k,i})}{\partial X}\right) \right\|_{\F}^2 = \left\| \GRAD\;g_{k\vert \STM{p}{n}}(X_{k,i}) \right\|_{\F}^2, 
    \end{aligned}
\]
as expected. Because in SCF 
\[
    \eta_{k,i} := \tr(\what X^{\T}_{k,i}H_k(X_{k,i})\what X_{k,i}) - \tr(X_{k,i}^{\T}H_k(X_{k,i})X_{k,i}) \geq 0
\]
by design, we have
\begin{equation}\label{smooth.proof.suff.scf.mono}
    g_k(X_{k,i+1}) \geq g_k(X_{k,i}) + \omega \eta_{k,i} \geq g_k(X_{k,i})
\end{equation}
by the NEPv Ansatz. By Lemma \ref{le3.5}, the following inequality links $\eta_{k,i}$ to \\
$\left\| H_k(X_{k,i})X_{k,i} - X_{k,i}\Phi_{k,i} \right\|_{\F}^2$ in \eqref{smooth.proof.suff.relat.grad.nepv} via
\begin{equation}\label{smooth.proof.suff.bound}
    \delta_{k,i}\frac{\left\| H_k(X_{k,i})X_{k,i} - X_{k,i}\Phi_{k,i} \right\|_{\F}^2}{4\left\| H_k(X_{k,i}) \right\|_2^2} \leq \delta_{k,i}\frac{\left\| H_k(X_{k,i})X_{k,i} - X_{k,i}\Phi_{k,i} \right\|_{\F}^2}{\left[\lambda_1(H_k(X_{k,i})) - \lambda_n(H_k(X_{k,i}))\right]^2} \leq \eta_{k,i},
\end{equation}
where $\delta_{k,i} = \lambda_p(H_k(X_{k,i})) - \lambda_{p+1}(H_k(X_{k,i}))$. Combining \eqref{smooth.proof.suff.relat.grad.nepv}, \eqref{smooth.proof.suff.scf.mono} and \eqref{smooth.proof.suff.bound} all together, we get
\[
g_k(X_{k,i+1}) - g_k(X_{k,i}) \geq c_{k,i} \left\| \GRAD\;g_{k\vert \STM{p}{n}}(X_{k,i}) \right\|_{\F}^2,
\]
where $c_{k,i} = \omega \delta_{k,i} / \left(4\| H_k(X_{k,i}) \|_2^2\right)$. Under the assumption \eqref{smooth.proof.suff.cond}, we have $0 < c = \omega \delta / \left(4\lambda^2\right) \leq c_{k,i}$. That completes the proof.
\end{proof}

Consequently, Lemma \ref{le4} yields the following corollary, which establishes the monotonicity in objective $g_k(\cdot)$
with respect to the $X$-iterates after $n_{\scf}$ SCF iterations.

\begin{corollary}\label{cor:suff.decrease.X-sub}
Suppose that $H_k(X) \in \bbS^{n \times n}$ satisfies \eqref{equa.nepv.parde} and that $\scrM_k\left(X_{k,i}\right)$ is symmetric for all $i \in \{0, \dots, n_{\scf} - 1\}$. Assume  that the NEPv Ansatz holds for $g_k(X)$ with $H_k(X)$, and that there exist positive constants $\delta$ and $\lambda$ such that \eqref{smooth.proof.suff.cond} holds for any $X \in \STM{p}{n}$. Then it holds that
\begin{equation}\label{smooth.proof.suff.res}
    g_k(X_{k+1}) - g_k(X_k) \geq c \left\| \GRAD\;g_{k\vert \STM{p}{n}}(X_k) \right\|^2_{\F},
\end{equation}
where $c = \omega \delta / \left(4\lambda^2\right)$.
\end{corollary}

\begin{proof}
For $n_{\scf} = 1$, we know from Algorithm \ref{alg:smooth.nepvadmm} that $X_k = X_{k,0}$ and $X_{k+1} = X_{k,1}$. Consequently, \eqref{smooth.proof.suff.res.1scf} directly implies \eqref{smooth.proof.suff.res}.

In the case where $n_{\scf} > 1$, the following holds provided that the assumptions of Corollary \ref{cor:suff.decrease.X-sub} are met:
\[
\begin{aligned}
    g_k(X_{k+1}) - g_k(X_k) &= \sum_{i=0}^{n_{\scf}-1} \Big( g_k(X_{k,i+1}) - g_k(X_{k,i}) \Big) \\
    &\geq c \sum_{i=0}^{n_{\scf}-1} \left\| \GRAD \; g_{k\vert \STM{p}{n}}(X_{k,i}) \right\|^2_{\F} \\
    &\geq c \left\| \GRAD \; g_{k\vert \STM{p}{n}}(X_k) \right\|^2_{\F},
\end{aligned}
\]
which is \eqref{smooth.proof.suff.res}.
\end{proof}

Next, we bound the change in the augmented Lagrangian function $\scrL_{\beta,\gamma}$ in \eqref{smoothprob.auglag}
per iteration within Algorithm~\ref{alg:smooth.nepvadmm}.

\begin{lemma}\label{le5}
    Suppose that the NEPv Ansatz holds for $g_k(X)$ with $H_k(X) \in \bbS^{n \times n}$ satisfying \eqref{equa.nepv.parde}. Assume further that $\scrM_k\left(X_{k,i}\right)$ is symmetric for all $i \in \{0, \dots, n_{\scf} - 1\}$, and that there exist positive constants $\delta$ and $\lambda$ such that \eqref{smooth.proof.suff.cond} holds for any $X \in \STM{p}{n}$. Let $\left\{\left(X_k, Y_k, Z_k,\Lambda_k\right)\right\}$ be the sequence generated by Algorithm~\ref{alg:smooth.nepvadmm}. Then we have for $k\ge 0$
    \begin{align}
         \scrL_{\beta,\gamma}\left(X_{k+1}, Y_{k+1}, \Lambda_{k+1}\right)
             & - \scrL_{\beta,\gamma}\left(X_k, Y_k, \Lambda_k\right) \nonumber \\
         &\leq  \left(\frac{1}{\beta \gamma^2} - \frac{\beta}{2}\right)\| Y_{k+1} - Y_k \|_{\F}^2
             - c\left\| \GRAD\;g_{k\vert \STM{p}{n}}(X_k) \right\|_{\F}^2.  \label{smooth.proof.auglag.change}
    \end{align}
\end{lemma}

\begin{proof}
Express the left-hand side of \eqref{smooth.proof.auglag.change} as
    \[
        \begin{aligned}
             {\rm LHS \; of \;} \eqref{smooth.proof.auglag.change}
        &= \bigl[\scrL_{\beta,\gamma}(X_{k+1}, Y_{k+1}, \Lambda_{k+1}) - \scrL_{\beta,\gamma}(X_{k+1}, Y_{k+1}, \Lambda_k)\bigr] \\
        &\quad + \bigl[ \scrL_{\beta,\gamma}(X_{k+1}, Y_{k+1}, \Lambda_k) - \scrL_{\beta,\gamma}(X_{k+1}, Y_k, \Lambda_k) \bigr] \\
        &\quad + \bigl[ \scrL_{\beta,\gamma}(X_{k+1}, Y_k, \Lambda_k) - \scrL_{\beta,\gamma}(X_k, Y_k, \Lambda_k) \bigr] \\
        &\equiv {\rm L_1} + {\rm L_2} + {\rm L_3}.
        \end{aligned}
    \]
    For the term ${\rm L_1}$, we obtain the following estimate:
\begin{equation}\label{smooth.proof.auglag.change.lam}
    {\rm L_1} = \frac{1}{\beta}\| \Lambda_{k+1} - \Lambda_k \|_{\F}^2 
     \leq \frac{1}{\beta \gamma^2}\| Y_{k+1} - Y_k \|_{\F}^2,
\end{equation}
where the final inequality follows directly from \eqref{bound.dual.by.primal-2} in Lemma~\ref{le3}.
For ${\rm L_2}$, we expand the norms and apply the multiplier update to obtain:
\begin{align}
    {\rm L_2} = & \; r_{\gamma}(Y_{k+1}) - r_{\gamma}(Y_k) + \left\langle \Lambda_k, Y_k - Y_{k+1} \right\rangle \nonumber \\
    & - \frac{\beta}{2}\left( \| Y_{k+1} - Y_k \|_{\F}^2 - 2 \| Y_{k+1} \|_{\F}^2 + 2 \langle Y_k, Y_{k+1} \rangle - 2 \langle \Psi X_{k+1}, Y_k - Y_{k+1} \rangle \right) \nonumber \\
     = & \; r_{\gamma}(Y_{k+1}) - r_{\gamma}(Y_k) + \left\langle \Lambda_{k+1}, Y_k - Y_{k+1} \right\rangle - \frac{\beta}{2}\| Y_{k+1} - Y_k \|_{\F}^2 \nonumber \\
     = & \; r_{\gamma}(Y_{k+1}) - r_{\gamma}(Y_k) + \left\langle \nabla r_{\gamma}(Y_{k+1}), Y_k - Y_{k+1} \right\rangle - \frac{\beta}{2}\| Y_{k+1} - Y_k \|_{\F}^2 \nonumber \\
     \leq & - \frac{\beta}{2}\| Y_{k+1} - Y_k \|_{\F}^2, \label{smooth.proof.auglag.change.z}
\end{align}
where the final equality invokes \eqref{bound.dual.by.primal-1} from Lemma~\ref{le3}, and the concluding inequality \eqref{smooth.proof.auglag.change.z} follows from the convexity of $r_{\gamma}$, which guarantees that
\[
r_{\gamma}(Y_{k+1}) + \left\langle \nabla r_{\gamma}(Y_{k+1}), Y_k - Y_{k+1} \right\rangle \leq r_{\gamma}(Y_k).
\]
Finally, since $g_k(X) = - \scrL_{\beta,\gamma}(X, Y_k, \Lambda_k)$, the term ${\rm L_3}$ is given by \eqref{smooth.proof.suff.res} in Corollary~\ref{cor:suff.decrease.X-sub}. Substituting the bounds from \eqref{smooth.proof.auglag.change.lam}, \eqref{smooth.proof.auglag.change.z}, and \eqref{smooth.proof.suff.res} into our initial expansion yields the desired inequality \eqref{smooth.proof.auglag.change}.
\end{proof}

As the objective function of the optimization problem \eqref{eq:TheProb} is continuous
on the Stiefel manifold, it has minimum value $F_{\ast}$.
The following lemma provides a lower bound for the augmented Lagrangian function $\scrL_{\beta,\gamma}$ \eqref{smoothprob.auglag}.

\begin{lemma}[{\cite[Lemma 7]{li2024riemannian}}]\label{le6}
    Suppose $\beta \gamma \geq 1$. Then the sequence $\left\{\scrL_{\beta,\gamma}(X_k, Y_k, \Lambda_k)\right\}$ is uniformly lower bounded by $F_{\ast} - \gamma L_{r}^2$.
\end{lemma}

\begin{proof}
Because $\nabla r_{\gamma}$ is $1/\gamma$-Lipschitz continuous (see Lemma~\ref{le2}), we get
\[
r_{\gamma}\left(\Psi X_k\right) \leq r_{\gamma}\left(Y_k\right) + \left\langle \nabla r_{\gamma}\left(Y_k\right), \Psi X_k - Y_k \right\rangle + \frac{1}{2\gamma}\left\| \Psi X_k - Y_k \right\|_{\F}^2.
\]
Rearranging these terms and substituting $\nabla r_{\gamma}(Y_k) = \Lambda_k$ from Lemma~\ref{le3} yields
\begin{equation}\label{ineq.bound.for.r.and.lam}
    r_{\gamma}\left(\Psi X_k\right) - \frac{1}{2\gamma}\left\| \Psi X_k - Y_k  \right\|_{\F}^2 \leq r_{\gamma}\left(Y_k\right) + \left\langle \Lambda_k, \Psi X_k - Y_k \right\rangle.
\end{equation}
It then follows that the augmented Lagrangian can be bounded below as:
    \[
    \begin{aligned}
        \scrL_{\beta,\gamma}\left(X_k, Y_k, \Lambda_k\right) & \geq f\left(X_k\right) + r_{\gamma}\left(\Psi X_k\right) + \left( \frac{\beta}{2} - \frac{1}{2\gamma} \right)\left\| \Psi X_k - Y_k \right\|_{\F}^2 \\
        & \geq f\left(X_k\right) + r_{\gamma}\left(\Psi X_k\right) \geq f\left(X_k\right) + r\left(\Psi X_k\right) - \gamma L_r^2 \geq F_{\ast} - \gamma L_r^2,
    \end{aligned}
    \]
    where the first inequality applies the bound from \eqref{ineq.bound.for.r.and.lam}, the second relies on the assumption $\beta \gamma \geq 1$, and the third follows from Lemma~\ref{le2}.
\end{proof}

The following lemma gives an upper bound on $G_k \in \partial \scrL\left( X_k, Z_k, \Lambda_k \right)$.

\begin{lemma}\label{le7}
Let $\left\{\left(X_k, Y_k, Z_k, \Lambda_k\right)\right\}$ be the sequence of iterates generated by Algorithm~\ref{alg:smooth.nepvadmm}.
For any $k \geq 1$, there exists
\[
G_k \in \partial \scrL\left( X_k, Z_k, \Lambda_k \right)
  = \left[ \begin{array}{c}
          \GRAD_{X}\scrL_{\vert \STM{p}{n}}\left(X_k, Z_k, \Lambda_k\right) \\
          \partial r\left( Z_k \right) - \Lambda_k \\ \Psi X_k - Z_k
          \end{array} \right]
\]
that satisfies
\begin{equation}\label{smooth.proof.subgrad.bound}
    \| G_k \|^2_{\F} \leq 2\left\| \GRAD\;g_{k\vert \STM{p}{n}}(X_k) \right\|_{\F}^2 + \frac{2\left(\beta^2\|\Psi\|_2^2 + 1\right)}{\beta^2\gamma^2}\| Y_k - Y_{k-1} \|_{\F}^2 + 2\gamma^2 L_{r}^2.
\end{equation}
\end{lemma}

\begin{proof}
    We have $\Lambda_k \in \partial r\left( Z_k \right)$ by \eqref{smooth.proof.lam.in.parde.r}. Therefore, there exists $G_k \in \partial \scrL\left(X_k, Z_k, \Lambda_k \right)$ such that    
    \[
    \begin{aligned}
        \| G_k \|_{\F}^2  = & \; \left\| \GRAD_{X}\scrL_{\vert \STM{p}{n}}\left(X_k, Z_k, \Lambda_k\right) \right\|^2_{\F} + \| \Psi X_k - Z_k \|^2_{\F} \\
         = & \; \left\| \mit{\Pi}_{X}\left( \frac{\partial g_k (X_k)}{\partial X} \right) - \mit{\Pi}_{X}\left(\beta \Psi^{\T}(\Psi X_k - Y_k)\right)\right\|^{{\rm 2}}_{\F} + \| \Psi X_k - Y_k + Y_k - Z_k \|^{{\rm 2}}_{\F} \\
         \leq & \; 2\left\| \GRAD\;g_{k\vert \STM{p}{n}}(X_k) \right\|^2_{\F} + 2 \beta^2 \left\| \Psi^{\T}(\Psi X_k - Y_k) \right\|_{\F}^{{\rm 2}} \\
        & + {\rm 2}\| \Psi X_k - Y_k \|^{{\rm 2}}_{\F} + {\rm 2}\| Y_k - Z_k \|^{{\rm 2}}_{\F} \\
         \leq & \; 2\left\| \GRAD\;g_{k\vert \STM{p}{n}}(X_k) \right\|^2_{\F} + 2\left(\beta^2\|\Psi\|_2^2 + 1\right)\| \Psi X_k - Y_k \|^2_{\F} + 2\| Y_k - Z_k \|^2_{\F}.
    \end{aligned}
    \]
    Using \eqref{bound.dual.by.primal-2} in Lemma~\ref{le3}, we have
    \[
    \begin{aligned}
        2\left(\beta^2\|\Psi\|_2^2 + 1\right)\| \Psi X_k - Y_k \|^2_{\F} & = 2\left(\beta^2\|\Psi\|_2^2 + 1\right)\left\| \frac{1}{\beta}\left(\Lambda_k - \Lambda_{k-1}\right) \right\|^2_{\F}\\
        & \leq \frac{2\left(\beta^2\|\Psi\|_2^2 + 1\right)}{\beta^2\gamma^2}\left\| Y_k - Y_{k-1} \right\|^2_{\F}.
    \end{aligned}
    \]
    Rearranging the terms in \eqref{smoothprob.z.kkt.b} yields
    \[
    Y_{k+1} - Z_{k+1} \in \gamma \partial r(Z_{k+1}).
    \]
    Since $r$ is $L_r$-Lipschitz continuous, we have
    \[
    \| Y_{k+1} - Z_{k+1} \|_{\F} \leq \gamma L_r.
    \]
    Substituting these two upper bounds back into the inequality for $\| G_k \|_{\F}^2$ yields \eqref{smooth.proof.subgrad.bound}.
\end{proof}

\begin{theorem}\label{th2}
Suppose that the NEPv Ansatz holds for $g_k(X)$ with $H_k(X) \in \bbS^{n \times n}$ and
\eqref{equa.nepv.parde} holds. Assume further that $\scrM_k\left(X_{k,i}\right)$ is symmetric
for all $i \in \{0, \dots, n_{\scf} - 1\}$, and that there exist positive constants $\delta$ and $\lambda$ such that \eqref{smooth.proof.suff.cond} holds for any $X \in \STM{p}{n}$.
Let $\left\{\left(X_k, Y_k, Z_k, \Lambda_k\right)\right\}$ be the sequence generated by Algorithm~\ref{alg:smooth.nepvadmm},
and let $\epsilon > 0$ be the predefined tolerance, and set $\beta = 1 / \epsilon$, $\gamma = \sqrt{ 2 / \beta^2 + 2c\left(\beta^2\|\Psi\|_2^2 + 1\right) / \beta^3}$. Then there exists $G_k \in \partial \scrL\left(X_k, Z_k, \Lambda_k \right), k = 1,2,...$, such that
    \[
    \min_{k = 1,...,K} \left\| G_k \right\|_{\F}^2 \leq \epsilon^2,
    \]
    provided that
    \[
    K = \mathcal{O}\left(\frac{1}{\epsilon^3}\right).
    \]
\end{theorem}

\begin{proof}
    Inequality \eqref{smooth.proof.auglag.change} can be rewritten as
    \begin{align*}
    2\left\| \GRAD \; g_{k\vert \STM{p}{n}}(X_k) \right\|_{\F}^2
        &\leq \frac{2}{c}\left[\scrL_{\beta,\gamma}(X_k, Y_k, \Lambda_k) - \scrL_{\beta,\gamma}(X_{k+1}, Y_{k+1}, \Lambda_{k+1})\right] \\
        &\quad+ \frac{2}{c}\left(\frac{1}{\beta \gamma^2} - \frac{\beta}{2}\right)\| Y_{k+1} - Y_k \|_{\F}^2.
    \end{align*}
    Then combine it with \eqref{smooth.proof.subgrad.bound} to get
    \[
    \begin{aligned}
        \| G_k \|^2_{\F} \leq \; & \frac{2}{c}\left[\scrL_{\beta,\gamma}(X_k, Y_k, \Lambda_k) - \scrL_{\beta,\gamma}(X_{k+1}, Y_{k+1}, \Lambda_{k+1})\right] \\
        & + \left[\frac{2\left(\beta^2\|\Psi\|_2^2 + 1\right)}{\beta^2\gamma^2}\| Y_k - Y_{k-1} \|_{\F}^2 - \frac{2}{c}\left(\frac{\beta}{2} - \frac{1}{\beta \gamma^2}\right)\| Y_{k+1} - Y_k \|_{\F}^2\right] + 2\gamma^2L_{r}^2.
    \end{aligned}
    \]
    With $\beta$ and $\gamma$ as described in the theorem, we have
    \[
        \beta \gamma > 1, \;\;
        \frac{2}{c}\left(\frac{\beta}{2} - \frac{1}{\beta \gamma^2}\right) \geq \frac{2\left(\beta^2\|\Psi\|_2^2 + 1\right)}{\beta^2\gamma^2}.
    \]
    Therefore, we have
    \begin{equation}\label{smooth.proof.subgrad.bound.final}
    \begin{split}
        \| G_k \|^2_{\F} \leq \; & \frac{2}{c}\left[\scrL_{\beta,\gamma}(X_k, Y_k, \Lambda_k) - \scrL_{\beta,\gamma}(X_{k+1}, Y_{k+1}, \Lambda_{k+1})\right] \\
        & + \frac{2}{c}\left(\frac{\beta}{2} - \frac{1}{\beta \gamma^2}\right)\left(\| Y_k - Y_{k-1} \|_{\F}^2 - \| Y_{k+1} - Y_k \|_{\F}^2\right) + 2\gamma^2L_{r}^2.
    \end{split}
    \end{equation}
    Summing the inequality \eqref{smooth.proof.subgrad.bound.final} over $k = 1,...,K$, we get
    \[
    \frac{1}{K}\sum_{k=1}^{K} \| G_k \|_{\F}^2 \leq \frac{2}{cK}\left[\scrL_{\beta,\gamma}(X_1, Y_1, \Lambda_1) - F_\ast + \gamma L_{r}^2\right] + \frac{\beta}{cK}\| Y_1 - Y_{0} \|_{\F}^2 + 2\gamma^2L_{r}^2.
    \]
    Because $\beta = 1 / \epsilon $, $\gamma = \mathcal{O}\left(\sqrt{\epsilon}\right)$, $\min_{k = 1,...,K}\| G_k \|_{\F}^2 \leq \epsilon^2$ is satisfied by $K = \mathcal{O}\left(\epsilon^{-3}\right)$, which means that Algorithm~\ref{alg:smooth.nepvadmm} will generate an $\epsilon$-stationary point to
    problem~\eqref{splitprob.z} in $\mathcal{O}\left(\epsilon^{-3}\right)$ iterations.
\end{proof}

\subsection{Iteration Complexity of NEPvADMM with LOCG}
In this subsection, we analyze the iteration complexity of Algorithm~\ref{locg.smooth.nepvadmm-simplified}
for obtaining an $\epsilon$-stationary point of problem \eqref{splitprob.z}. Throughout this subsection, we will assume
\begin{quote}
The NEPv Ansatz holds for $g_k(X)$ with $H_k(X)$, $\wtd H_k(P) = W_k^{\T}H_k(W_kP)W_k$, $\bbP_k := W_k^{\T}\bbX \subseteq \STM{p}{\hat n}$,
and $W_k\bbP_k\subset\bbX$.
\end{quote}

Following the approach in subsection \ref{Se3.2}, we need to establish the monotonicity of the $X$-step in Algorithm~\ref{locg.smooth.nepvadmm-simplified}. The following lemma shows the sufficient decrease of one SCF iteration for the reduced $X$-subproblem.

\begin{lemma}\label{le4221}
     Suppose that $\widetilde{\cM}_k\left(P_{k,i}\right)$ is symmetric,  and there exist positive constants $\widetilde{\delta}$ and $\widetilde{\lambda}$ such that for any $P \in \STM{p}{\hat n}$ and $k \geq 0$,
\begin{equation}\label{locg.proof.x.suff.cond}
    \lambda_p(\wtd H_k(P)) - \lambda_{p+1}(\wtd H_k(P)) \geq \widetilde{\delta}, \quad
    \vert \lambda_1(\wtd H_k(P)) \vert \leq \widetilde{\lambda} .
\end{equation}
    Then for all $k \geq 0$ and $i \geq 0$, we have
    \begin{equation}\label{locg.proof.x.suff.res.1scf}
        \widetilde{g}_k(P_{k,i+1}) - \widetilde{g}_k(P_{k,i}) \geq \widetilde{c}
           \left\| \GRAD \; \widetilde{g}_{k\vert \STM{p}{\hat n}}(P_{k,i}) \right\|^2_{\F},
    \end{equation}
    where $\widetilde{c} = \omega \widetilde{\delta} / (4\widetilde{\lambda}^2)$, and $\omega > 0$ is a constant from the NEPv Ansatz in Section~\ref{Se2}.
\end{lemma}

\begin{proof}
Similarly to the proof of \eqref{smooth.proof.suff.relat.grad.nepv}, we have the following equation for the reduced problem:
\begin{equation}\label{locg.proof.x.suff.grad.nepv.rela}
    \left\| \GRAD\; \widetilde{g}_{k\vert\STM{p}{m}}(P_{k,i})\right\|_{\F}^2 = \left\| \wtd H_k(P_{k,i})P_{k,i} - P_{k,i}\widetilde{\Phi}_{k,i} \right\|_{\F}^2,
\end{equation}
where $\widetilde{\Phi}_{k,i} = P_{k,i}^{\T}\wtd H_k(P_{k,i})P_{k,i}$. After the partial eigen-decomposition in Algorithm~\ref{locg.smooth.nepvadmm-simplified}, let
$
\eta_{k,i} := \tr( \widehat{P}^{\T}_{k,i}\wtd H_k(P_{k,i})\widehat{P}_{k,i} ) - \tr( P_{k,i}^{\T}\wtd H_k(P_{k,i})P_{k,i} ) \geq 0,
$
and then we have
\begin{equation}\label{locg.proof.x.suff.scf.mono}
     \widetilde{g}_k(P_{k,i+1}) \geq \widetilde{g}_k(P_{k,i}) + \omega\eta_{k,i}
\end{equation}
according to the NEPv Ansatz, where $\omega > 0$ is a constant in the Ansatz. It follows from Lemma \ref{le3.5} that
\begin{equation}\label{locg.proof.x.suff.bound}
    \widetilde{\delta}_{k,i}\frac{\big\| \wtd H_k(P_{k,i})P_{k,i} - P_{k,i}\widetilde{\Phi}_{k,i} \big\|_{\F}^2}
                             {4\big\| \wtd H_k(P_{k,i}) \big\|_2^2}
    \leq \widetilde{\delta}_{k,i}\frac{\big\| \wtd H_k(P_{k,i})P_{k,i} - P_{k,i}\widetilde{\Phi}_{k,i} \big\|_{\F}^2}
                                  {\big[\lambda_1(\wtd H_k(P_{k,i})) - \lambda_m(\wtd H_k(P_{k,i}))\big]^2}
    \leq \eta_{k,i},
\end{equation}
where $\widetilde{\delta}_{k,i} = \lambda_p(\wtd H_k(P_{k,i})) - \lambda_{p+1}(\wtd H_k(P_{k,i}))$. Combining \eqref{locg.proof.x.suff.grad.nepv.rela}, \eqref{locg.proof.x.suff.scf.mono} and \eqref{locg.proof.x.suff.bound} together, we have
$
\widetilde{g}_k(P_{k,i+1}) - \widetilde{g}_k(P_{k,i}) \geq \widetilde{c}_{k,i}
  \left\| \GRAD \; \widetilde{g}_{k\vert\STM{p}{\hat n}}(P_{k,i})\right\|_{\F}^2,
$
where $\widetilde{c}_{k,i} = \omega\widetilde{\delta}_{k,i}/(4\| \wtd H_k(P_{k,i}) \|_2^2)$. Under assumption \eqref{locg.proof.x.suff.cond}, we have $0 < \widetilde{c} = \omega\widetilde{\delta}/(4\widetilde{\lambda}^2) \leq \widetilde{c}_{k,i}$. That completes the proof.
\end{proof}

The following corollary, derived from Lemma \ref{le4221}, establishes the monotonicity in objective value
of the $X$-step over $n_{\rm LOCGscf}$ SCF iterations.

\begin{corollary}\label{cor:suff.decrease.X-sub.LOCG}
Suppose that $\widetilde{\cM}_k\left(P_{k,i}\right)$ is symmetric for all $i \in \{0, \dots, n_{\rm LOCGscf} - 1\}$, and there exist positive constants $\widetilde{\delta}$ and $\widetilde{\lambda}$ such that \eqref{locg.proof.x.suff.cond} holds for any $P \in \STM{p}{\hat n}$. Then, for any $n_{\rm LOCGscf} \geq 1$, it holds that
\begin{equation}\label{locg.proof.x.suff.res}
    g_k(X_{k+1}) - g_k(X_k) \geq \widetilde{c} \left\| \GRAD \; g_{k \vert \STM{p}{n}}(X_k) \right\|^2_{\F},
\end{equation}
where $\widetilde{c} = \omega \widetilde{\delta} / (4\widetilde{\lambda}^2)$.
\end{corollary}

\begin{proof}
We know from Algorithm~\ref{locg.smooth.nepvadmm-simplified} that
\[
\widetilde{g}_k(P_{k,0}) = g_k(X_k) \;\;\;\; {\rm and} \;\;\;\; \widetilde{g}_k(P_{k,n_{\rm LOCGscf}}) = g_k(X_{k+1}).
\]
If $n_{\rm LOCGscf} = 1$, then \eqref{locg.proof.x.suff.res.1scf} implies
$
g_k(X_{k+1}) - g_k(X_k) \geq \widetilde{c} \left\| \GRAD \; \widetilde{g}_{k\vert \STM{p}{m}}(P_k) \right\|^2_{\F}.
$
If $n_{\rm LOCGscf} > 1$, then we have
\[
\begin{aligned}
    g_{k}(X_{k+1}) - g_k(X_k) &= \sum_{i=0}^{n_{\rm LOCGscf}-1} \Big( \widetilde{g}_k(P_{k,i+1}) - \widetilde{g}_k(P_{k,i}) \Big) \\
    &\geq \widetilde{c} \sum_{i=0}^{n_{\rm LOCGscf}-1} \left\| \GRAD \; \widetilde{g}_{k\vert \STM{p}{\hat n}}(P_{k,i}) \right\|^2_{\F} \\
    &\geq \widetilde{c} \left\| \GRAD \; \widetilde{g}_{k\vert \STM{p}{\hat n}}(P_k) \right\|^2_{\F}.
\end{aligned}
\]
The Riemannian gradients $\GRAD \; g_{k\vert\STM{p}{n}}(X_k)$ and $\GRAD \; \widetilde{g}_{k\vert\STM{p}{\hat n}}(P_k)$ are related by
\[
\begin{aligned}
    \left\| \GRAD \; g_{k\vert\STM{p}{n}}(X_k)\right\|_{\F}^2 & = \left\| W_k^{\T}\GRAD \; g_{k\vert\STM{p}{n}}(X_k)\right\|_{\F}^2\\
    & = \left\| W_k^{\T}\left[\frac{\partial g_k(X_k)}{\partial X} - X_k\sym\left( X_k^{\T}\left( \frac{\partial g_k(X_k)}{\partial X} \right) \right)\right]\right\|_{\F}^2 \\
    & = \left\| \frac{\partial \widetilde{g}_k(P_k)}{\partial P} - P_k\sym\left( P_k^{\T}\left( \frac{\partial \widetilde{g}_k(P_k)}{\partial P} \right) \right)\right\|_{\F}^2 \\
    & = \left\| \GRAD \; \widetilde{g}_{k\vert \STM{p}{\hat n}}(P_k)\right\|_{\F}^2.
\end{aligned}
\]
Substituting this equality back into the sufficient decrease inequality yields \eqref{locg.proof.x.suff.res}.
\end{proof}

Once we have Corollary~\ref{cor:suff.decrease.X-sub.LOCG}, the remaining analysis is similar to the proof of Lemma \ref{le5}, Lemma \ref{le6} and Lemma \ref{le7} for Algorithm~\ref{alg:smooth.nepvadmm} in subsection \ref{Se3.2}. We state the convergence result  in the next theorem.

\begin{theorem}\label{th3}
    Suppose that $\widetilde{\cM}_k\left(P_{k,i}\right)$ is symmetric for all $i \in \{0, \dots, n_{\rm LOCGscf} - 1\}$,
    and that there exist positive constants $\widetilde{\delta}$ and $\widetilde{\lambda}$ such that \eqref{locg.proof.x.suff.cond} holds for any $P \in \STM{p}{\hat n}$.
    Let $\left\{\left(X_k, Y_k, Z_k,  \Lambda_k\right)\right\}$ be the sequence generated by Algorithm~\ref{locg.smooth.nepvadmm-simplified}, and $\epsilon > 0$ the predefined tolerance. Set $\beta = 1 / \epsilon$, $\gamma = \sqrt{ 2 / \beta^2 + 2\widetilde{c}\left(\beta^2\|\Psi\|_2^2 + 1\right) / \beta^3}$. Then there exist $G_k \in \partial \scrL\left(X_k, Z_k, \Lambda_k \right), k = 1,2,...$, such that
    \[
    \min_{k = 1,...,K} \left\| G_k \right\|_{\F}^2 \leq \epsilon^2,
    \]
    provided that
    \[
    K = \mathcal{O}\left(\frac{1}{\epsilon^3}\right).
    \]
\end{theorem}

\section{Applications}\label{Se5}
Many applications can be transformed into the form of the optimization problem \eqref{eq:TheProb} within the areas of machine learning, data science, and physics. We briefly mention a few of them in the following.

\begin{example}[Sparse principle component analysis (sPCA)]\label{eg:sPCA}
Principal component analysis is one of the most widely used unsupervised learning methods in machine learning for dimensionality reduction. It was first proposed by Pearson \cite{pearson1901lines} in 1901 and later developed by Hotelling \cite{hotelling1933analysis} in 1933. 
However, for traditional PCA, it is difficult to interpret the corresponding principal components, as each of them is a dense linear combination of all original features \cite{zou2006sparse}.
Nowadays there is huge amount of data that need to be processed and these data are often high-dimensional, and its lack of feature selection capability has limited the performance of traditional PCA in many applications.
To address these issues \cite{wang2023communication}, sparse PCA was proposed, motivated by the notion that most of the variation in the data can be explained through only a small number of significant original features, so that principal components
are now formed by sparse combinations instead.
There are different ways to induce sparsity. In this example, we focus on the following formulation: 
\begin{equation}\label{spca}
\begin{split}
     \min_{X\in\STM{p}{n}} \; -\frac{1}{2}\tr(X^{\T}D^{\T}DX) + \mu |\!|\!| X |\!|\!|_1,
\end{split}
\end{equation}
where $D \in \bbR^{m\times n}$ represents the data matrix and $\mu > 0$ is a weighting parameter.
\end{example}

\begin{example}[Dual principal component pursuit (DPCP)]\label{eg:dpcp}
DPCP \cite{tsakiris2018dual,zhu2018dual} aims to solve the problem of learning a linear subspace from  dataset corrupted by outliers, which is a fundamental issue in statistics and machine learning. The problem can be formulated as follows.
Denote the corrupted data matrix by $B = \left[ \widetilde{I}, \; \wtd O \right]\Gamma \in \bbR^{n \times (p_1 + p_2)}$, where $\widetilde{I} \in \bbR^{n \times p_1}$ consists of the $p_1$ inlier points
that span a $d$-dimensional subspace $\mathcal{S}$ of $\bbR^n$ where $d < p_1$,
$\wtd O \in \bbR^{n \times p_2}$ consists of the $p_2$ outliers points
which have no linear structure, and
$\Gamma \in \bbR^{(p_1 + p_2) \times (p_1 + p_2)}$ is a permutation matrix, but otherwise unknown,
that separates inliers from outliers. The goal is to recover this inlier subspace $\mathcal{S}$.
Instead of trying to fit directly a low-dimensional linear subspace to the entire data, the idea of DPCP is to search for a hyperplane $\mathcal{H}_1 = {\rm span}(x_1)^{\perp}$ that contains all the inliers $\widetilde{I}$, where $x_1$ is the normal vector to $\mathcal{H}_1$. Then we can focus on the data points that lie in $\mathcal{H}_1$.
Tsakiris and Vidal \cite{tsakiris2018dual} in 2018 proved that we may restrict our search to the hyperplane that is maximal with respect to the data matrix $B$, and the optimization problem can be formulated as
\begin{equation}\label{dpcp.onevari}
        \min_{x_1 \in \bbR^n} \;\;  \| B^{\T}x_1 \|_1,
\quad
        {\rm s.t.} \;\;  \left\| x_1 \right\|_2 = 1.
\end{equation}
According to \cite{tsakiris2018dual,zhu2018dual}, under certain conditions, solving problem \eqref{dpcp.onevari} indeed yields a vector that is orthogonal to $\mathcal{S}$ given that $p_2$ is at most on the order of $O\left(p_1^2\right)$.
If the true dimension $d$ is known, we can keep searching for a second hyperplane $\mathcal{H}_2$ which contains $\widetilde{I}$, such that its normal vector $x_2$ is perpendicular to the normal vector $x_1$ of $\mathcal{H}_1$.
Repeat this process $p := n - d$ times and recover $\mathcal{S}$ as the intersection of these $p$ hyperplanes: $\cap_{k = 1}^p\mathcal{H}_k = \mathcal{S}$. The optimization problem becomes
\begin{equation}\label{dpcp}
\begin{split}
    \min_{X \in \STM{p}{n}} \; \; |\!|\!| B^{\T}X |\!|\!|_1.
\end{split}
\end{equation}
\end{example}

\begin{example}[Unsupervised feature selection]\label{eg:ufs}
Feature selection focuses on selecting a subset of features from the high-dimensional dataset for a compact and accurate data representation. 
In contrast to supervised feature selection such as trace ratio and Fisher score, unsupervised feature selection is much more difficult to handle because of its lack of label information.
To address this issue, the authors of \cite{tang2012unsupervised,yang2011norm} assume that there is a linear classifier $W$ that can classify each data point $x_1,\ldots,x_n$ from the training dataset to a class, i.e., $g_i = W^{\T}x_i$.
By incorporating the $\ell_{2,1}$-norm minimization and discriminative analysis, the goal is to train $W$ corresponding to the highest discriminative scores for the training data:
\begin{equation}\label{feature.sele}
\begin{split}
    \min_{W \in \STM{p}{n}} \; \; \tr(W^{\T}MW) + \mu \| W \|_{2,1},
\end{split}
\end{equation}
where $\mu>0$ is a weighting parameter.
We point out that \eqref{feature.sele} shares the same formulation as the $\ell_{2,1}$-norm regularized sPCA
\cite{wawl:2025}, but here $M \in \bbS^{n \times n}$ is calculated by other means from the input data.
We refer the reader to \cite{yang2011norm, tang2012unsupervised} for more details of this problem.
\end{example}

\section{Numerical Experiments}\label{Se6}
The performance of our NEPvADMM 
is demonstrated by several numerical experiments on two different problems. For each experiment, we first address the construction of its $H_k(X)$, as well as the corresponding stopping criterion and the default settings of
the algorithmic parameters.
Next, we compare numerical performance of our algorithms against the existing ones and draw some conclusions. All numerical experiments in this section are conducted in MATLAB R2024b on a MacBook Pro with Apple M3 Pro chip, 11-Core CPU, 14-Core GPU, and 18GB unified memory.

\subsection{Numerical Results on sPCA}\label{subsecspca}
In this subsection, we will show the efficiency of NEPvADMM
in solving sPCA through comprehensive numerical experiments. In the general formulation \eqref{eq:TheProb} from
section~\ref{Se1}, we have
\begin{equation}\label{eq:sPCA:egs}
f(X) = -\frac{1}{2}\tr(X^{\T}MX), \;\;\; r(X) = \mu |\!|\!| X |\!|\!|_1, \;\;\; \Psi = I_n,
\end{equation}
where $M \in \bbS^{n \times n}$.
For solving the $X$-subproblem by the NEPv approach, we need to construct $H_k(X) \in \bbS^{n \times n}$ and check whether the NEPv Ansatz holds.
Here we have
\[
    g_k(X) = \frac{1}{2}\tr(X^{\T}MX) - r(Y_k) - \left\langle \Lambda_k, X - Y_k \right\rangle - \frac{\beta}{2}\| X - Y_k \|_{\F}^2,
\]
and the partial derivative of $g_k$ with respect to $X$ is
\[
\frac{\partial g_k(X)}{\partial X} = (M - \beta I_n + D_kX^{\T})X,
\]
where $D_k := \beta Y_k - \Lambda_k$. Using \eqref{cons.H.firstcase.H} to deal with the non-symmetric term $D_kX^{\T}$, we have the following NEPv
\[
H_k(X)X = X\Omega_k, \; X \in \STM{p}{n}, \; \Omega_k \in \bbS^{p \times p},
\]
where
\begin{equation}\label{h_for_spca}
    H_k(X) := M - \beta I_n + D_kX^{\T} + XD_k^{\T},
\end{equation}
implying $\scrM_k(X)=D_k^{\T}X$ in \eqref{equa.nepv.parde}.
The following lemma verifies the NEPv Ansatz.

\begin{lemma}\label{le8}
Consider the $X$-subproblem \eqref{eq:MRADMM-frame:X} on sPCA \eqref{eq:sPCA:egs}.
Given $X \in \bbX_k:= \STM{p}{n}_{D_k+}$, for any $\what X \in \STM{p}{n}$ that satisfies \eqref{ansatz.traceineq} with
$H_k(X)$ by \eqref{h_for_spca}, we have
    \[
    g_k(\wtd X) \geq g_k\left(X\right) + \frac{\eta}{2},
    \]
    where $Q = UV^{\T}$, $\wtd X = \what XQ$ with $U$ and $V$ from the SVD $\what X^{\T}D_k = U \Sigma V^{\T}$.
\end{lemma}

\begin{proof}
    Suppose \eqref{ansatz.traceineq} holds for $X \in \bbX_k$ and $\what X \in \STM{p}{n}$, i.e.,
    \begin{equation}\label{spca.smooth.ansatz.cond}
        \tr(\what X^{\T}H_k(X)\what X) \geq \tr(X^{\T}H_k(X)X) + \eta.
    \end{equation}
    The left-hand side of \eqref{spca.smooth.ansatz.cond} becomes
    \begin{equation}\label{spca.smooth.ansatz.cond.left}
        \tr(\what X^{\T}H_k(X)\what X) = \tr(\what X^{\T}M\what X)
        - \beta p + 2\tr(\what X^{\T}D_kX^{\T}\what X).
    \end{equation}
    By von Neumann's trace inequality \cite{von1937some}, the last term in \eqref{spca.smooth.ansatz.cond.left} can be bounded as
    \begin{equation}\label{spca.smooth.ansatz.cond.left.ineq}
    \begin{split}
    \tr(\what X^{\T}D_kX^{\T}\what X)
            & \leq \sum_{i = 1}^p\sigma_i(\what X^{\T}D_k)\sigma_i(X^{\T}\what X) \\
            & \leq \sum_{i = 1}^p\sigma_i(\what X^{\T}D_k)  = \| \what X^{\T}D_k \|_{\tr} = \tr(\wtd X^{\T}D_k), 
    \end{split}
    \end{equation}
    where $\wtd X = \what X(UV^{\T}) \in \bbX_k$, and $U, V$ are from SVD $\what X^{\T}D_k = U \Sigma V^{\T}$. To obtain the last equality in \eqref{spca.smooth.ansatz.cond.left.ineq}, we used that
    $
    \wtd X^{\T}D_k = VU^{\T}\what X^{\T}D_k = V\Sigma V^{\T}.
    $
    Combining \eqref{spca.smooth.ansatz.cond.left.ineq} and \eqref{spca.smooth.ansatz.cond.left}, we have
    \begin{equation}\label{spca.smooth.ansatz.cond.left.ineq.result}
    \begin{split}
        \tr(\what X^{\T}H_k(X)\what X) \leq \; & \tr(\wtd X^{\T}M\wtd X) - \beta p + 2\tr(\wtd X^{\T}D_k).
    \end{split}
    \end{equation}
   Adding $2\tr(Y_k^{\T}\Lambda_k) - \beta \tr(Y_k^{\T}Y_k) - 2 r(Y_k)$ to both sides of \eqref{spca.smooth.ansatz.cond.left.ineq.result}, we have
    \begin{equation}\label{spca.smooth.ansatz.cond.left.gk}
    \begin{split}
    \tr(\what X^{\T}H_k(X)\what X) + 2\tr(Y_k^{\T}\Lambda_k) - \beta \tr(Y_k^{\T}Y_k) - 2 r(Y_k) \leq 2g_k(\wtd X).
    \end{split}
    \end{equation}
    The right-hand side of \eqref{spca.smooth.ansatz.cond} is
    \begin{equation}\label{spca.smooth.ansatz.cond.right}
        \tr(X^{\T}H_k(X)X) + \eta = \tr(X^{\T}MX) - \beta p + 2\tr(X^{\T}D_k) + \eta.
    \end{equation}
    Similarly, we have
    \begin{equation}\label{spca.smooth.ansatz.cond.right.gk}
    \tr(X^{\T}H_k(X)X) + 2\tr(Y_k^{\T}\Lambda_k) - \beta \tr(Y_k^{\T}Y_k) - 2 r(Y_k) + \eta = 2g_k(X) + \eta.
    \end{equation}
    Combining \eqref{spca.smooth.ansatz.cond.left.gk}, \eqref{spca.smooth.ansatz.cond.right.gk} and \eqref{spca.smooth.ansatz.cond}, we get
    $
    \begin{aligned}
        g_k(\wtd X) \geq \; g_k(X) + \eta / 2,
    \end{aligned}
    $
    which means that the NEPv Ansatz on page~\pageref{NEPvAnsatz} holds with $\omega = 1/2$. In passing we note that the NEPv Ansatz can also be verified by using the idea of atomic function and applying the results from \cite{li2024theory} directly, but we provide a detailed proof for self-containedness.
\end{proof}

We now report numerical results under varying $\mu \in [0.1, 5]$ on sPCA solved by our proposed Algorithm~\ref{alg:smooth.nepvadmm} ($n_{\scf} = 1$), where the matrix $M$ is randomly generated. Specifically, for a given $n$, $M$ is synthesized in MATLAB as
\[
\begin{aligned}
    & \texttt{H = randn(n,n); H = (H+H')./2; [U,$\thicksim$] = eig(H);} \\
    & \texttt{v = rand(n,1) + 1e-6; M = U*diag(v)*U';}
\end{aligned}
\]
We employ the relative change in the objective function as the stopping criterion:
\begin{equation}\label{stopping.criterion.spca}
\frac{\left|F(X_{k+1}) - F(X_k)\right|}{\left|F(X_k)\right|} < \tol.
\end{equation}
Unless otherwise specified, we set $\tol = 10^{-5}$ and limit our method to a maximum of $10^3$ iterations.
We implement the partial eigen-decomposition step in our NEPvADMM-based algorithms using MATLAB’s {\tt eig}.
Figure~\ref{fig 0.5} plots
sparsity calculated as the percentage of entries with magnitude smaller than $10^{-5}$ at convergence,
constraint violation $\left\| X_k - Z_k \right\|_{\F}$,
objective function value,
and the number of iterations
on $4$ synthetic examples as $\mu$ varies.
It is observed that, as $\mu$ increases, sparsity first increases and then maintains to a certain level,
constraint violation becomes worse, and the rate at which the objective function value increases and the rate at which the number of iterations decreases are both slowing down.

\begin{figure}[htbp]
     \centering
     \subfloat{
         \includegraphics[width=0.4\textwidth]{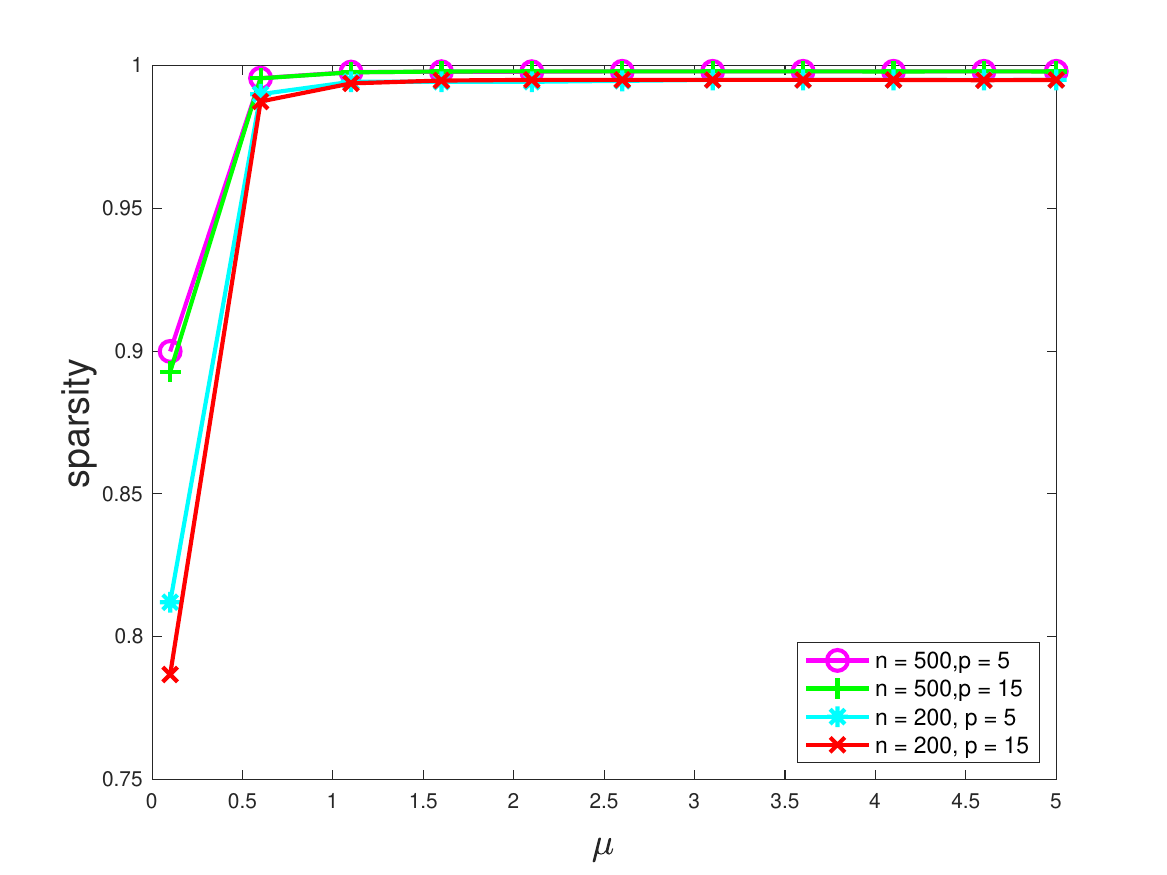}
     }
     \hfill
     \subfloat{
         \includegraphics[width=0.4\textwidth]{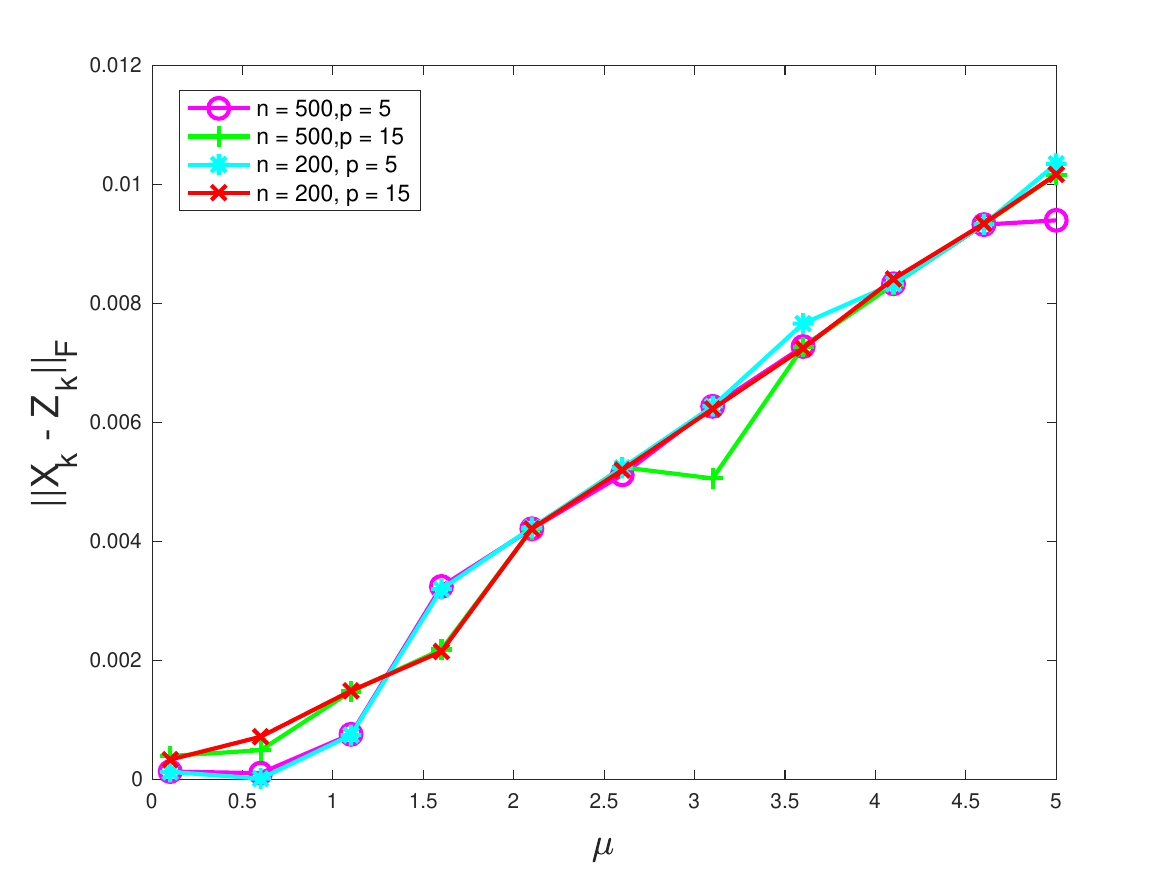}
     }
     
     \subfloat{
         \includegraphics[width=0.4\textwidth]{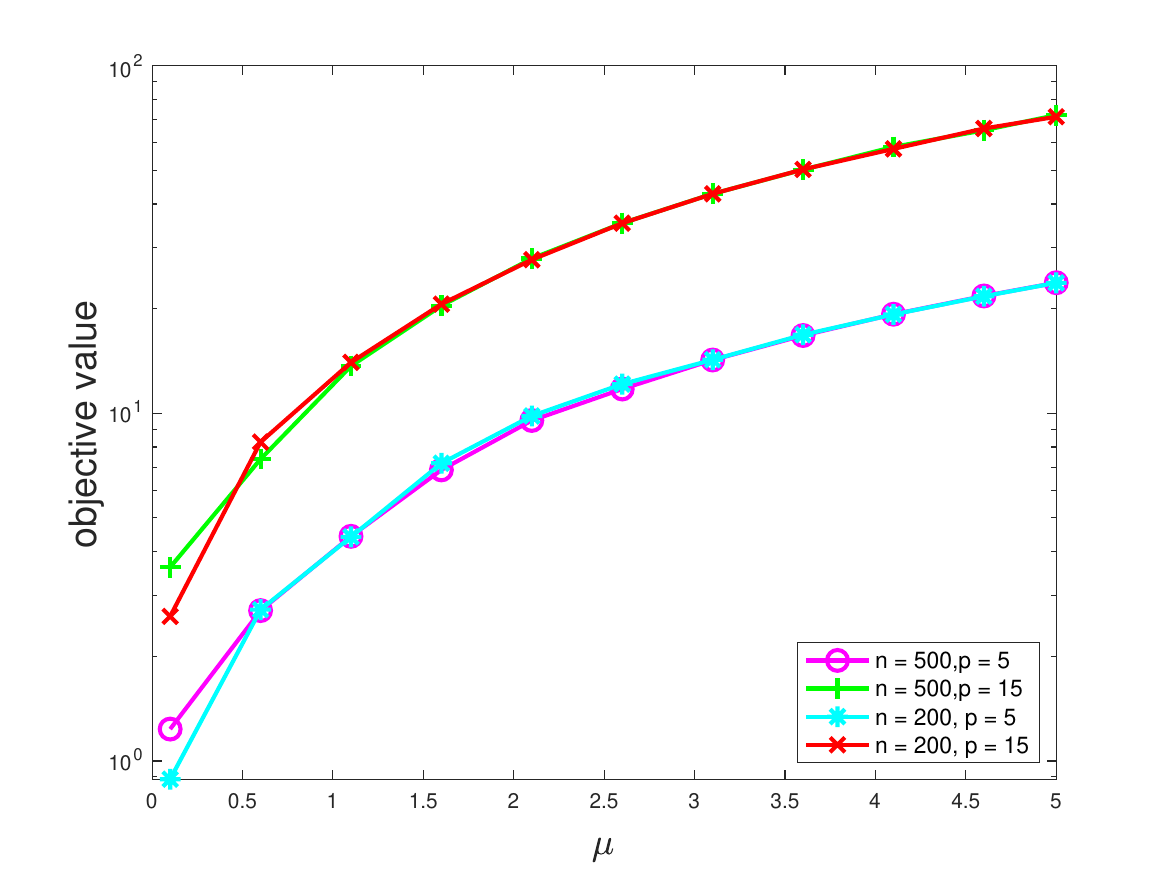}
     }
     \hfill
     \subfloat{
         \includegraphics[width=0.4\textwidth]{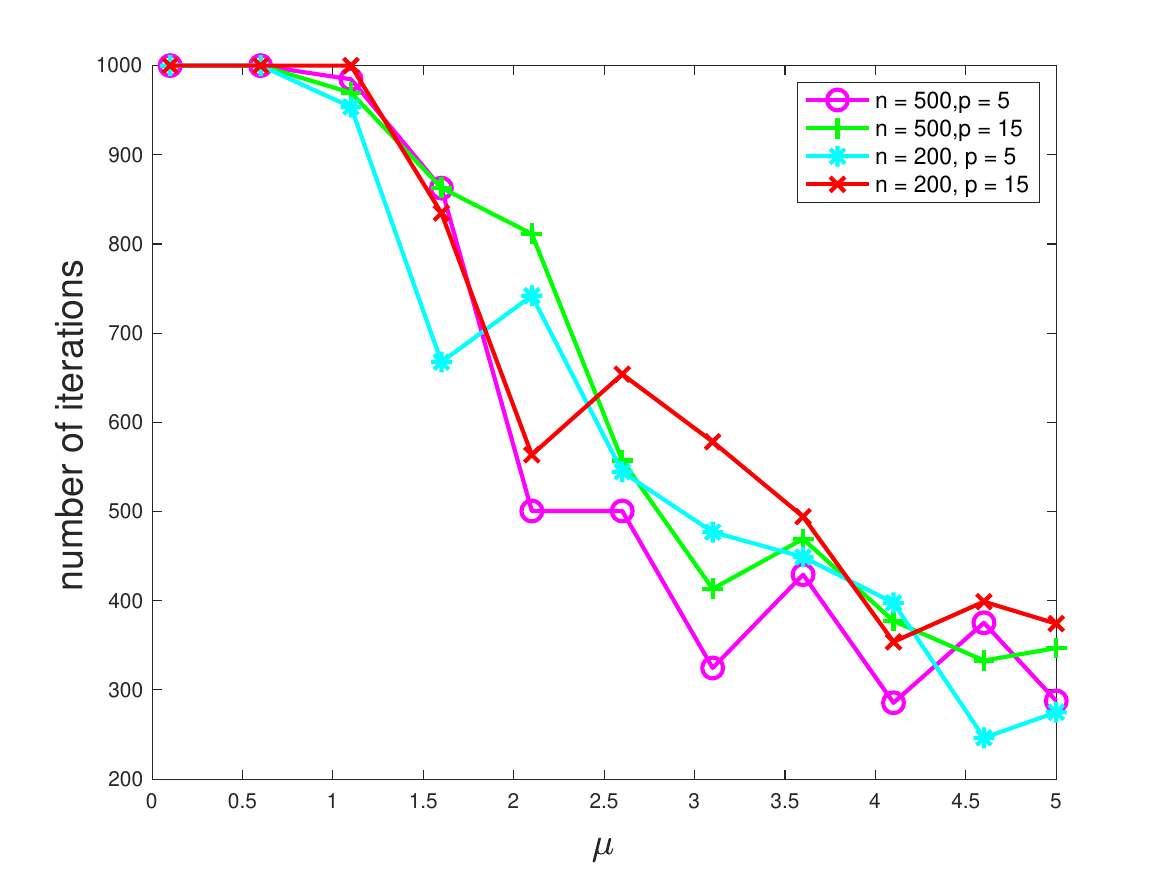}
     }
\caption{Solving sPCA~\eqref{spca} by NEPvADMM (Algorithm~\ref{alg:smooth.nepvadmm} with $n_{\scf} = 1$) on four synthetic examples and $\gamma=10^{-8}$,
as $\mu$ varies in $[0.1,5]$: sparsity, $\| X_k - Z_k \|_{\F}$, objective function value, and the number of iterations.}
        \label{fig 0.5}
\end{figure}

Next we compare the performance of our NEPvADMM (Algorithm~\ref{alg:smooth.nepvadmm} with $n_{\scf} = 1$) with five existing ADMM-type approaches: Riemannian subgradient method (RSG) \cite{ferreira1998subgradient}, SOC \cite{lai2013compressed}, ManPG \cite{chen2020proximal}, MADMM \cite{kovnatsky2016madmm} and RADMM \cite{li2024riemannian}.
For our Algorithm~\ref{alg:smooth.nepvadmm}, we set $\beta = 15$ and $\gamma = 10^{-8}$.
The MATLAB codes for RADMM and MADMM are downloaded from github\footnote{https://github.com/JasonJiaxiangLi/RADMM}, and the codes of ManPG, SOC, and RSG are also downloaded from github\footnote{https://github.com/chenshixiang/ManPG}, too.
We adopt the same stopping criterion as well as the default settings of the parameters used in \cite{chen2020proximal} for SOC, ManPG, and RSG, and in \cite{li2024riemannian} for MADMM and RADMM.

\begin{figure}[htbp]
     \centering
     \subfloat[Objective value\label{1a}]{\includegraphics[width=0.3\textwidth]{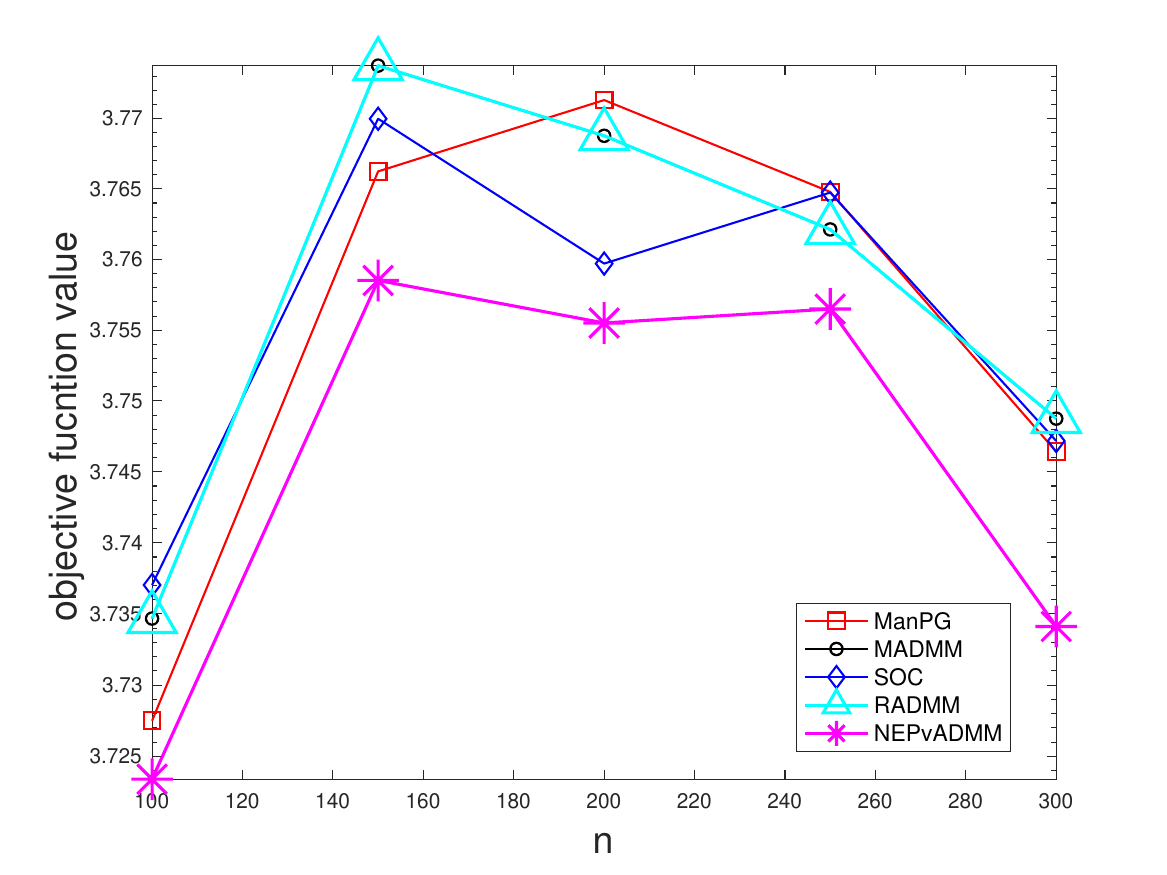}}\hfill
     \subfloat[Objective value\label{1b}]{\includegraphics[width=0.3\textwidth]{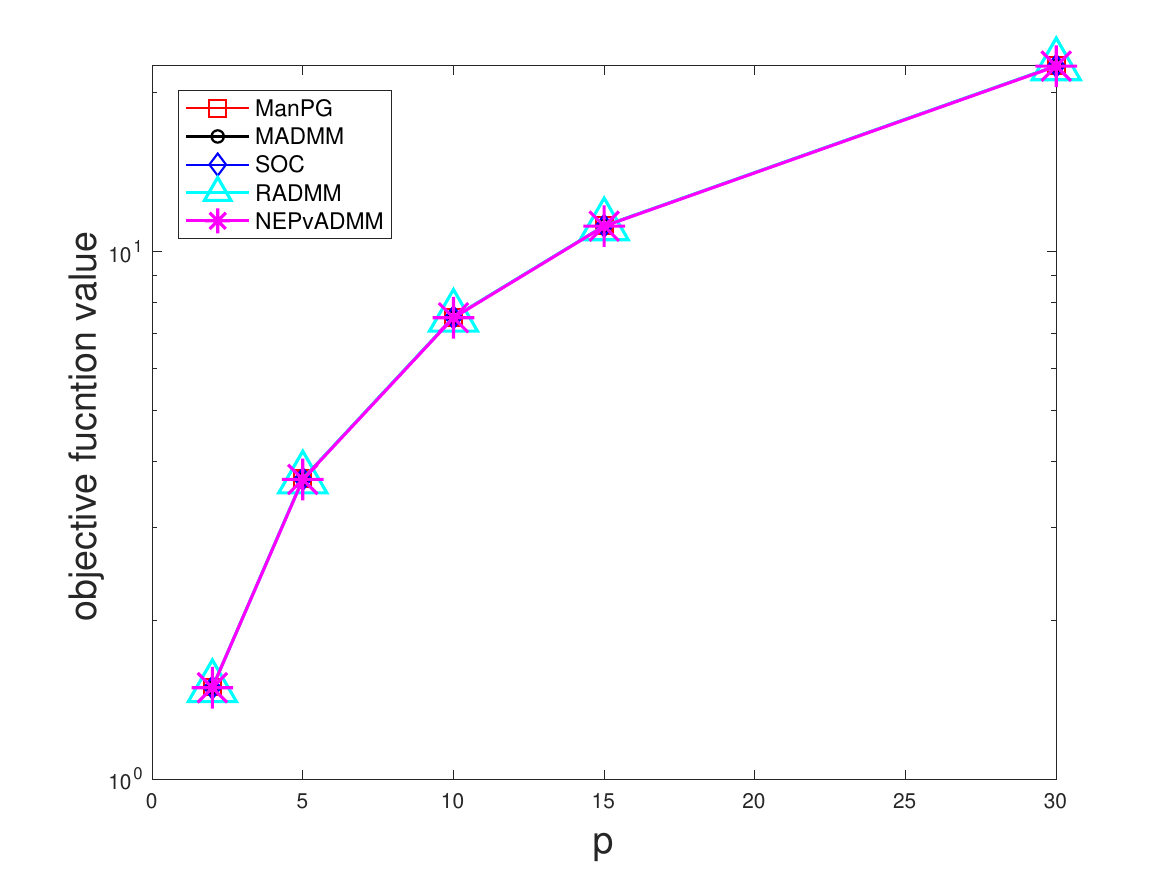}}\hfill
     \subfloat[Objective value\label{1c}]{\includegraphics[width=0.3\textwidth]{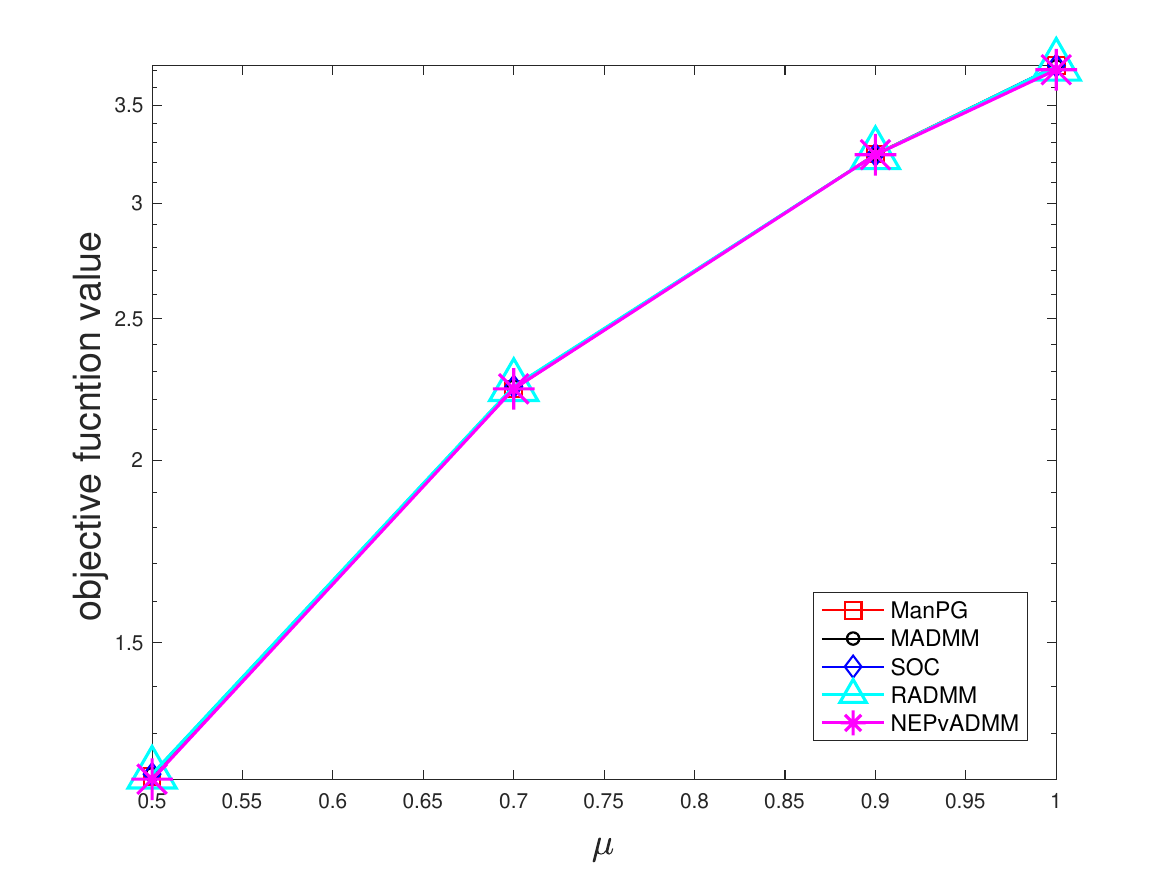}}
     
     \subfloat[Sparsity\label{1d}]{\includegraphics[width=0.3\textwidth]{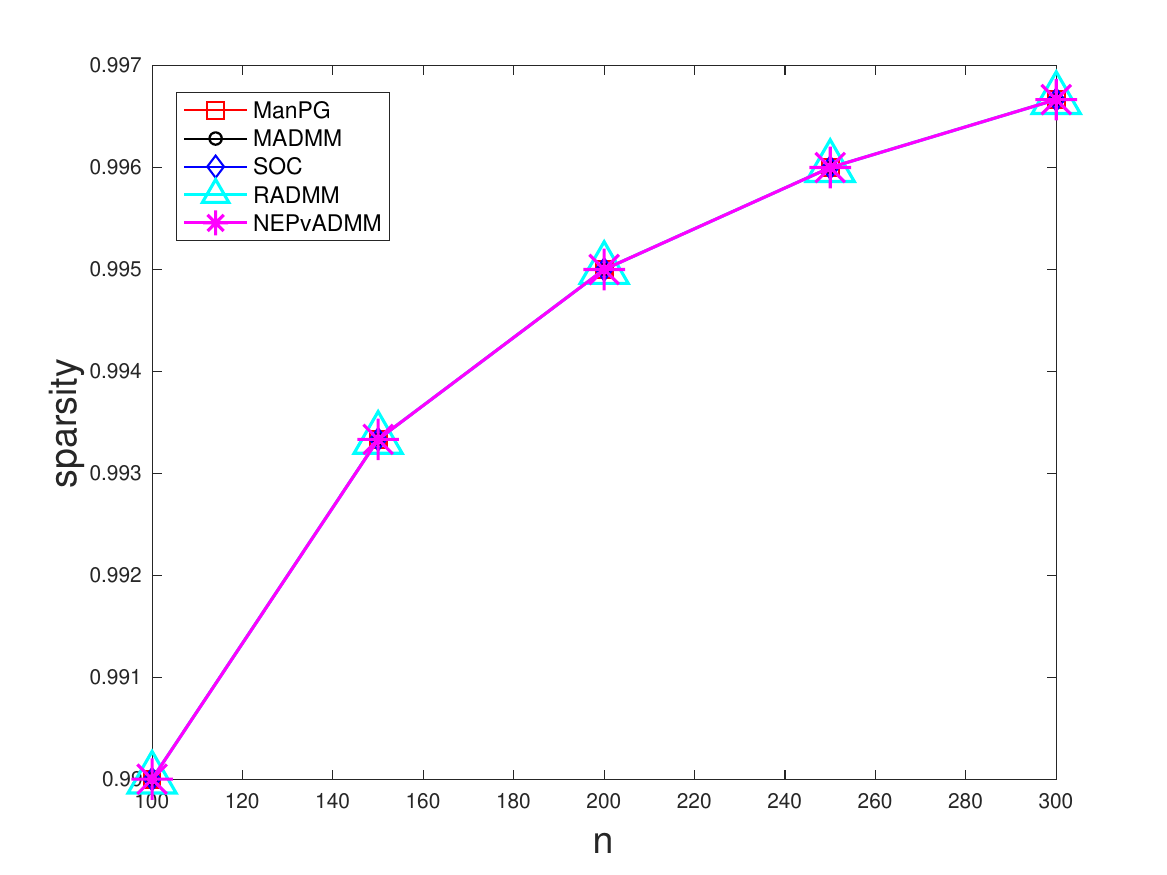}}\hfill
     \subfloat[Sparsity\label{1e}]{\includegraphics[width=0.3\textwidth]{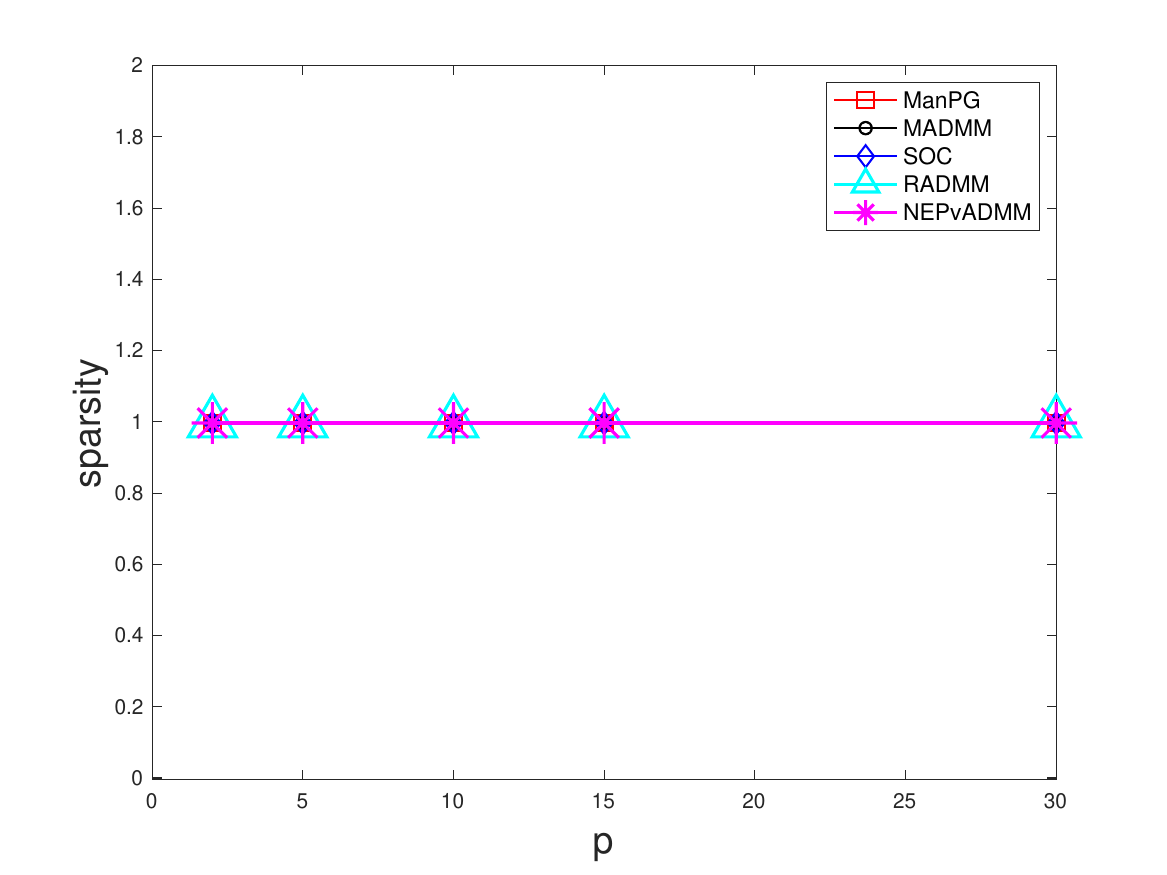}}\hfill
     \subfloat[Sparsity\label{1f}]{\includegraphics[width=0.3\textwidth]{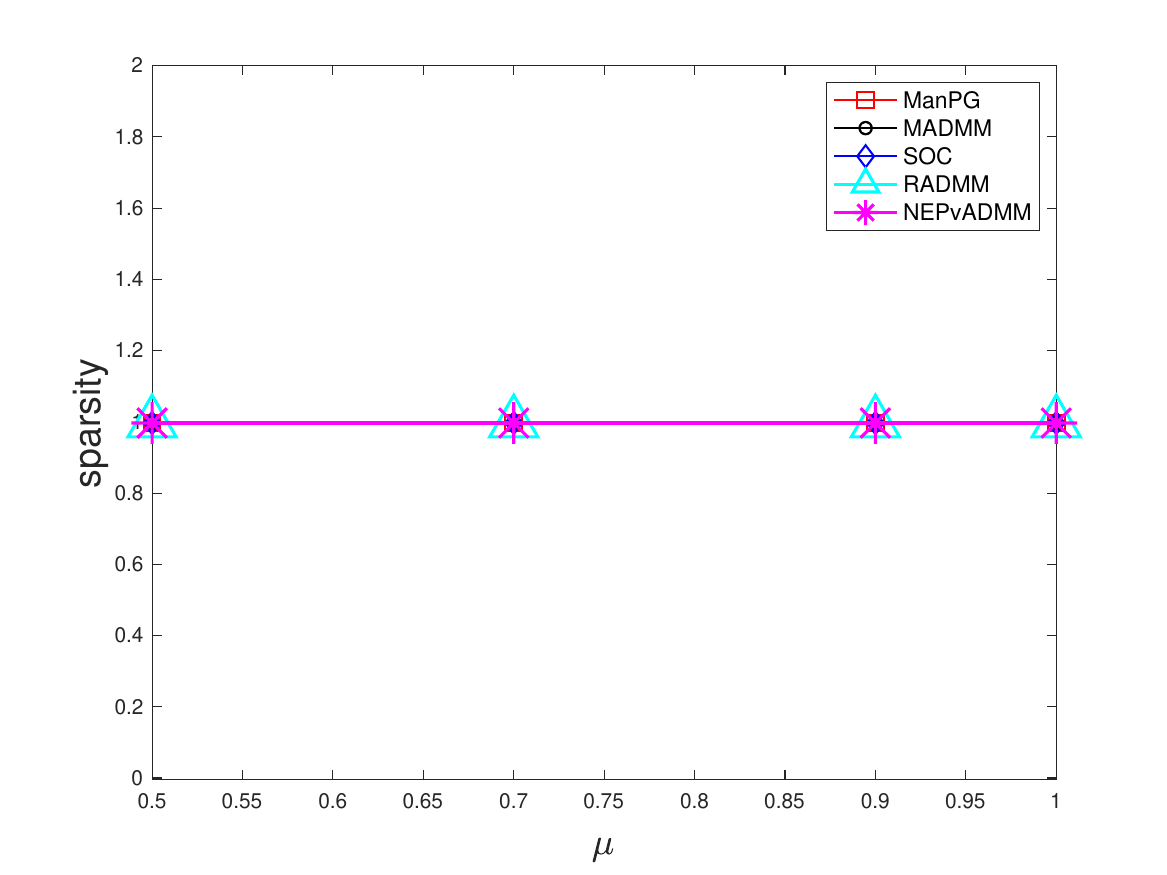}}
     
     \subfloat[Iteration count\label{1g}]{\includegraphics[width=0.3\textwidth]{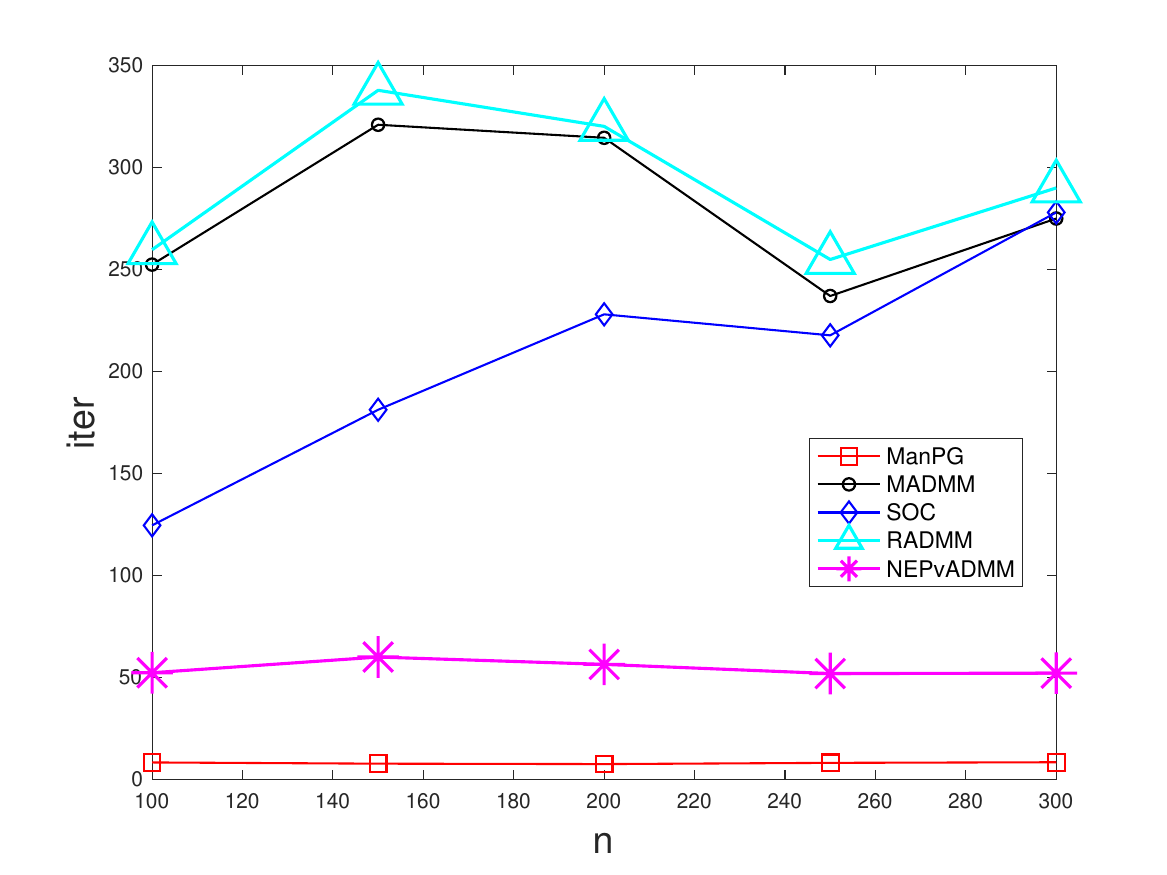}}\hfill
     \subfloat[Iteration count\label{1h}]{\includegraphics[width=0.3\textwidth]{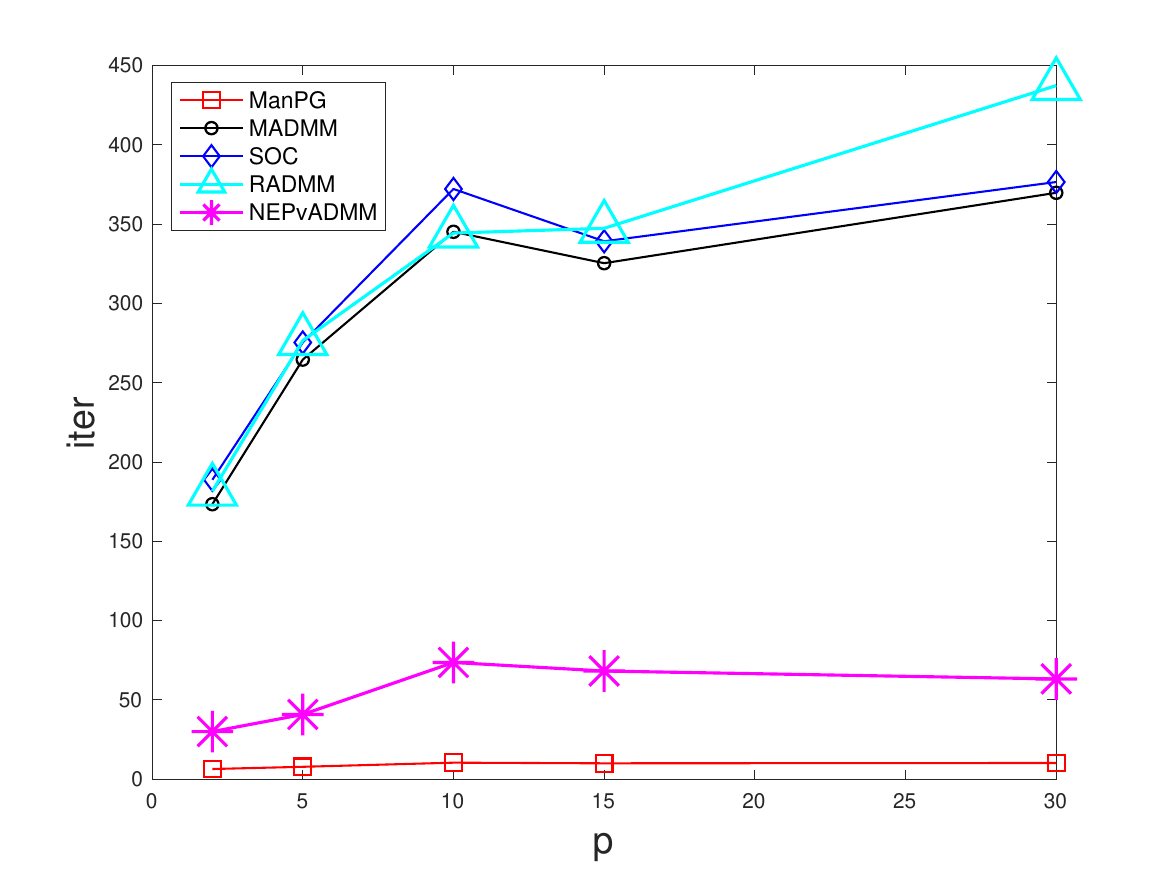}}\hfill
     \subfloat[Iteration count\label{1i}]{\includegraphics[width=0.3\textwidth]{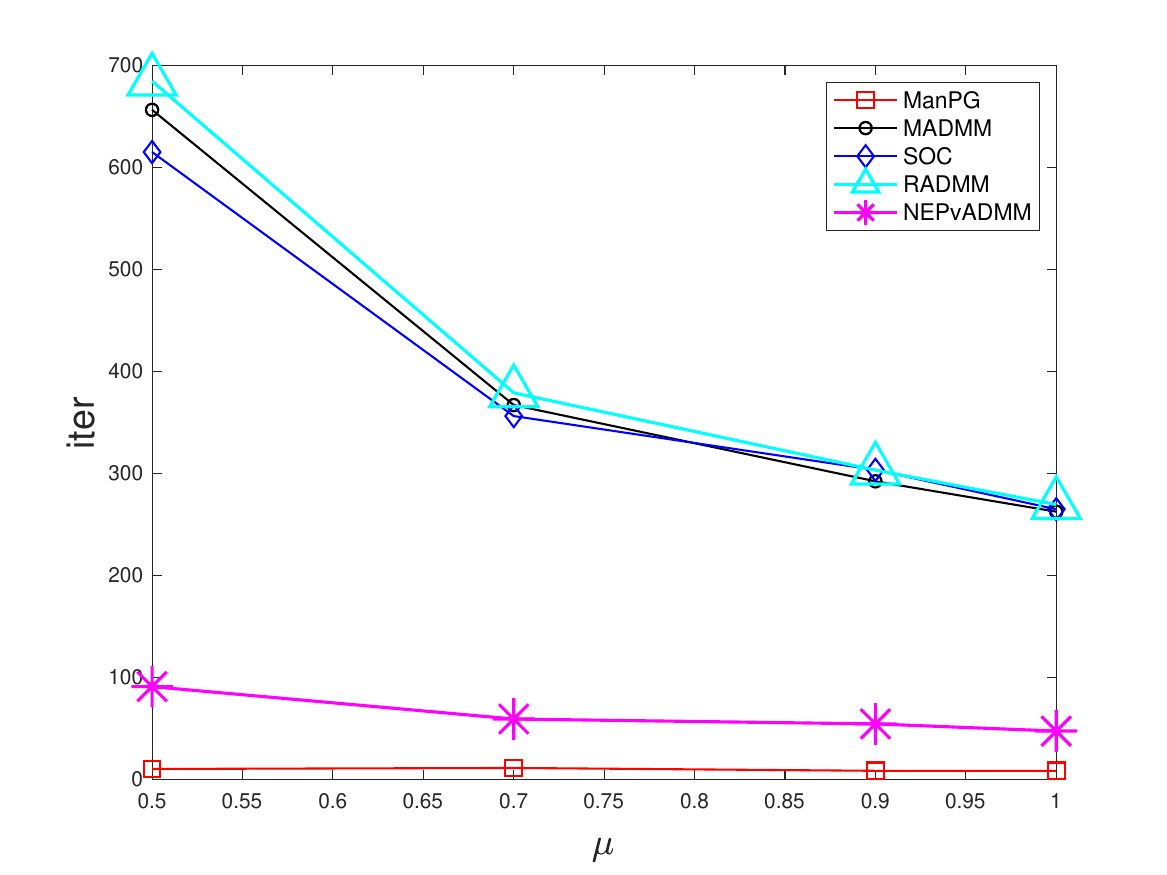}}
     
     \subfloat[CPU time\label{1j}]{\includegraphics[width=0.3\textwidth]{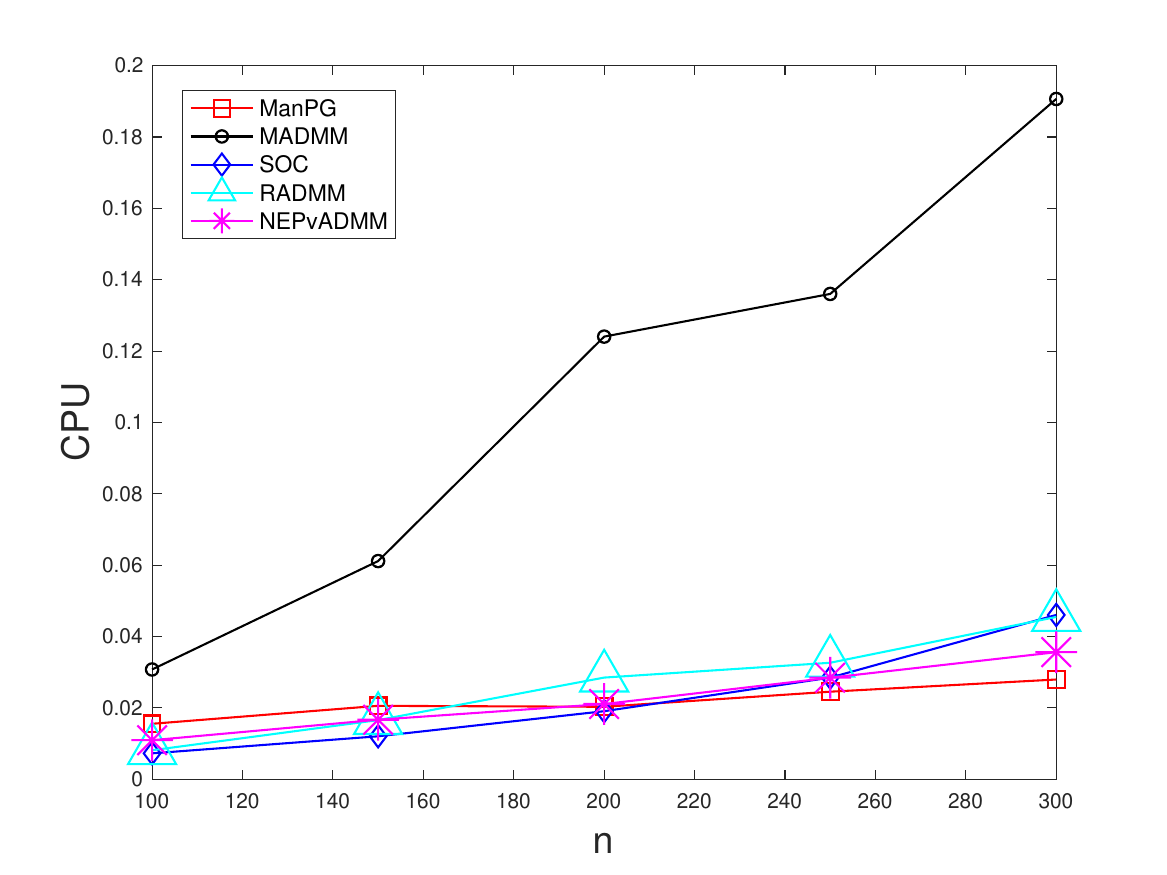}}\hfill
     \subfloat[CPU time\label{1k}]{\includegraphics[width=0.3\textwidth]{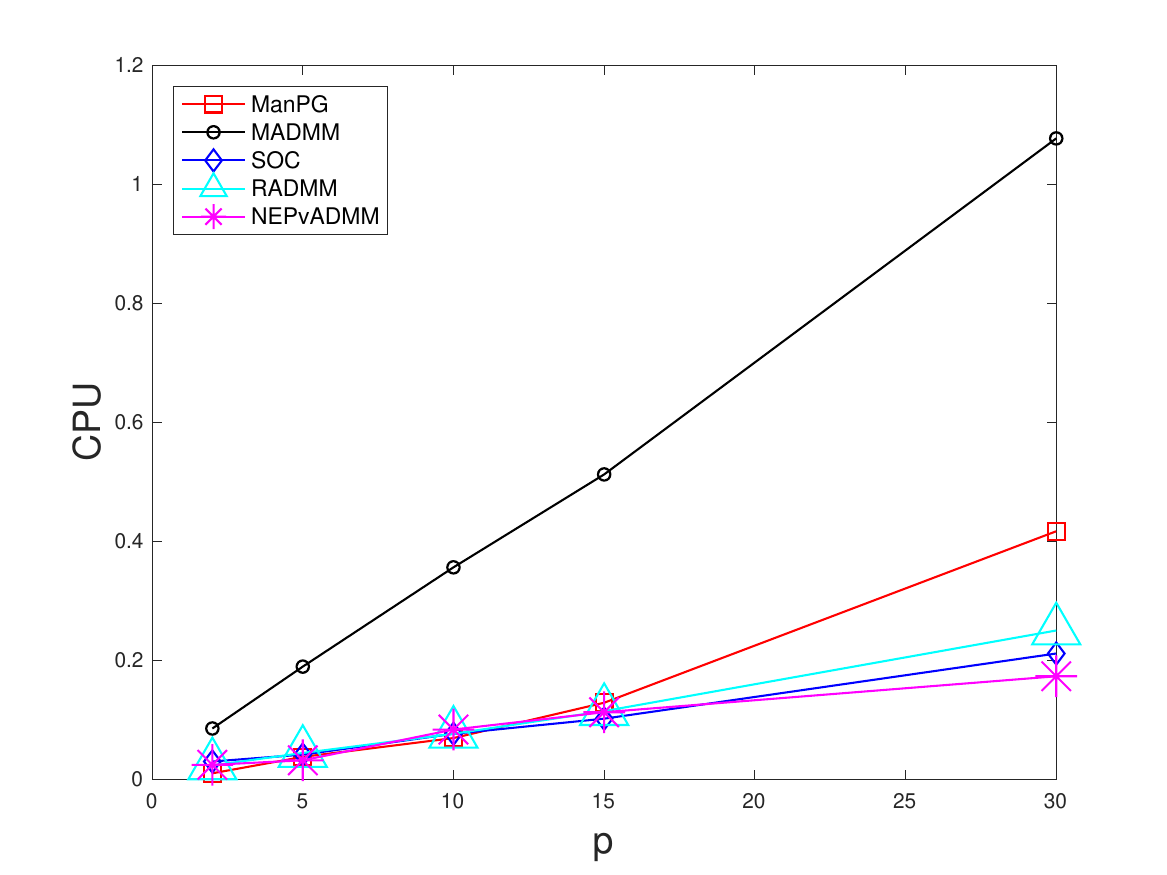}}\hfill
     \subfloat[CPU time\label{1l}]{\includegraphics[width=0.3\textwidth]{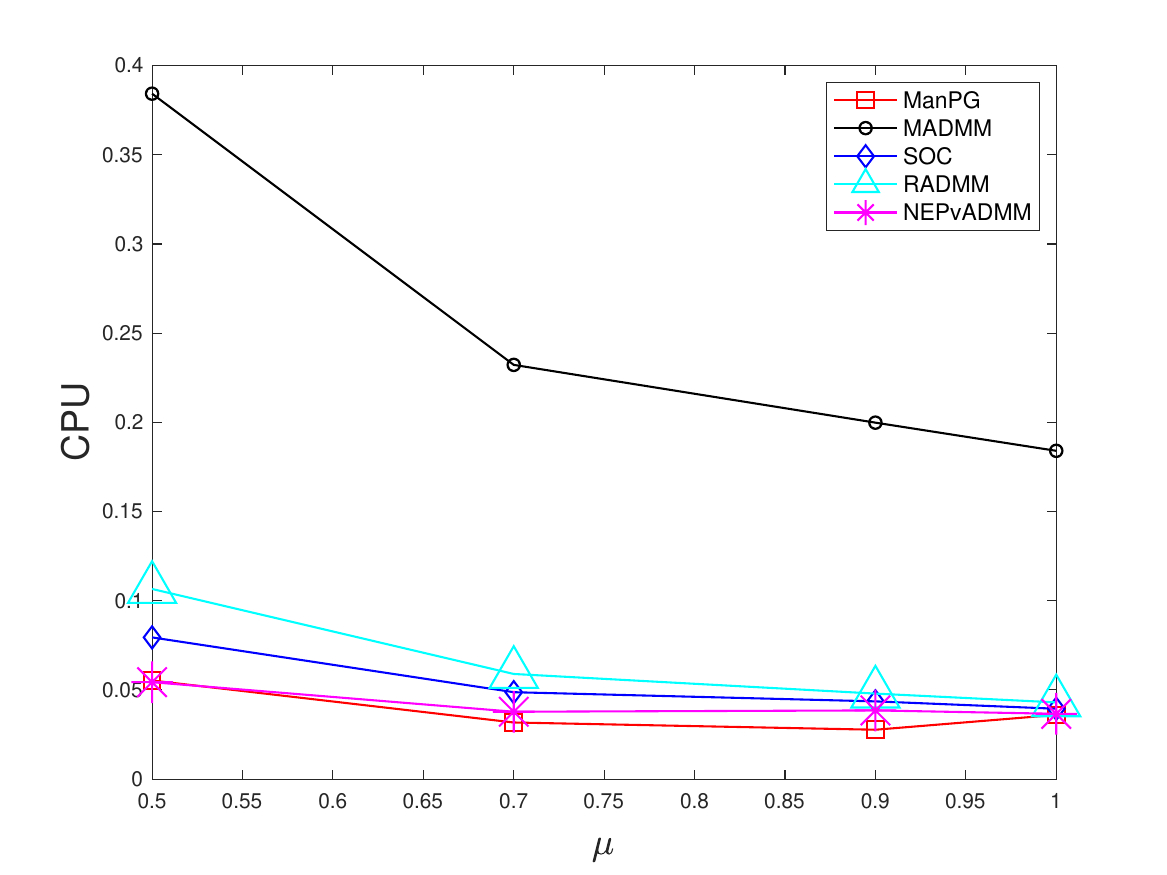}}
     
    \caption{Solving sPCA~\eqref{spca} by NEPvADMM (Algorithm~\ref{alg:smooth.nepvadmm} with $n_{\scf} = 1$) and the five existing algorithms under different parameter settings. 1st column: $p = 5$, $\mu = 1$ with varying $n$; 2nd column: $n = 300$, $\mu = 1$ with varying $p$; 3rd column: $n = 300$, $p = 5$ with varying $\mu$.}
    \label{fig 1}
\end{figure}

First, we compare the performance of the six algorithms mentioned above under different parameter combinations $(n,p,\mu)$. For each parameter combination, we randomly generate $50$ testing problems with all algorithms starting from the same randomly generated initial points, and record the average value of the corresponding objective function values at the last iterate, sparsity, CPU time, and the number of iterations.
With parameters $p = 5$, $\mu = 1$ and varying $n$ among $\{100, 150, 200, 250, 300\}$, we generate sub-Figures~\ref{1a}, \ref{1d}, \ref{1g}, and \ref{1j}. With $n = 300$, $\mu = 1$ and $p$ varying among $\{2, 5, 10, 15, 30\}$, we plot the results in sub-Figures~\ref{1b}, \ref{1e}, \ref{1h}, and \ref{1k}. Finally, we generate sub-Figures~\ref{1c}, \ref{1f}, \ref{1i}, and \ref{1l} for $n = 300$, $p = 5$ and  $\mu$ varying among $\{0.5, 0.7, 0.9, 1\}$. We have the following observations:
\begin{enumerate}[(1)]
  \item The objective function values generated by these algorithms are comparable as can be seen from sub-Figures~\ref{1a}, \ref{1b}, and \ref{1c}, suggesting all algorithms generate optimizers of comparable quality.
  \item The sparsity obtained by these algorithms is nearly the same in most cases.
  \item Our approach takes fewer iterations to converge compared to MADMM, SOC, and RADMM. The CPU time by Algorithm~\ref{alg:smooth.nepvadmm} is among the lowest.
\end{enumerate}
We point that $n$ is kept modest in this experiment. For large $n$ and small $p$, using {\tt eig} will be too wasteful.
The results of RSG are not included in the plots due to its poor performance and that it fails to satisfy the corresponding stopping criterion.

\begin{figure}[htbp]
     \centering
     \subfloat[\footnotesize{$n$ = 1000, $p$ = 20}\label{2a}]{\includegraphics[width=0.32\textwidth]{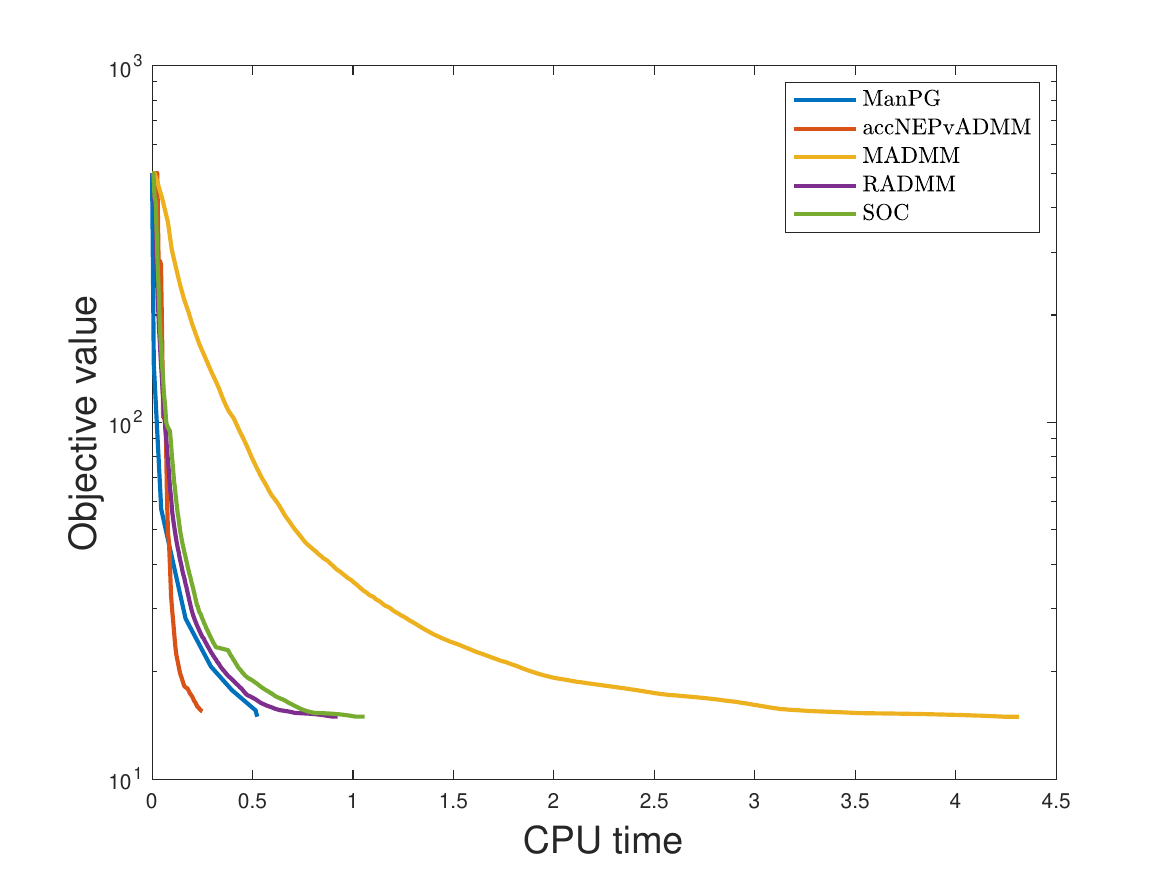}}\hfill
     \subfloat[\footnotesize{$n$ = 1500, $p$ = 20}\label{2c}]{\includegraphics[width=0.32\textwidth]{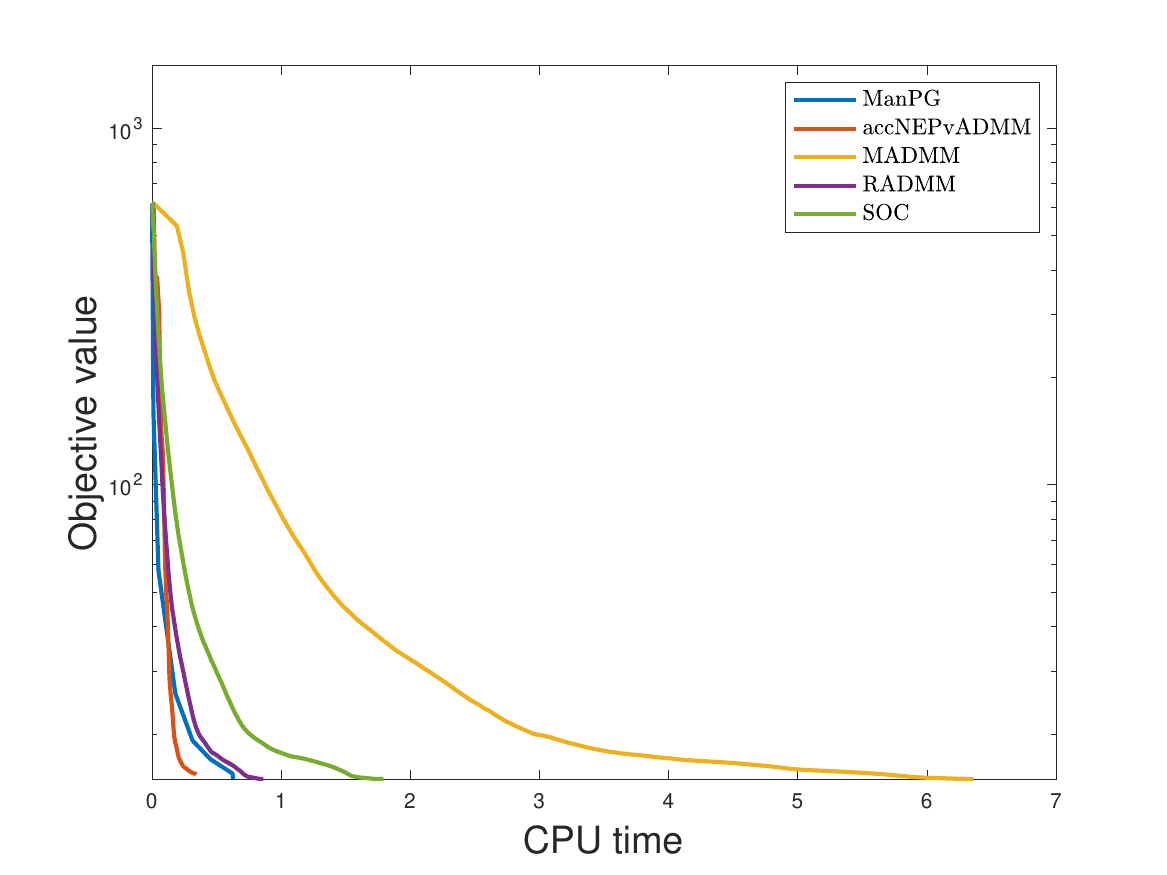}}\hfill
     \subfloat[\footnotesize{$n$ = 2000, $p$ = 20}\label{2e}]{\includegraphics[width=0.32\textwidth]{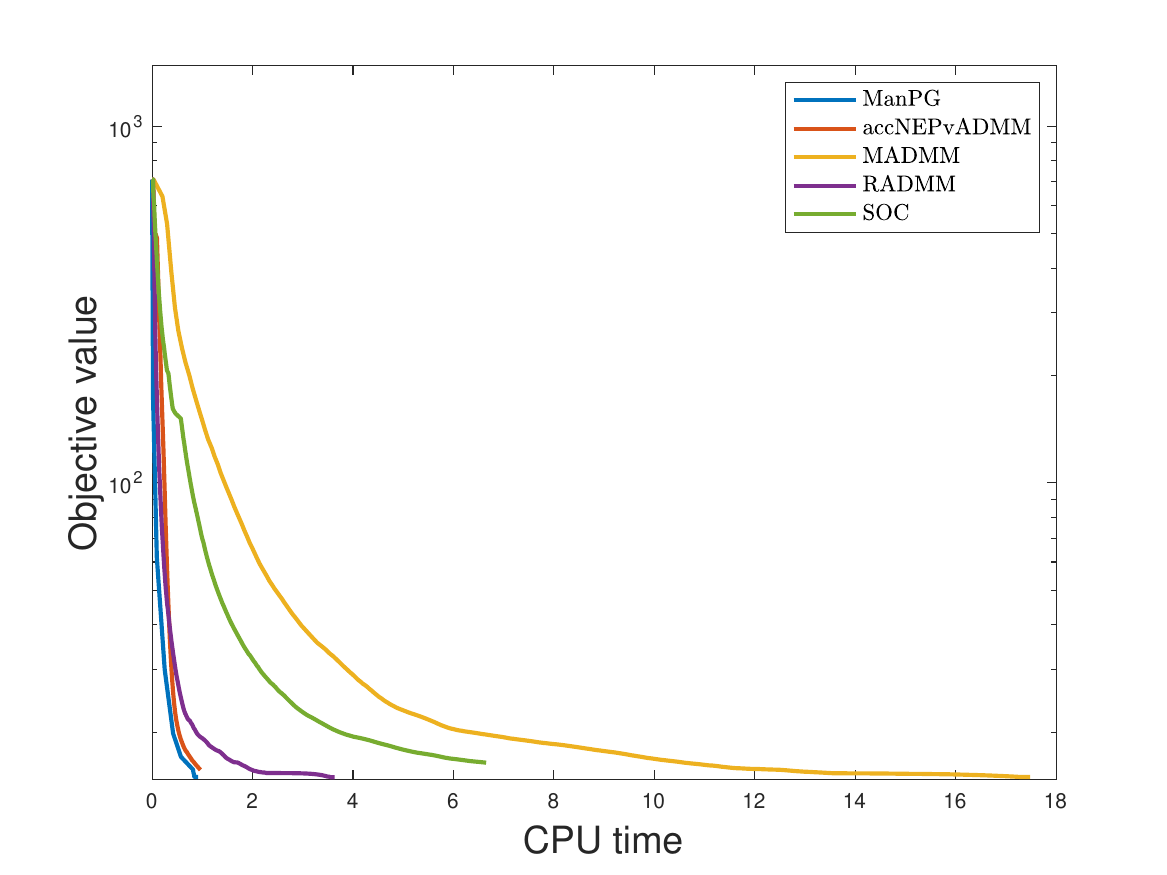}}
     
     \subfloat[\footnotesize{$n$ = 1000, $p$ = 20}]{\includegraphics[width=0.32\textwidth]{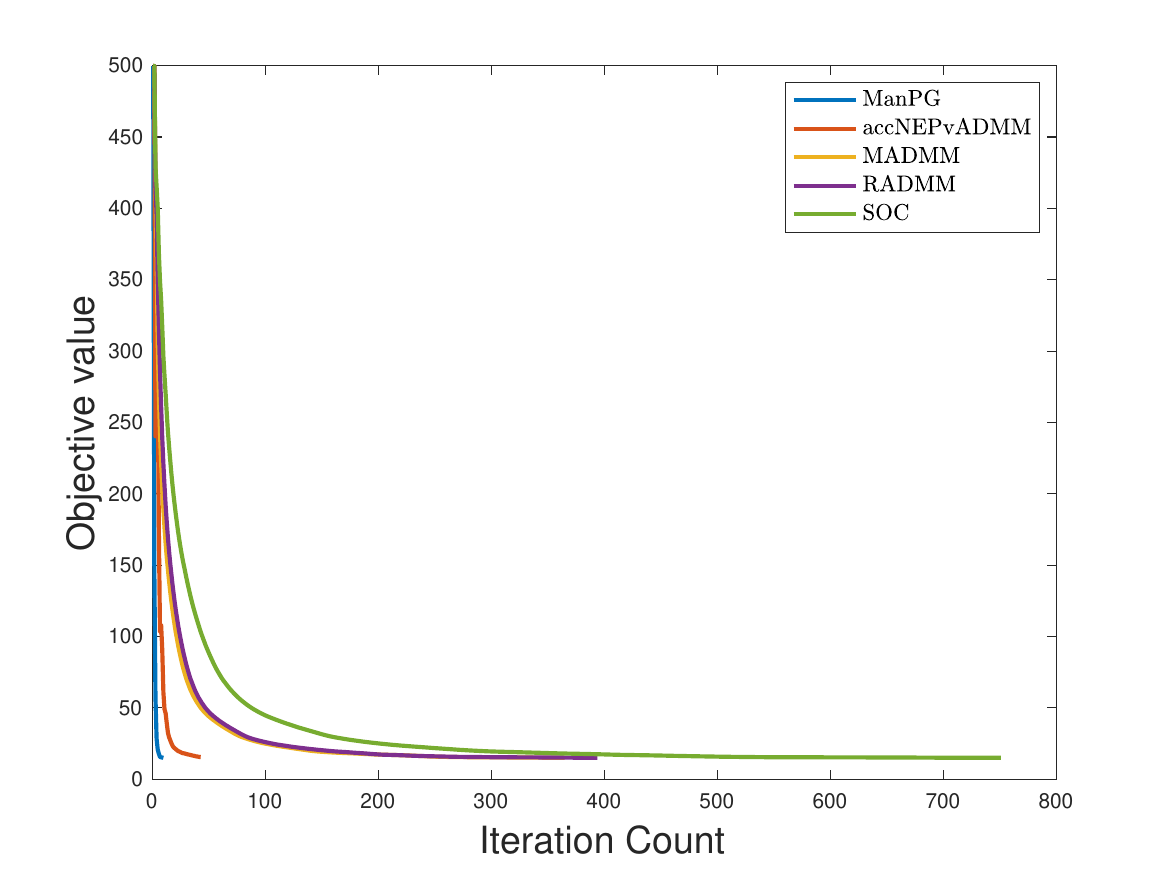}}\hfill
     \subfloat[\footnotesize{$n$ = 1500, $p$ = 20}]{\includegraphics[width=0.32\textwidth]{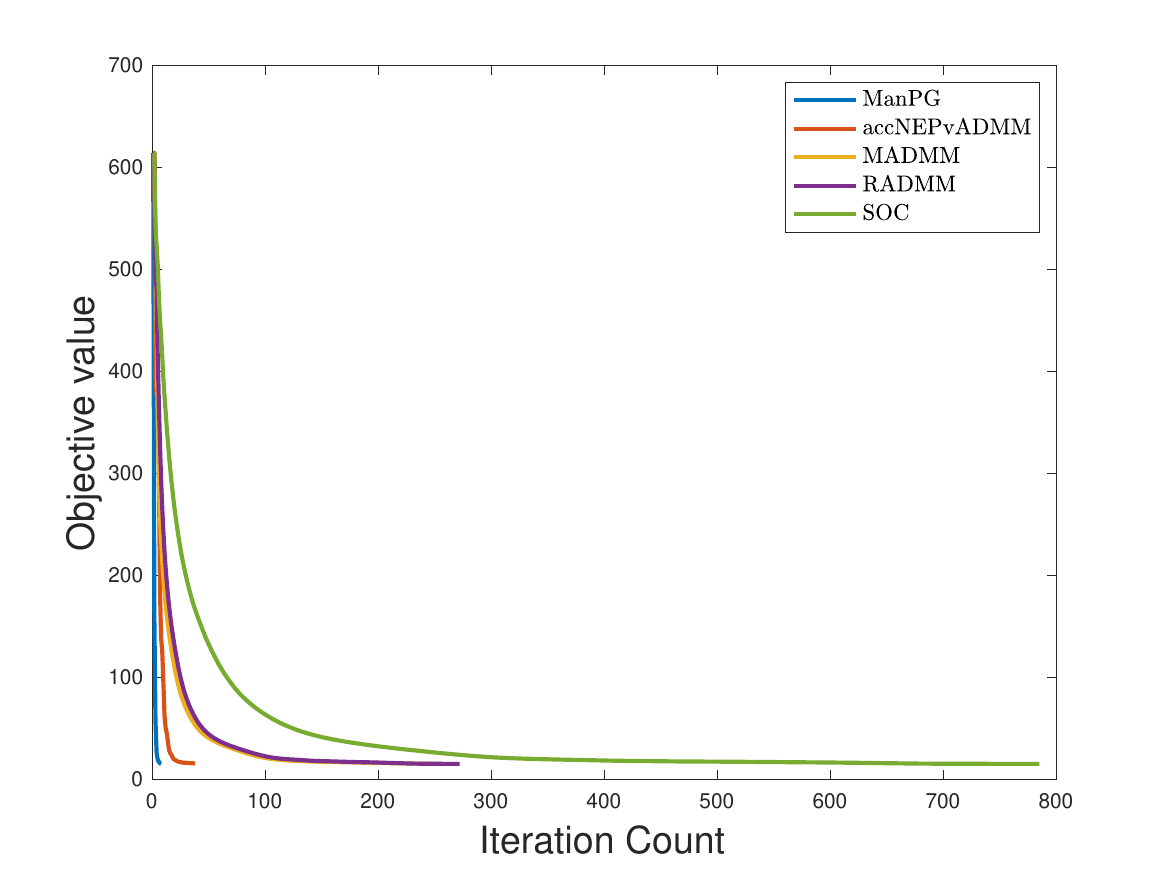}}\hfill
     \subfloat[\footnotesize{$n$ = 2000, $p$ = 20}]{\includegraphics[width=0.32\textwidth]{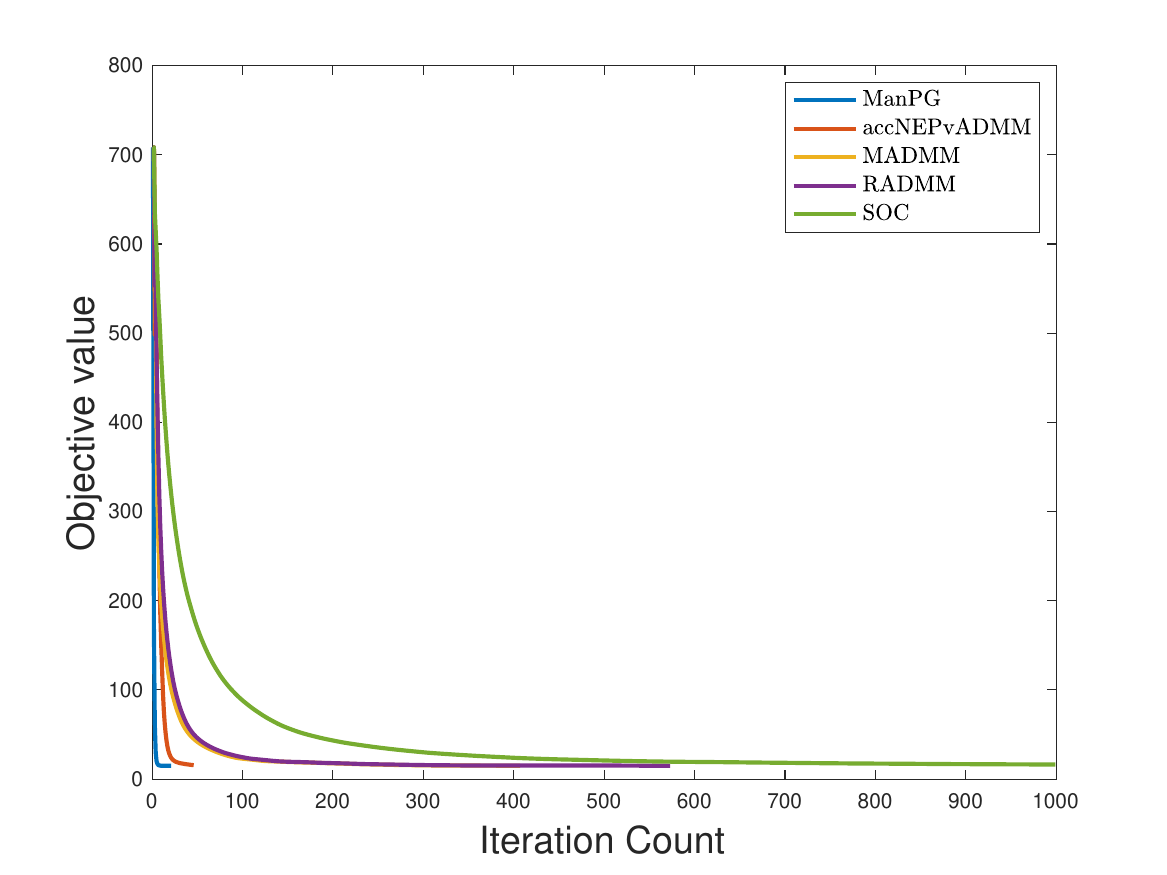}}
     
    \caption{Solving sPCA~\eqref{spca} by the LOCG accelerated NEPvADMM (Algorithm~\ref{locg.smooth.nepvadmm-simplified} with $n_{\rm LOCGscf} = 1$) and the five existing algorithms for the large $n$ case.}
\label{fig 2.5}
\end{figure}

As mentioned earlier, the cost of an $n \times n$ SEP in Algorithm~\ref{alg:smooth.nepvadmm} can be expensive for large $n$ and we have proposed the LOCG accelerated NEPvADMM for this case. In Figure~\ref{fig 2.5}, we test the performance of our Algorithm~\ref{locg.smooth.nepvadmm-simplified} with $n_{\rm LOCGscf} = 1$ when $n$ is greater than 1000 and compare it with the existing methods.
We find that MADMM is the slowest regardless of the size of $n$. As $n$ increases, SOC becomes slower than NEPvADMM (Algorithm~\ref{locg.smooth.nepvadmm-simplified}), ManPG, and RADMM, and it takes the most number of iterations.
The performances of our accelerated method, ManPG, and RADMM are generally close, but still we see that Algorithm~\ref{locg.smooth.nepvadmm-simplified} is slightly faster in subplots~\ref{2a} and \ref{2c}, and it requires fewer number of iterations.

\subsection{Numerical Results on DPCP}\label{ssec:DPCP}

In this subsection, we test the performance of our NEPvADMM on DPCP \eqref{dpcp}.
In the general setting \eqref{eq:TheProb},

\begin{equation}
    f(X) = 0, \;\;\; r(X) = |\!|\!| B^{\T} X |\!|\!|_1, \;\;\; \Psi = B^{\T}.
\end{equation}
To solve the $X$-subproblem \eqref{eq:MRADMM-frame:X} via the NEPv approach,
we need to construct $H_k(X) \in \mathbb{S}^{n \times n}$ and verify the validity of the NEPv Ansatz on page~\pageref{NEPvAnsatz}. We have
\[
g_k(X) = - r(Y_k) - \langle \Lambda_k, B^{\T} X - Y_k \rangle - \frac{\beta}{2}\Vert B^{\T} X - Y_k \Vert_{\F}^2,
\]
and the partial derivative of $g_k$ with respect to $X$ is
\[
\frac{\partial g_k(X)}{\partial X} = - B\Lambda_k + \beta BY_k - \beta BB^{\T}X.
\]
The KKT condition for the $X$-subproblem is
$\GRAD\;g_{k\vert \STM{p}{n}}(X) = 0$, or equivalently,
\begin{equation}\label{dpcp.smooth.admm.x.kkt}
    \begin{split}
        - B\Lambda_k + \beta BY_k - \beta BB^{\T}X = X\widehat{\Omega}_k, \;\; X \in \STM{p}{n}, \;\; \widehat{\Omega}_k \in \bbS^{p \times p}.
    \end{split}
\end{equation}
Here, \eqref{dpcp.smooth.admm.x.kkt} does not explicitly appear in the form of an NEPv due to the terms
in $D_k:=- B\Lambda_k + \beta BY_k=B(\beta Y_k - \Lambda_k)$, but using the idea of \eqref{cons.H.firstcase.H}, we construct
\begin{equation}\label{dpcp.smooth.admm.x.h}
H_k(X) := D_kX^{\T} + XD_k^{\T} - \beta BB^{\T} \in \bbS^{n \times n}
\end{equation}
for which we have
\[
H_k(X)X - \frac{\partial g_k(X)}{\partial X} = X( D_k^{\T}X ) \;\; {\rm for} \;\; X \in \STM{p}{n},
\]
and hence, in the form of \eqref{equa.nepv.parde}, $\scrM_k(X) = D_k^{\T}X$. We have the following result that verifies the NEPv Ansatz for the $X$-subproblem \eqref{eq:MRADMM-frame:X}.

\begin{lemma}\label{le10}
Consider the $X$-subproblem \eqref{eq:MRADMM-frame:X} for DPCP \eqref{dpcp}. Given $X \in \bbX_k := \STM{p}{n}_{D_k+}$, where $D_k := B ( \beta Y_k - \Lambda_k )$. 
If \eqref{ansatz.traceineq} with $H_k(X)$ by \eqref{dpcp.smooth.admm.x.h} is satisfied for $\what X \in \STM{p}{n}$, then
\[
 g_k( \wtd X ) \geq g_k(X) + \frac{\eta}{2},
\]
where $Q = UV^{\T}$, $\wtd X = \what XQ$ with $U$ and $V$ from the SVD $\what X^{\T}D_k = U \Sigma V^{\T}$.
\end{lemma}

\begin{proof}
Suppose \eqref{ansatz.traceineq} holds for $X \in \bbX_k$ and $\what X \in \STM{p}{n}$, i.e.,
\begin{equation}\label{dpcp.smooth.ansatz.cond}
    \tr(\what X^{\T}H_k(X)\what X) \geq \tr(X^{\T}H_k(X)X) + \eta.
\end{equation}
The right-hand side of \eqref{dpcp.smooth.ansatz.cond} becomes
\[
\tr(X^{\T}H_k(X)X) + \eta = -2\tr(\Lambda_k^{\T}B^{\T}X) - \beta \tr(X^{\T}BB^{\T}X) + 2\beta \tr(Y_k^{\T}B^{\T}X) + \eta.
\]
Add $2\tr( \Lambda_k^{\T}Y_k) - \beta \tr(Y_kY_k^{\T}) - 2 r(Y_k)$ to both sides of this equality to get
\begin{equation}\label{dpcp.ansa.right.gk}
    \tr(X^{\T}H_k(X)X) + \eta + 2\tr( \Lambda_k^{\T}Y_k) - \beta \tr(Y_kY_k^{\T}) - 2 r(Y_k) = 2g_k(X) + \eta.
\end{equation}
The left-hand side of \eqref{dpcp.smooth.ansatz.cond} is
\begin{align*}
\tr(\what X^{\T}H_k(X)\what X)
   & = - \beta \tr(\what X^{\T}BB^{\T}\what X) + 2\tr(\what X^{\T}B(\beta Y_k - \Lambda_k)X^{\T}\what X) \\
   & \leq \beta \tr(\wtd X^{\T}BB^{\T}\wtd X) + 2\tr(\wtd X^{\T}B(\beta Y_k - \Lambda_k)). 
\end{align*}
Similarly, we have
\begin{equation}\label{dpcp.ansa.left.gk}
    \tr(\what X^{\T}H_k(X)\what X) + 2\tr( \Lambda_k^{\T}Y_k) - \beta \tr(Y_kY_k^{\T}) - 2 r(Y_k) \leq 2g_k(\wtd X).
\end{equation}
Combining \eqref{dpcp.smooth.ansatz.cond}, \eqref{dpcp.ansa.right.gk} and \eqref{dpcp.ansa.left.gk}, we have
$
g_k(\wtd X) \geq g_k(X) + \eta / 2,
$
i.e., the NEPv Ansatz on page~\pageref{NEPvAnsatz} holds with $\omega = 1/2$.
\end{proof}

The experiment settings for this problem are built in the same way as described in \cite{chen2021manifold, li2024riemannian}.
We construct the data matrix as $B = [SR,\wtd O]$, where both $\wtd O \in \bbR^{n\times p_2}$ and $R \in \bbR^{d\times p_1}$ are built using random Gaussian entries, and $S \in \bbR^{n \times d}$ with orthonormal column vectors.
Now we compare the performance of our NEPvADMM approach with five state-of-the-art algorithms for solving this problem: RADMM, MADMM, SOC, the manifold proximal point algorithm (ManPPA) \cite{chen2021manifold}, and the projected subgradient method (PSGM) \cite{zhu2018dual}. We adopt the same procedures and the parameter settings described in \cite{li2024riemannian} and the codes of these algorithms are available on Github website.\footnote{https://github.com/JasonJiaxiangLi/RADMM} For the last two algorithms, interested readers may refer to the papers listed above for more details.

We set the number of inliers $p_1 = 500$, number of outliers $p_2 = 1200$ and $p = 5$.
The ambient dimension $n$ here is generally small (e.g. $n = 30$) compared to the previous example sPCA. To prevent some existing algorithms from producing abnormal results due to large $n$ values and to ensure a fair comparison, the existing methods are tested under $n = \{30, 50, 70\}$.
In this case, the computational advantage of LOCG acceleration is negligible for small values of $n$ and $p$. Consequently, we employ Algorithm~\ref{alg:smooth.nepvadmm} ($n_{\scf} = 1$) for comparison.
For our Algorithm~\ref{alg:smooth.nepvadmm}, we choose $\beta = 30$ and $\gamma = 10^{-9}$.
Similarly to the last experiment, we introduce the error from the violation of constraint. In the $k$-th iteration, we use
$\left\| Z_k - B^{\T}X_k \right\|_{\F}/\left\| B^{\T}X_k \right\|_{\F}$
as the normalized error for Algorithm~\ref{alg:smooth.nepvadmm} and RADMM, and use
$\left\| Y_k - B^{\T}X_k \right\|_{\F} / \left\| B^{\T}X_k \right\|_{\F}$
for MADMM, and adopt
$\left\| X_k - Y_k \right\|_{\F} / \left\| Y_k \right\|_{\F}$
for SOC. We terminate our NEPvADMM when
$\left|F(X_{k+1}) - F(X_k)\right| / \left|F(X_k)\right| < 10^{-6}$.

\begin{figure}[htbp]
     \centering
     \subfloat[$n$ = 30, $p$ = 5\label{3a}]{\includegraphics[width=0.32\textwidth]{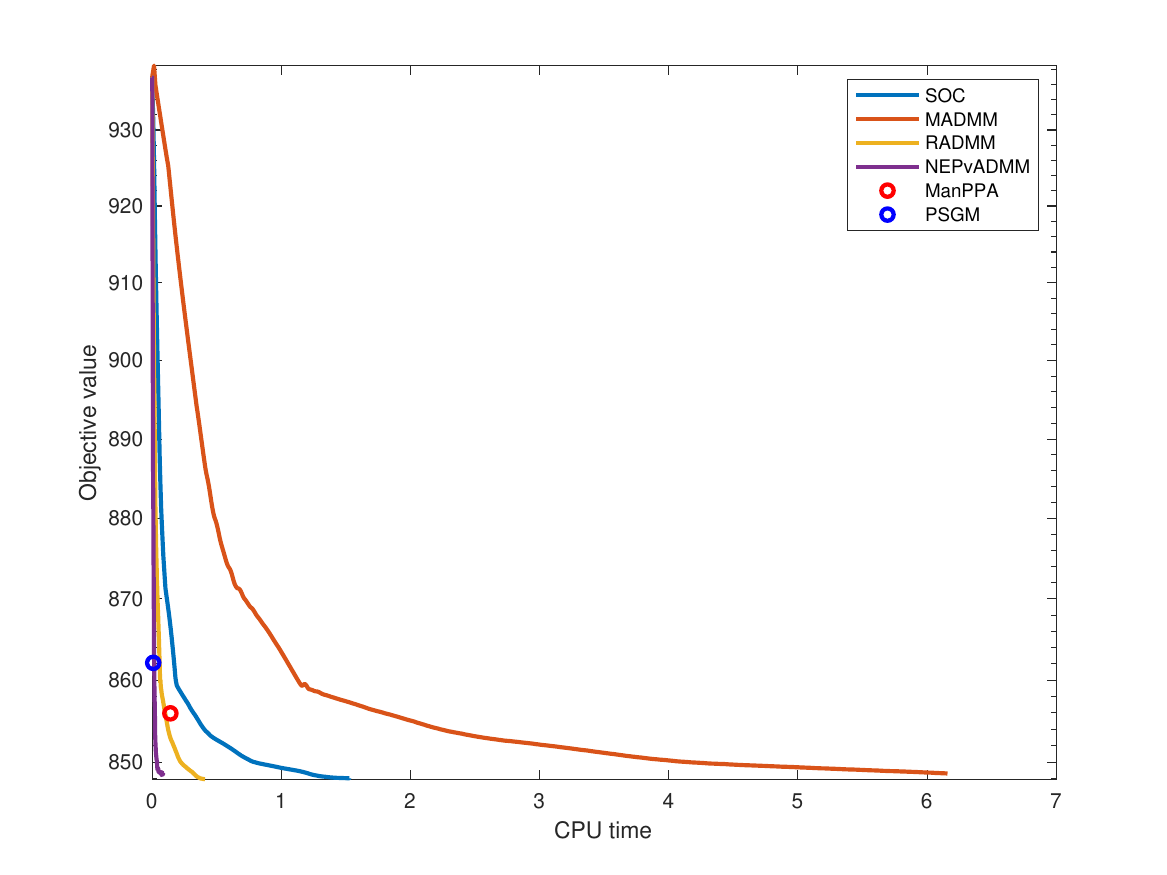}}\hfill
     \subfloat[$n$ = 50, $p$ = 5\label{3c}]{\includegraphics[width=0.32\textwidth]{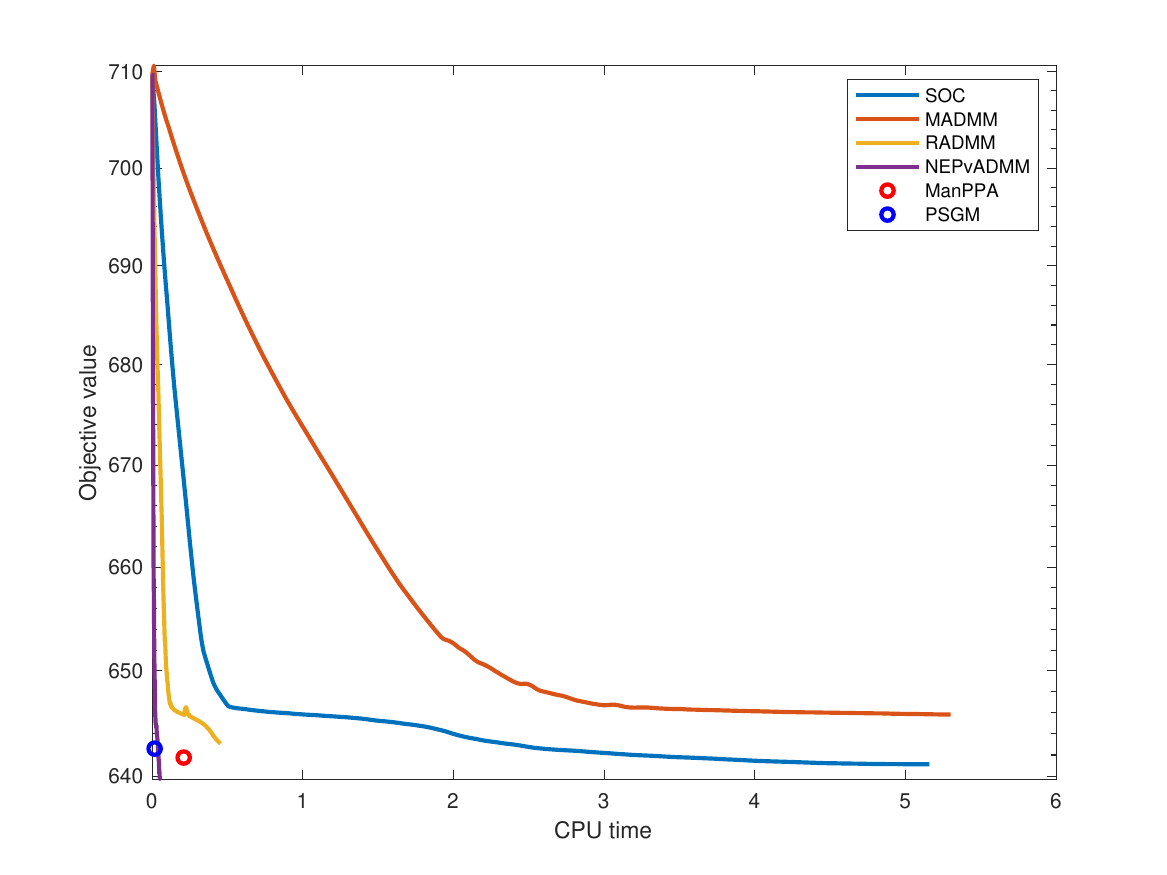}}\hfill
     \subfloat[$n$ = 70, $p$ = 5\label{3e}]{\includegraphics[width=0.32\textwidth]{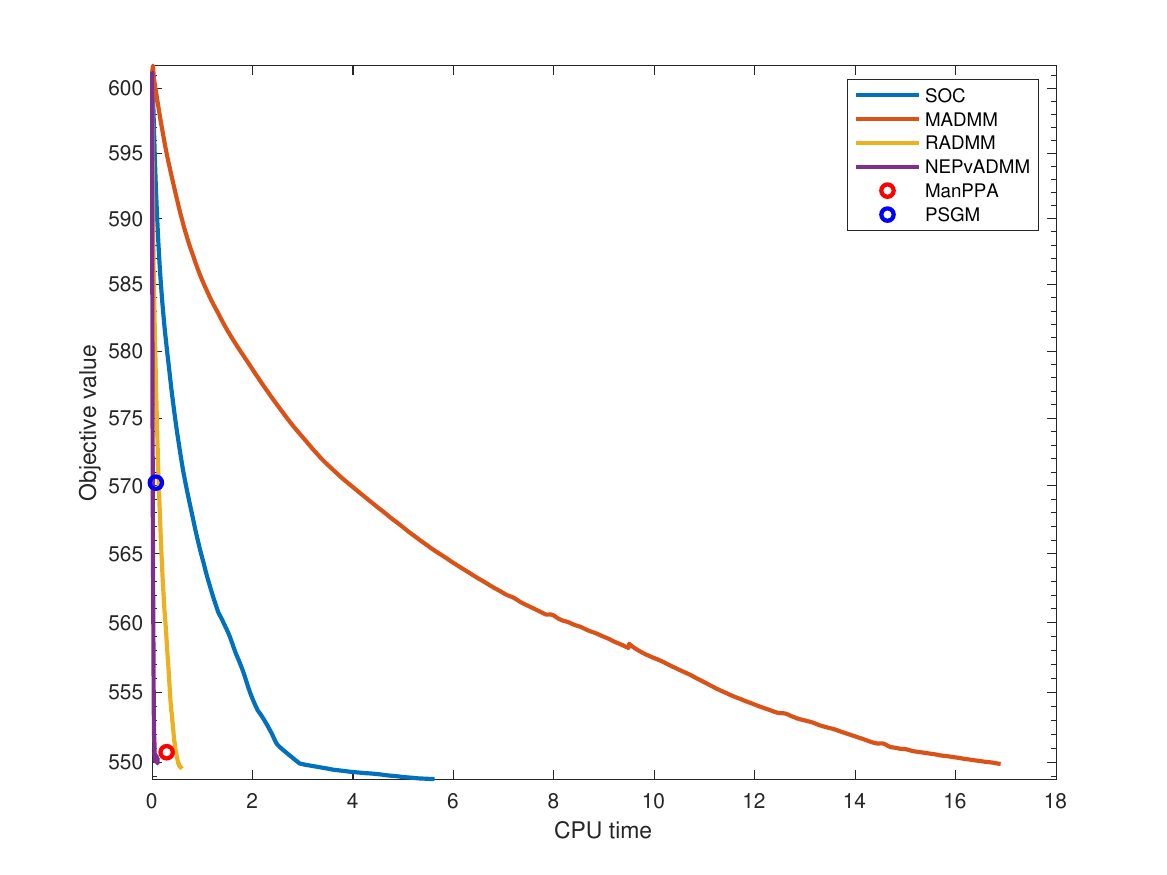}}
     
     \subfloat[$n$ = 30, $p$ = 5\label{3b}]{\includegraphics[width=0.32\textwidth]{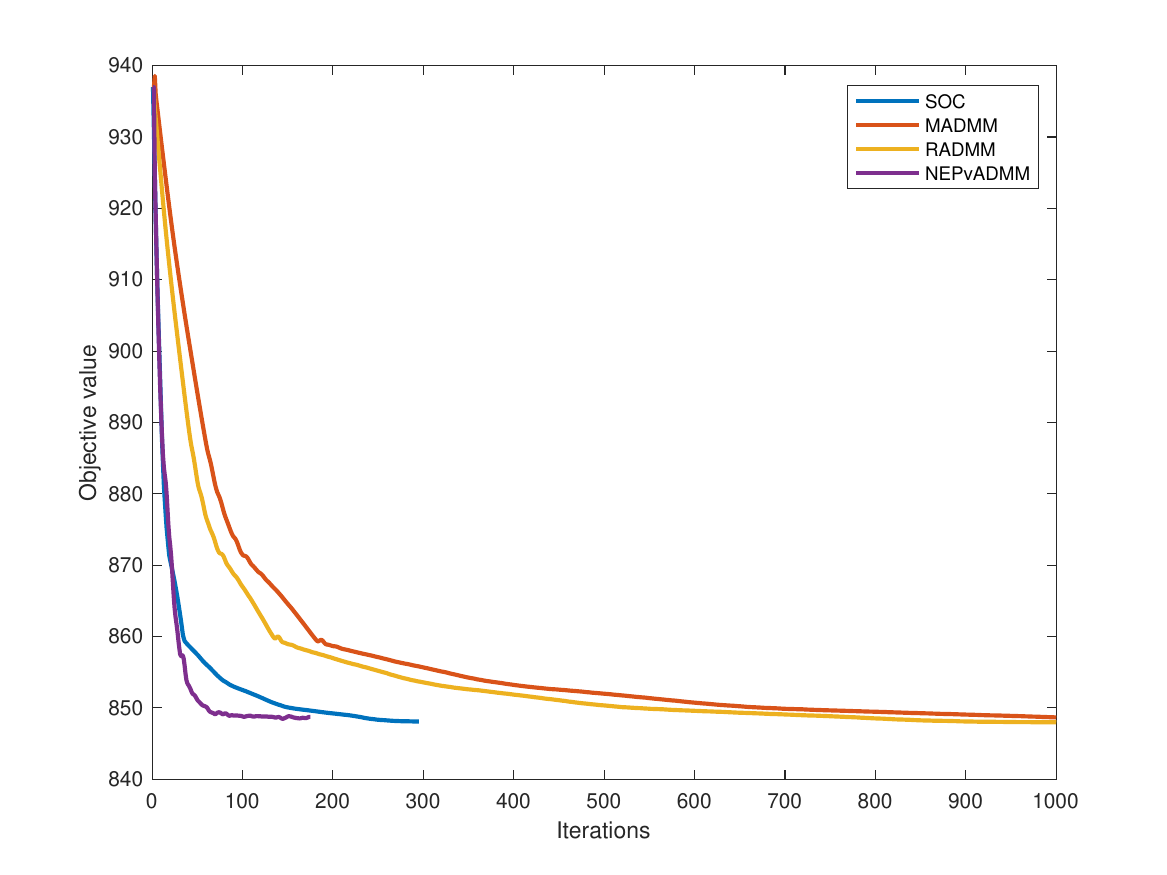}}\hfill
     \subfloat[$n$ = 50, $p$ = 5\label{3d}]{\includegraphics[width=0.32\textwidth]{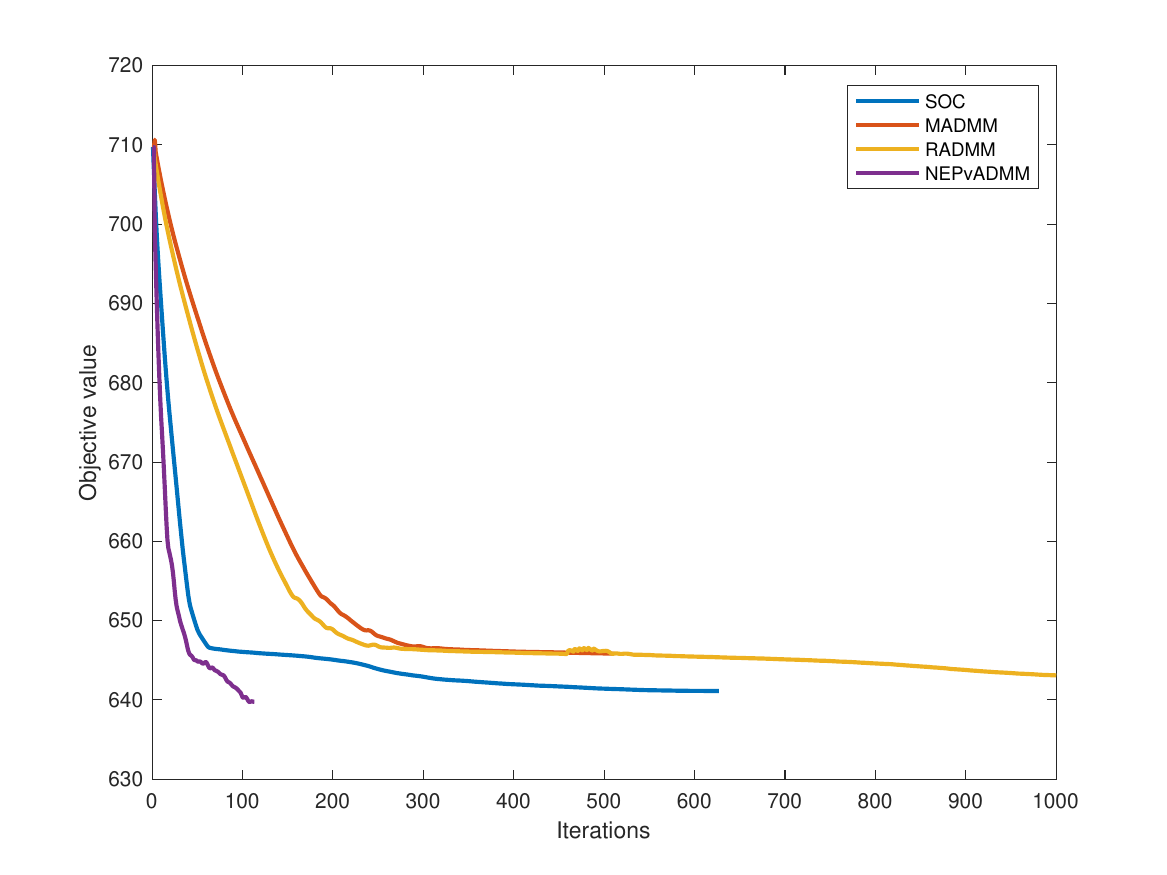}}\hfill
     \subfloat[$n$ = 70, $p$ = 5\label{3f}]{\includegraphics[width=0.32\textwidth]{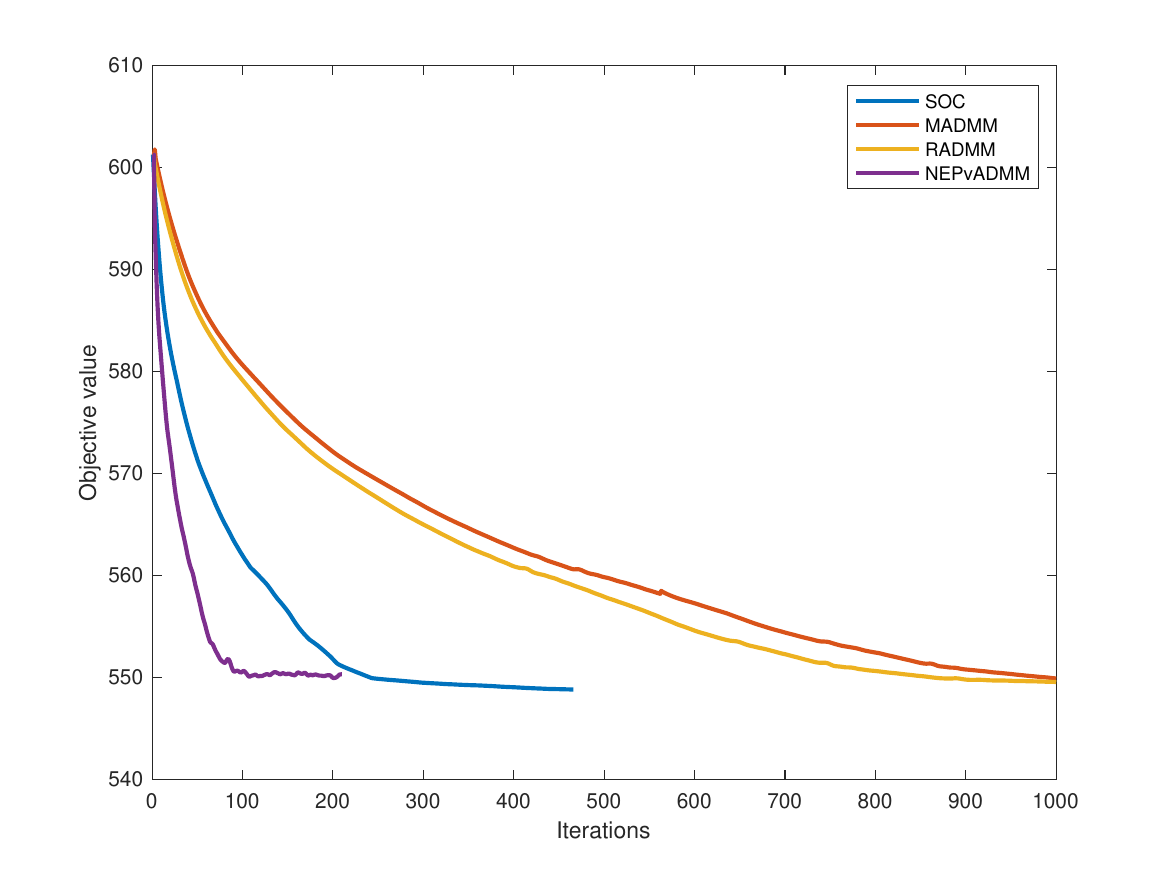}}
     
    \caption{Solving DPCP~\eqref{dpcp} by six methods: objective function value vs. CPU time and iteration counts.}
    \label{fig 3}
\end{figure}

Figure~\ref{fig 3} compares how the objective function value vs. the number of iterations and the CPU time.
Because PSGM and ManPPA conduct the minimization in a column-wise manner, we only record their objective function values at the final output in our experiments.
We see that our method can generate slightly lower objective function values while requiring significantly less CPU time. As observed from subplots~\ref{3b}, \ref{3d}, and \ref{3f}, it also takes the least number of iterations.
Here, MADMM is again the least efficient approach from the aspect of iteration counts and CPU time. We record the data and provide a detailed comparison in Table \ref{ta1}, and it can be seen that our approach have a slightly larger constraint violations compared to others.

\begin{table}
\centering
\caption{Solving DPCP~\eqref{dpcp} by six algorithms: objective function value, error and CPU time.}
  \scalebox{0.85}{
  \begin{tabular}{c|cccccc|cccccc}
    \toprule
     \hline
      \multicolumn{1}{c|}{} &
      \multicolumn{6}{c|}{NEPvADMM} &
      \multicolumn{6}{c}{RADMM} \\
      \hline
       {$(n,p)$} & \multicolumn{2}{c|}{Err} & \multicolumn{2}{c|}{Obj} & \multicolumn{2}{c|}{CPU} & \multicolumn{2}{c|}{Err} & \multicolumn{2}{c|}{Obj} & \multicolumn{2}{c}{CPU}\\
       \hline
       {(30,5)} & \multicolumn{2}{c|}{$4.90 \cdot 10^{-3}$} & \multicolumn{2}{c|}{849.08} & \multicolumn{2}{c|}{$2.76 \cdot 10^{-2}$} & \multicolumn{2}{c|}{$1.00 \cdot 10^{-3}$} & \multicolumn{2}{c|}{855.15} & \multicolumn{2}{c}{$3.71 \cdot 10^{-2}$} \\
       {(50,5)} & \multicolumn{2}{c|}{$9.60 \cdot 10^{-3}$} & \multicolumn{2}{c|}{654.30} & \multicolumn{2}{c|}{$1.92 \cdot 10^{-2}$} & \multicolumn{2}{c|}{$1.30 \cdot 10^{-3}$} & \multicolumn{2}{c|}{660.70} & \multicolumn{2}{c}{$1.17 \cdot 10^{-1}$} \\
       {(70,5)} & \multicolumn{2}{c|}{$6.70 \cdot 10^{-3}$} & \multicolumn{2}{c|}{555.74} & \multicolumn{2}{c|}{$5.14 \cdot 10^{-2}$} & \multicolumn{2}{c|}{$2.40 \cdot 10^{-3}$} & \multicolumn{2}{c|}{565.18} & \multicolumn{2}{c}{$2.55 \cdot 10^{-1}$} \\
     \hline
     \toprule
     \hline
      \multicolumn{1}{c|}{} &
      \multicolumn{6}{c|}{SOC} &
      \multicolumn{6}{c}{MADMM} \\
      \hline
       {$(n,p)$} & \multicolumn{2}{c|}{Err} & \multicolumn{2}{c|}{Obj} & \multicolumn{2}{c|}{CPU} & \multicolumn{2}{c|}{Err} & \multicolumn{2}{c|}{Obj} & \multicolumn{2}{c}{CPU}\\
       \hline
       {(30,5)} & \multicolumn{2}{c|}{$2.00 \cdot 10^{-4}$} & \multicolumn{2}{c|}{849.84} & \multicolumn{2}{c|}{$7.66 \cdot 10^{-1}$} & \multicolumn{2}{c|}{$4.00 \cdot 10^{-3}$} & \multicolumn{2}{c|}{853.63} & \multicolumn{2}{c}{$7.69 \cdot 10^{-1}$} \\
       {(50,5)} & \multicolumn{2}{c|}{$1.00 \cdot 10^{-4}$} & \multicolumn{2}{c|}{655.55} & \multicolumn{2}{c|}{$8.99 \cdot 10^{-1}$} & \multicolumn{2}{c|}{$1.50 \cdot 10^{-3}$} & \multicolumn{2}{c|}{656.23} & \multicolumn{2}{c}{$3.96 \cdot 10^{0}$} \\
       {(70,5)} & \multicolumn{2}{c|}{$1.00 \cdot 10^{-4}$} & \multicolumn{2}{c|}{564.26} & \multicolumn{2}{c|}{$1.58 \cdot 10^{0}$} & \multicolumn{2}{c|}{$2.60 \cdot 10^{-3}$} & \multicolumn{2}{c|}{565.22} & \multicolumn{2}{c}{$7.90 \cdot 10^{0}$} \\
     \hline
     \toprule
     \hline
      \multicolumn{1}{c|}{} &
      \multicolumn{6}{c|}{ManPPA} &
      \multicolumn{6}{c}{PSGM} \\
      \hline
       {$(n,p)$} & \multicolumn{3}{c|}{Obj} & \multicolumn{3}{c|}{CPU} & \multicolumn{3}{c|}{Obj} & \multicolumn{3}{c}{CPU}\\
       \hline
       {(30,5)} & \multicolumn{3}{c|}{851.18} & \multicolumn{3}{c|}{$9.28 \cdot 10^{-2}$} & \multicolumn{3}{c|}{855.56} & \multicolumn{3}{c}{$1.28 \cdot 10^{-2}$} \\
       {(50,5)} & \multicolumn{3}{c|}{653.34} & \multicolumn{3}{c|}{$1.75 \cdot 10^{-1}$} & \multicolumn{3}{c|}{659.22} & \multicolumn{3}{c}{$4.56 \cdot 10^{-2}$} \\
       {(70,5)} & \multicolumn{3}{c|}{553.14} & \multicolumn{3}{c|}{$7.30 \cdot 10^{-1}$} & \multicolumn{3}{c|}{582.65} & \multicolumn{3}{c}{$8.04 \cdot 10^{-2}$} \\
     \hline
    \bottomrule
  \end{tabular}}
\label{ta1}
\end{table}

We now compare the performance of these existing algorithms with our Algorithm~\ref{alg:smooth.nepvadmm} ($n_{\scf} = 1$) under varying $p_1$ and $p_2$. For each case, we randomly generate $30$ testing problems with all algorithms starting from the same randomly generated initial points, and record the average value of the corresponding objective function values at the last iterate, CPU time, and the number of iterations.
With parameters $n = 90$, $p = 5$, $p_2 = 1200$, and varying $p_1$ among $\{300, 350,\dots, 700\}$, we generate sub-Figures~\ref{dpcp.1a}, \ref{dpcp.1c}, and \ref{dpcp.1e}. With $n = 90$, $p = 5$, $p_1 = 500$, and varying $p_2$ among $\{800, 850,\dots,1500\}$, we plot the results in sub-Figures~\ref{dpcp.1b}, \ref{dpcp.1d}, and \ref{dpcp.1f}. We have the following observations:
\begin{enumerate}[(1)]

  \item As illustrated in sub-Figures~\ref{dpcp.1a} and \ref{dpcp.1b}, the PSGM method yields slightly higher objective function values, whereas the objective function values produced by all other algorithms remain comparable.

  \item Our approach takes fewer iterations to converge compared to MADMM, SOC, and RADMM. The CPU time by our Algorithm~\ref{alg:smooth.nepvadmm} is among the lowest.

\end{enumerate}

\begin{figure}[htbp]
     \centering
     \subfloat[Objective value\label{dpcp.1a}]{\includegraphics[width=0.32\textwidth]{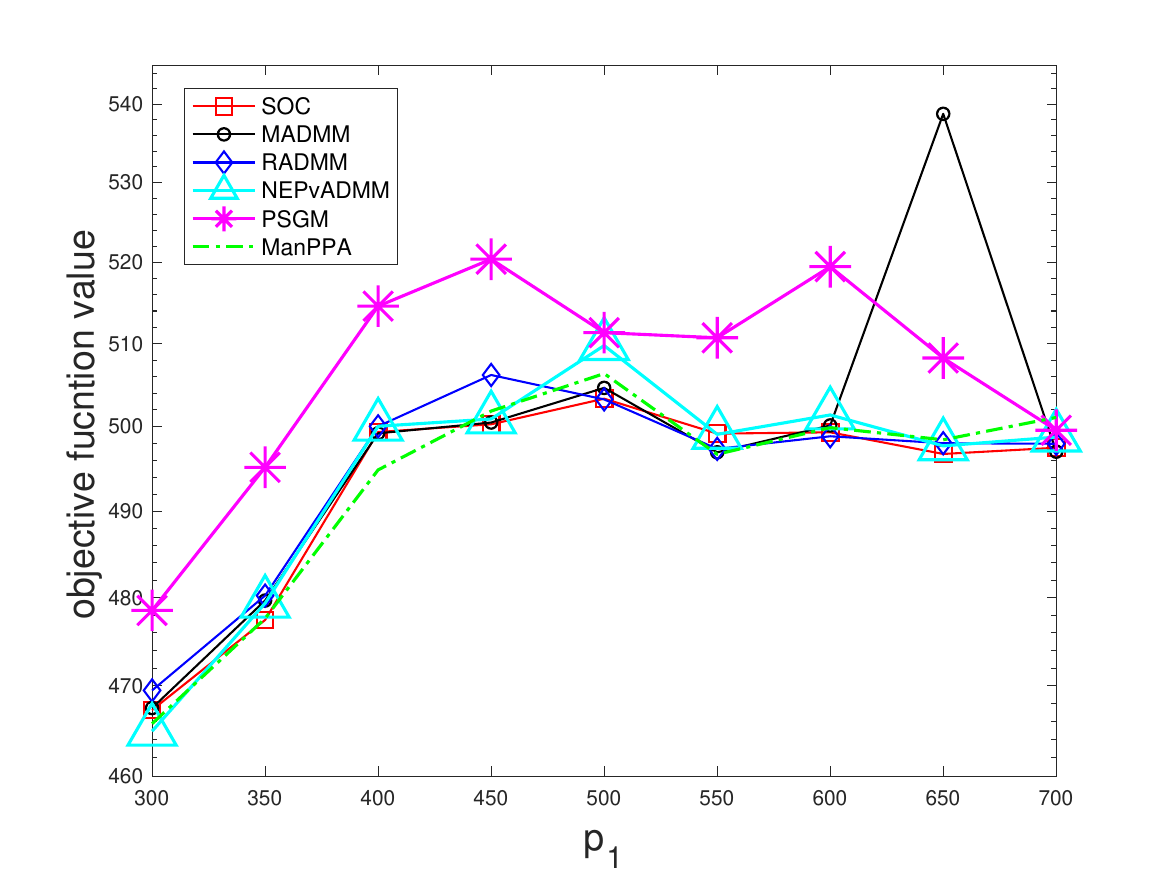}}\hfill
     \subfloat[Iteration count\label{dpcp.1c}]{\includegraphics[width=0.32\textwidth]{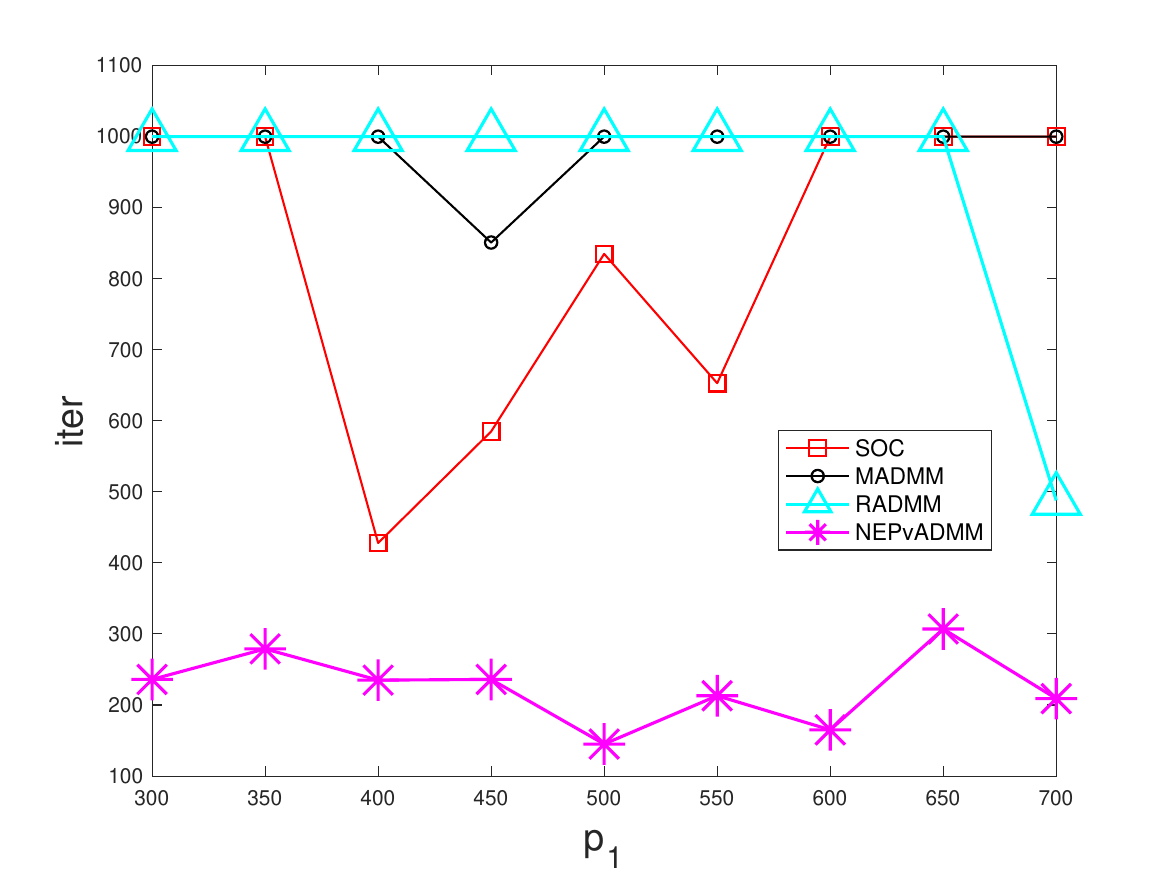}}\hfill
     \subfloat[CPU time\label{dpcp.1e}]{\includegraphics[width=0.32\textwidth]{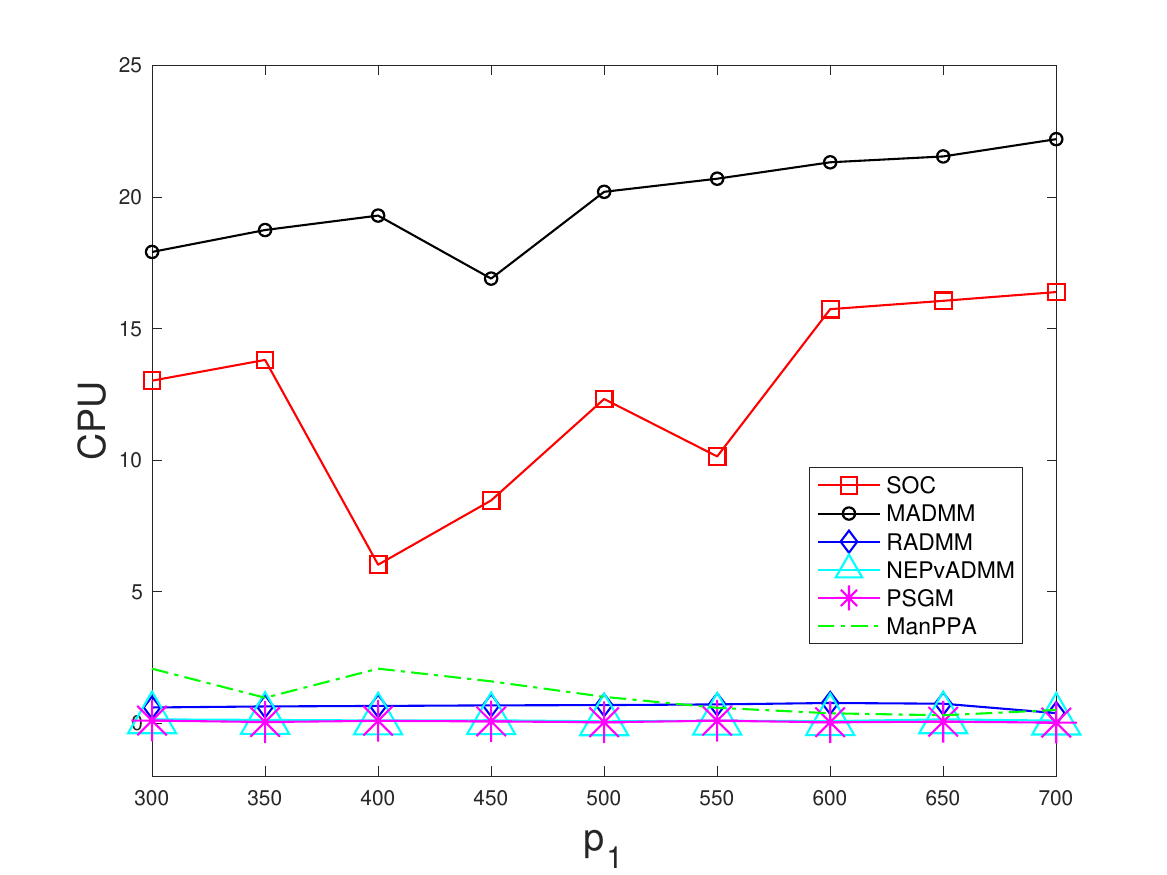}}
     
     \subfloat[Objective value\label{dpcp.1b}]{\includegraphics[width=0.32\textwidth]{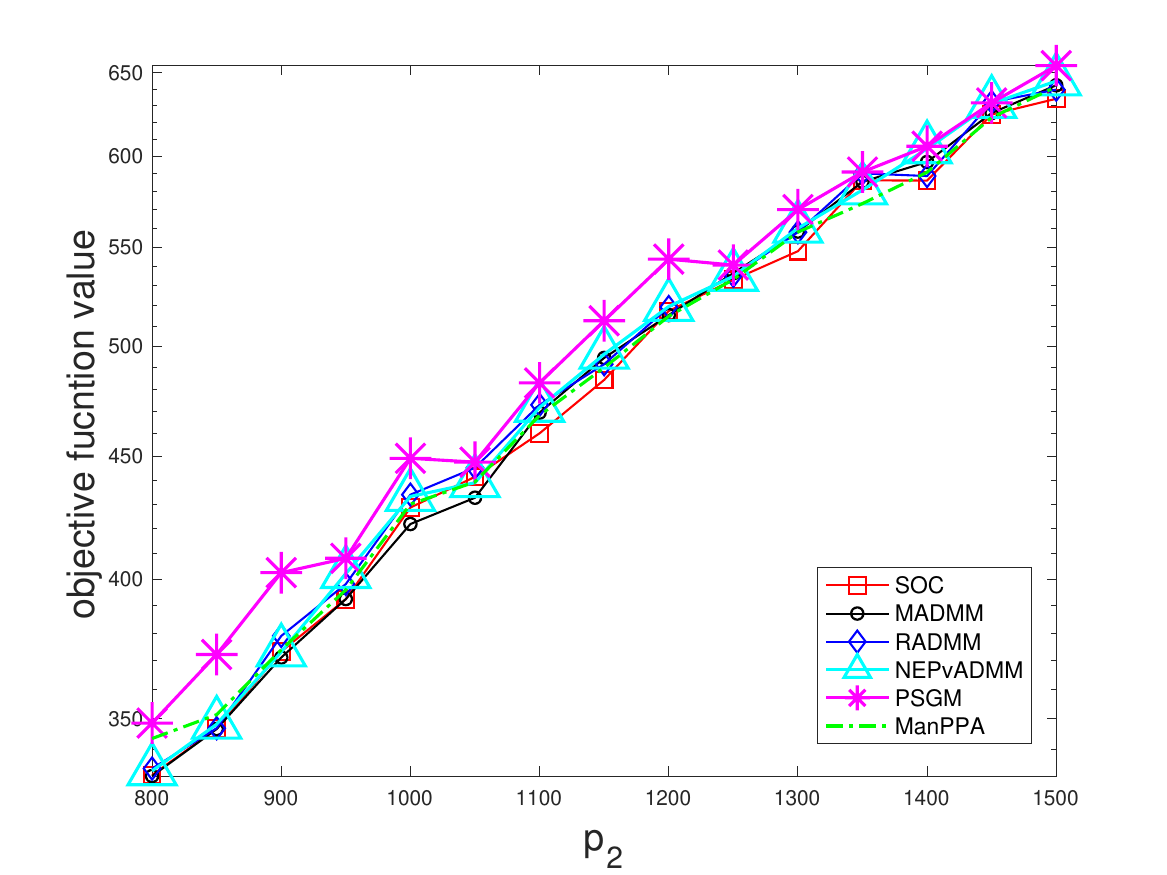}}\hfill
     \subfloat[Iteration count\label{dpcp.1d}]{\includegraphics[width=0.32\textwidth]{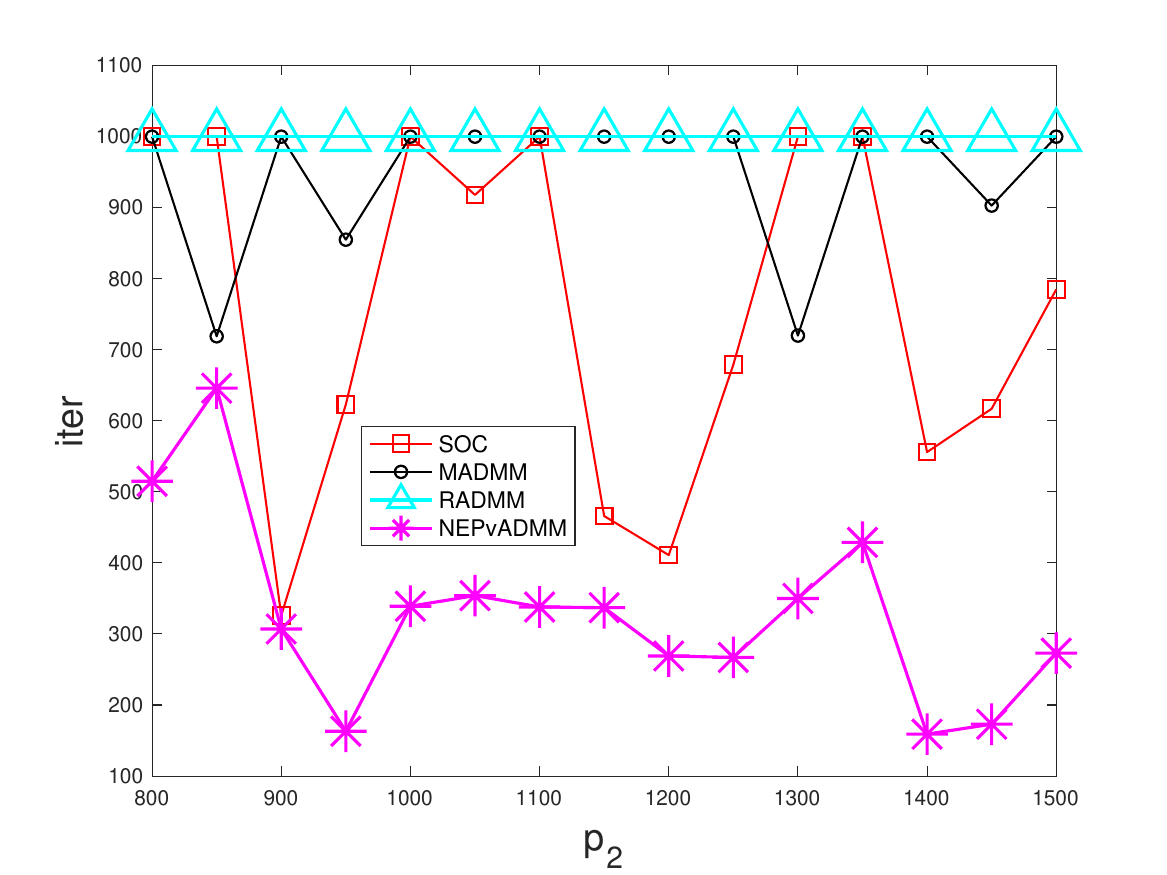}}\hfill
     \subfloat[CPU time\label{dpcp.1f}]{\includegraphics[width=0.32\textwidth]{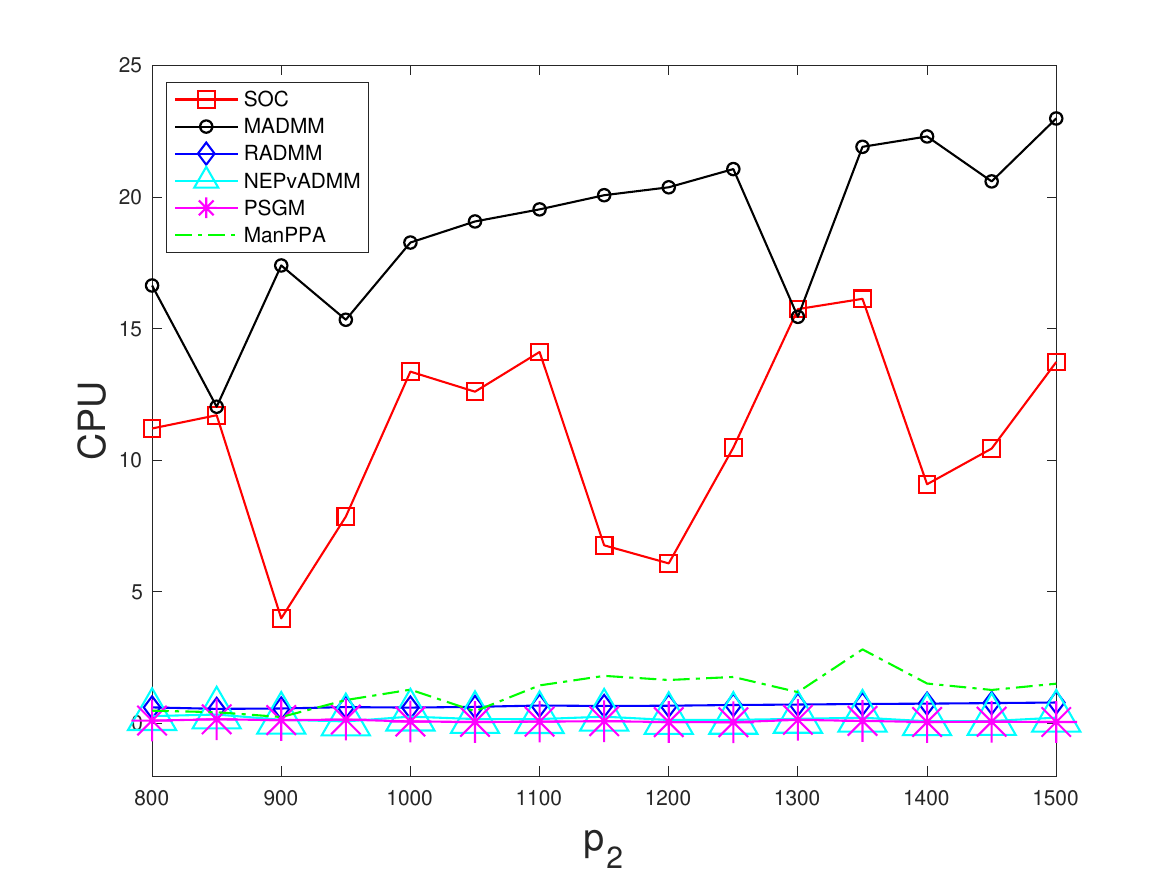}}
     
    \caption{Solving DPCP~\eqref{dpcp} by NEPvADMM (Algorithm~\ref{alg:smooth.nepvadmm} with $n_{\scf} = 1$) and the five existing algorithms under different parameter settings. 1st row: $p_2 = 1200$, $n = 90$, and $p = 5$ with varying $p_1$; 2nd row: $p_1 = 500$, $n = 90$, and $p = 5$ with varying $p_2$.}
    \label{fig.dpcp.gen.comp}
\end{figure}

\section{Conclusion}\label{Se7}
In this paper, we introduce an innovative method to address the non-smooth composite optimization on the Stiefel manifold by employing a nonlinear eigen-approach to go with ADMM.
Our approach fully utilizes the unique properties of the problem, such as dimension reduction and sparsity, and has great advantages compared to the existing algorithms.
The steps of our NEPvADMM are easy to implement, and can be applied to many applications in machine learning and data science.
In addition to these, our method can be customized to specific problems in order to further improve efficiency. 

\vspace{1em}

\noindent\textbf{Data Availability Statement} Data sharing is not applicable to this article as all numerical experiments rely on synthetically generated data. The procedures required to reproduce the synthetic data in MATLAB are described within the manuscript.

\vspace{1em}

\noindent\textbf{Code Availability} Our MATLAB source code is publicly available on GitHub at \url{https://github.com/WANGJiawei000/NEPvADMM-Stiefel}.

\section*{Declaration}

\textbf{Conflict of interest} The authors declare that they have no conflict of interest.

\bibliographystyle{spmpsci}
\bibliography{reference,refLi}

\end{document}